\documentclass[reqno,tbtags,a4paper,12pt]{amsart}

\usepackage[in]{fullpage}
\usepackage[matrix,arrow,curve]{xy}

\usepackage[matrix,arrow,curve]{xy}
\usepackage{tikz}
\usetikzlibrary{cd}
\usetikzlibrary{snakes}
\usetikzlibrary{decorations.markings}

\tikzset{->-/.style={decoration={
     markings,
     mark=at position #1 with {\arrow{>}}},postaction={decorate}}}

\usepackage{amsmath,amssymb,amsthm,amsfonts,amscd}
\usepackage{graphics,graphicx}

\newcommand{\Z}{\mathbb{Z}}
\newcommand{\Q}{\mathbb{Q}}
\newcommand{\R}{\mathbb{R}}

\newcommand{\BB}{\mathbb{B}}
\newcommand{\BC}{\mathbf{C}}
\newcommand{\F}{\mathbf{F}}

\newcommand{\CA}{\mathcal{A}}
\newcommand{\B}{\mathcal{B}}
\newcommand{\CC}{\mathcal{C}}
\newcommand{\CF}{\mathcal{F}}
\newcommand{\CH}{\mathcal{H}}
\newcommand{\I}{\mathcal{I}}
\newcommand{\CP}{\mathcal{P}}
\newcommand{\orb}{\mathcal{O}}
\newcommand{\CW}{\mathcal{W}}
\newcommand{\M}{\mathcal{M}}
\newcommand{\CL}{\mathcal{L}}
\newcommand{\K}{\mathcal{K}}
\newcommand{\res}{\mathcal{R}}

\newcommand{\fA}{\mathbf{A}}
\newcommand{\CX}{\mathfrak{X}}

\newcommand{\hI}{\widehat{\I}}
\newcommand{\hsigma}{\widehat{\sigma}}
\newcommand{\hrho}{\widehat{\rho}}

\newcommand{\bgamma}{\boldsymbol{\gamma}}
\newcommand{\btheta}{\vartheta}

\newcommand{\ba}{\mathbf{a}}
\newcommand{\bb}{\mathbf{b}}
\newcommand{\bc}{\mathbf{c}}
\newcommand{\bx}{\mathbf{x}}
\newcommand{\by}{\mathbf{y}}

\newcommand{\Mod}{\mathop{\mathrm{Mod}}\nolimits}
\newcommand{\PMod}{\mathop{\mathrm{PMod}}\nolimits}
\newcommand{\Tot}{\mathop{\mathrm{Tot}}\nolimits}
\newcommand{\Arf}{\mathop{\rm Arf}\nolimits}
\newcommand{\Sp}{\mathrm{Sp}}
\newcommand{\SL}{\mathrm{SL}}
\newcommand{\Hom}{\mathop{\mathrm{Hom}}\nolimits}
\newcommand{\Stab}{\mathop{\mathrm{Stab}}\nolimits}
\newcommand{\rank}{\mathop{\mathrm{rank}}}
\newcommand{\pr}{\mathrm{pr}}
\newcommand{\Int}{\mathop{\mathrm{Int}}}

\newtheorem{theorem}{Theorem}[section] 
\newtheorem{propos}[theorem] {Proposition}
\newtheorem{cor}[theorem] {Corollary}
\newtheorem{lem}[theorem]{Lemma}
\newtheorem{fact}[theorem] {Fact}
\theoremstyle{definition}
\newtheorem{remark}[theorem]{Remark}

\numberwithin{equation}{section}

\author{Alexander A. Gaifullin}

\address{Steklov Mathematical Institute of Russian Academy of Sciences, Moscow, Russia}
\address{Skolkovo Institute of Science and Technology, Skolkovo, Russia}
\address{Institute for the Information Transmission Problems of the Russian Academy of Sciences (Kharkevich Institute), Moscow, Russia}
\email{agaif@mi-ras.ru}

\title[On spectral sequence for genus~$3$ Torelli group]{On spectral sequence for the action of genus~$3$ Torelli group on the complex of cycles}
\date{}

\begin{document}
\begin{abstract}
The Torelli group of a genus~$g$ closed oriented surface~$S_g$ is the subgroup~$\I_g$ of the mapping class group~$\Mod(S_g)$ consisting of all mapping classes that act trivially on the homology of~$S_g$. One of the most intriguing open problems concerning Torelli groups is the question of whether the group~$\I_3$ is finitely presented or not. A possible approach to this problem relies upon the study of the second homology group of~$\I_3$ using the spectral sequence~$E^r_{p,q}$ for the action of~$\I_3$ on the complex of cycles. In this paper we obtain a partial result towards the conjecture that~$H_2(\I_3;\Z)$ is not finitely generated and hence~$\I_3$ is not finitely presented. Namely, we prove that the term~$E^3_{0,2}$ of the spectral sequence is infinitely generated, that is, the group~$E^1_{0,2}$ remains infinitely generated after taking quotients by images of the differentials~$d^1$ and~$d^2$. If one proceeded with the proof that it also remains infinitely generated after taking quotient by the image of~$d^3$, he would complete the proof of the fact that $\I_3$ is not finitely presented. 
\end{abstract}

\maketitle

\begin{flushright}
\textit{In memory of Sergei Ivanovich Adian,\\ a bright and deep mathematician,\\ to whom I am very grateful for bringing me\\ into the marvelous world of Torelli groups}
\end{flushright}

\section{Introduction}

Let $S_g$ be a closed oriented surface of genus~$g$, and let $\Mod(S_g)$ be the mapping class group of it. By definition, the \textit{Torelli group} $\I_g=\I(S_g)$ is the subgroup of~$\Mod(S_g)$ consisting of all mapping classes that act trivially on~$H_1(S_g;\Z)$. In other words, $\I_g$ is the kernel of the natural surjective homomorphism 
\begin{equation*}
\Mod(S_g)\to \Sp(2g,\Z).
\end{equation*}

It is well known that $\Mod(S_1)= \SL(2,\Z)$, so the group~$\I_1$ is trivial. McCullough and Miller~\cite{MCM86} proved that the group~$\I_2$ is not finitely generated, and then Mess~\cite{Mes92} proved that in fact it is an infinitely generated free group. On the other hand, Johnson~\cite{Joh83} showed that $\I_g$ is finitely generated, provided that $g\ge 3$.

One of the most intriguing open problems concerning Torelli groups is the question of whether the groups~$\I_g$, $g\ge 3$, are finitely presented. This problem is contained in Kirby's list of problems in low-dimensional topology~\cite[Problem~2.9(A)]{Kir97} and is attributed there to Mess. A usual expectation is that $\I_g$ are finitely presented for all~$g\ge 4$ but $\I_3$ is not. This problem is closely related to the question of whether the second homology group~$H_2(\I_g;\Z)$ is finitely generated.  Indeed, if $H_2(\I_g;\Z)$ were proved to be not finitely generated, this would immediately imply that $\I_g$ is not finitely presented. In this paper, we suggest an approach towards proving that~$H_2(\I_3;\Z)$ is not finitely generated and hence~$\I_3$ is not finitely presented, and obtain some partial result on this way.

Bestvina, Bux, and Margalit~\cite{BBM07} introduced a contractible $(2g-3)$-dimensional CW complex~$\B_g$ called the \textit{complex of cycles} on which the Torelli group~$\I_g$ acts cellularly and without rotations. This action yields the spectral sequence
\begin{equation}\label{eq_SpSeqIntro}
E^1_{p,q}\cong \bigoplus_{\sigma\in\mathfrak{X}_p}H_q\bigl(\Stab_{\I_g}(\sigma);\Z\bigr)\quad\Longrightarrow \quad H_{p+q}(\I_g;\Z),
\end{equation}
where $\mathfrak{X}_p$ is a set of representatives for $\I_g$-orbits of $p$-cells of~$\B_g$, see Section~\ref{section_CL} for details. Recall that $E^r_{p,q}$ is a first-quadrant spectral sequence, its differentials $d^r$ have bi-degrees~$(-r,r-1)$, respectively, and the limiting sheet $E^{\infty}_{p,q}$ is the graded group associated with certain filtration in~$H_{p+q}(\I_g)$. In particular, $E^{\infty}_{0,2}$, $E^{\infty}_{1,1}$, and~$E^{\infty}_{2,0}$ are the associated graded parts of~$H_2(\I_g;\Z)$. So $H_2(\I_g;\Z)$ is not finitely generated if and only if at least one of the three groups~$E^{\infty}_{0,2}$, $E^{\infty}_{1,1}$, and~$E^{\infty}_{2,0}$ is not finitely generated. In this paper, we provide some evidence for the conjecture that, in case~$g=3$, the group~$E^{\infty}_{0,2}$ and hence $H_2(\I_3;\Z)$ are infinitely generated, and thus, $\I_3$ is not finitely presented.

Suppose that $g\ge 3$. It is not hard to see that, for each vertex $v$ of~$\B_g$, the group $H_2\bigl(\Stab_{\I_g}(v);\Z\bigr)$ is nontrivial. Indeed,  this group contains an abelian cycle $\CA(T_{\gamma},T_{\delta})$ corresponding to two commuting Dehn twists~$T_{\gamma}$ and~$T_{\delta}$ about disjoint separating simple closed curves, and non-triviality of this abelian cycle can be deduced from a result by Brendle and Farb~\cite{BrFa07}, see Proposition~\ref{propos_theta_main}(a) of the present paper for a detailed proof in case $g=3$. Since there are  infinitely many $\I_g$-orbits of vertices in~$\B_g$, we obtain that the group~$E^1_{0,2}$ is not finitely generated.
The group $E^{\infty}_{0,2}=E^4_{0,2}$, which injects into~$H_2(\I_g;\Z)$, is obtained from~$E^1_{0,2}$ by taking consecutive quotients by the images of the differentials~$d^1$, $d^2$, and~$d^3$. So the question is whether the group~$E^1_{0,2}$ remains infinitely generated after taking these three quotients. The aim of this paper is to prove that in case $g=3$ this group remains infinitely generated after taking the quotients by the images of the first two differentials, $d^1$ and~$d^2$.

\begin{theorem}\label{theorem_main}
Suppose that $E^r_{p,q}$ is spectral sequence~\eqref{eq_SpSeqIntro} for the action of~$\I_3$ on~$\B_3$. Then the group~$E^3_{0,2}$ is not finitely generated.
\end{theorem}

To prove this theorem we will construct infinitely many linearly independent homomorphisms $\btheta_A\colon E^1_{0,2}\to\Z/2$ that vanish on the images of the differentials~$d^1$ and~$d^2$.
Certainly, the most interesting question is whether the homomorphisms~$\btheta_A$, or at least some infinite number of linear combinations of them, vanish on the image of the differential~$d^3$. If this was the case, we would obtain that $H_2(\I_3;\Z)$ is not finitely generated and hence~$\I_3$ is not finitely presented. The main difficulty on this way is that we lack good description of the group~$E^3_{3,0}$, from which~$d^3$ maps to~$E^3_{0,2}$. In fact, it is not so hard to describe explicitly the group~$E^2_{3,0}=H_3(\B_3/\I_3;\Z)$. The desired group~$E^3_{3,0}$ is then the kernel of the differential $d^2$ that maps~$E^2_{3,0}$ to the group~$E^2_{1,1}$, the structure of which is also completely unclear. So we are rather far from computing the group~$E^3_{3,0}$. Actually, the author even does not know how to present at least one nonzero element of it.

Since we always deal with homomorphisms $\btheta_A\colon E^r_{0,2}\to\Z/2$, it may seem that we could work over~$\Z/2$ from the very beginning and study the spectral sequence with coefficients in the group~$\Z/2$ rather than in~$\Z$. However, this is not the case. In fact, our proof uses considerably divisibility properties of  certain  homomorphisms modulo~$4$, see, e.\,g., Propositions~\ref{propos_1} and~\ref{propos_2}. So the author does not know whether Theorem~\ref{theorem_main} remains true for the spectral sequence with coefficients in~$\Z/2$. At least our proof cannot be carried over to this case without changes.

Now, let us give a brief overview of some known results on homology of the Torelli groups. The result of  Mess~\cite{Mes92} mentioned above yields that $H_1(\I_2;\Z)$ is a  free abelian group of infinite rank. Johnson and Millson developed the approach by Mess and proved that the group~$H_3(\I_3;\Z)$ contains a free abelian subgroup of infinite rank, cf.~\cite{Mes92}. Later Hain~\cite{Hai02} showed that the group~$H_4(\I_3;\Z)$ also contains a free abelian subgroup of infinite rank. The first result for general genus was obtained by Akita~\cite{Aki01} who proved that the total homology group $H_*(\I_g;\Q)$ is an infinite-dimensional vector space, provided that  $g\ge 7$. 

Most of further results in this direction use spectral sequence~\eqref{eq_SpSeqIntro} for the action of~$\I_g$ on the complex of cycles~$\B_g$. In the original paper~\cite{BBM07}, in which the complex~$\B_g$ was first constructed, Bestvina, Bux, and Margalit used it to show that the cohomological dimension of~$\I_g$ is equal to $3g-5$ and the top homology group $H_{3g-5}(\I_g;\Z)$ contains a free abelian subgroup of  infinite rank. Recently, the author~\cite{Gai18} has used the same spectral sequence to show that every group $H_k(\I_g;\Z)$, where $2g-3\le k<3g-5$, also contains a free abelian subgroup of  infinite rank. This method has also proven to be useful for study of homology of another important subgroup of~$\Mod(S_g)$, the \textit{Johnson kernel}~$\K_g$, which is the group generated by all Dehn twists about separating simple closed curves (see~\cite{Joh85a}). Namely, it has been proven that the cohomological dimension of~$\K_g$ is equal to~$2g-3$ (Bestvina, Bux, and Margalit~\cite{BBM07}) and the top homology group~$H_{2g-3}(\K_g;\Z)$ contains a free abelian subgroup of  infinite rank (the author~\cite{Gai19}). Nevertheless, still nothing is known on the infiniteness properties of the homology of~$\I_g$ beneath the dimension~$2g-3$ (excluding dimension~$1$), in particular, in dimension~$2$. In contrast, the first homology group~$H_1(\I_g;\Z)$ is finitely generated whenever $g\ge 3$, and was computed explicitly by Johnson~\cite{Joh85b}. Note also that Hain~\cite{Hai97} proved that the Malcev Lie algebra associated with the Torelli group~$\I_g$ is finitely presented, provided that $g\ge 3$. 

Let us also mention several recent results on representation-theoretic aspects of the homology groups~$H_2(\I_g;\Z)$. Kassabov and Putman proved that, for $g\ge 3$, $H_2(\I_g;\Z)$ is a finitely generated  $\Sp(2g,\Z)$-module and moreover the Torelli group~$\I_g$ has finite $\Mod(S_g)$-equivariant presentation, see~\cite[Theorems~A and~C]{KaPu18}. Miller, Patzt, and Wilson~\cite{MPW19} proved central stability for the series of groups~$H_2(\I_g;\Z)$. Kupers and Randal-Williams~\cite{KuRW20} calculated  $H^2(\I_{g,1};\Q)^{\mathrm{alg}}$ as an $\Sp(2g,\Z)$-module in a stable range $g\gg 0$. Here $\I_{g,1}$ is the Torelli group of an oriented genus~$g$ surface with one boundary component and $H^2(\I_{g,1};\Q)^{\mathrm{alg}}$ is the algebraic part of the cohomology group~$H^2(\I_{g,1};\Q)$, i.\,e., the sum of all finite-dimensional $\Sp(2g,\Z)$-subrepresentations that extend to representations of the algebraic group $\Sp(2g)$.

\smallskip

\textbf{Acknowledgements.} This work was started in 2008--2010 during my research visits to Bielefeld as a member of research group headed by Sergei Ivanovich Adian. I am highly indebted to him for introducing me to this wonderful field of study, as well as for his constant interest to my work, and valuable fruitful discussions.  I am grateful to J.~Mennicke for his hospitality during those visits, to I.\,A.~Spiridonov, A.\,L.~Talambutsa, and D.\,S.~Ulyumdzhiev for useful discussions and comments, and to the anonymous reviewer for his remarks that have improved the paper considerably.

\section{Scheme of the proof of Theorem~\ref{theorem_main}}\label{section_outline}

In this section we outline the proof of Theorem~\ref{theorem_main}, which takes up the rest of the present paper. All necessary definitions, constructions, etc. will be given later throughout the paper. We consider the genus~$3$ surface only and conveniently put $S=S_3$, $\I=\I_3$, and~$\B=\B_3$.

The construction of the complex of cycles~$\B$ depends on the choice of a primitive homology class $x\in H$, where $H=H_1(S;\Z)$. A precise construction of~$\B$ will be given in Section~\ref{subsection_complex_cycles}. At the moment, we recall only that cells of~$\B$ are indexed by certain oriented multicurves in~$S$; we denote by~$\M_p$ the set of oriented multicurves that index $p$-dimensional cells of~$\B$. 

To an oriented multicurve~$M$ is assigned the multiset~$[M]$ of homology classes of components of~$M$. Let~$\CH_p$ be the set of all multisets~$[M]$, where $M$ runs over~$\M_p$. Then $\I$-orbits of cells of~$\B$ correspond to multisets in~$\CH_p$. This correspondence is not one-to-one. Namely,  one, or two, or infinitely many $\I$-orbits of cells may correspond to a multiset in~$\CH_p$, see Proposition~\ref{propos_orbits} and Remark~\ref{remark_infinite_orbits} for details. Nevertheless, decomposition~\eqref{eq_SpSeqIntro} for~$E^1_{p,q}$ can be rewritten in the form
\begin{equation}\label{eq_CL_2}
E^1_{p,q} = \bigoplus_{C\in\CH_p}E^1_{p,q}(C),
\end{equation}
where $E^1_{p,q}(C)$ is the direct sum of the groups $H_q(\I_M;\Z)$ over representatives of all orbits of oriented multicurves~$M$ with $[M]=C$. (Hereafter, we use notation $\I_M=\Stab_{\I}(M)$.)

Consider a subset $A=\{a_1,a_2,a_3\}$ of~$H$  that satisfies the following conditions:
\begin{itemize}
\item $a_1$, $a_2$,~$a_3$ are linearly independent and generate a direct summand of~$H$,
\item $a_i\cdot a_j=0$ for all~$i$ and~$j$,
\item $x=n_1a_1+n_2a_2+n_3a_3$ for some positive integers~$n_1$, $n_2$, and~$n_3$.
\end{itemize}
Any such~$A$ belongs to~$\CH_0$; let~$\CH_0'$ be the subset of~$\CH_0$ consisting of all $3$-element sets~$A$ that satisfy these three conditions. (Such $3$-element sets do not exhaust the set~$\CH_0$: it also contains the $1$-element set~$\{x\}$ and infinitely many $2$-element sets, which are not so important for our construction.)  For each $A\in\CH_0'$, all $3$-component oriented multicurves~$M$ in~$S$ with $[M]=A$ lie in the same $\I$-orbit and hence $E^1_{0,q}(A)\cong H_q(\I_M;\Z)$ for any such multicurve~$M$.

Recall that there are \textit{Birman--Craggs homomorphisms}  $\rho_{\omega}\colon \I\to\Z/2$ indexed by $\Sp$-quadratic functions $\omega\colon H_1(S;\Z/2)\to\Z/2$ with zero Arf invariant, see Section~\ref{subsection_BC} for details and references. For each set $A=\{a_1,a_2,a_3\}\in \CH_0'$, there are exactly four $\Sp$-quadratic functions~$\omega$ that satisfy $\Arf(\omega)=0$ and $\omega(a_1)=\omega(a_2)=\omega(a_3)=1$; let $\rho_1,\ldots,\rho_4$ be the corresponding Birman--Craggs homomorphisms. These homomorphisms can be considered as elements of the cohomology group~$H^1(\I;\Z/2)$. A central role in our proof will be played by the cohomology class
$$
\theta_A=\sum_{1\le i<j\le 4}\rho_i\rho_j\in H^2(\I;\Z/2),
$$
where multiplication is the cup-product. Using this cohomology class, we can construct the homomorphism
\begin{equation*}
\btheta_A\colon E^1_{0,2}=\bigoplus_{C\in\CH_0}E^1_{0,2}(C)\xrightarrow{\mathrm{projection}} E^1_{0,2}(A)=H_2(\I_M;\Z)\xrightarrow{i_*}
 H_2(\I;\Z)\xrightarrow{\langle\theta_A,\cdot\rangle} \Z/2,
 \end{equation*}
where $M$ is an oriented multicurve with~$[M]=A$ and $i_*$ is the homomorphism induced by the inclusion $\I_M\subset\I$. In Section~\ref{section_theta}, we will show that $\bigl\langle\theta_A,\CA(T_{\gamma},T_{\delta})\bigr\rangle=1$ whenever $\gamma$ and~$\delta$ are two non-homotopic separating simple closed curves that are disjoint from~$M$ and from each other (see Proposition~\ref{propos_theta_main}(a)). This immediately implies that the homomorphisms~$\btheta_A\colon E^1_{0,2}\to\Z/2$, where $A\in\CH_0'$, are non-trivial and therefore linearly independent.

The core of our proof is the fact that the constructed homomorphisms~$\btheta_A$ vanish on the images of the differentials~$d^1$ and~$d^2$ and hence induce linearly independent homomorphisms $\btheta_A\colon E^3_{0,2}\to\Z/2$, $A\in\CH_0'$. Since the set $\CH_0'$ is infinite, this immediately implies Theorem~\ref{theorem_main}.

Vanishing of~$\btheta_A$ on the image of the differential $d^1\colon E^1_{1,2}\to E^1_{0,2}$ is proved as follows. The group~$E^1_{1,2}$ is the direct sum of the groups~$H_2(\I_M;\Z)$, where $M$ runs over some set of representatives of $\I$-orbits in~$\M_1$. So it is sufficient to prove that $\btheta_A(d^1y)=0$ for $y$ in every summand~$H_2(\I_M;\Z)$. Let $P_{N_1}$ and~$P_{N_2}$ be the endpoints of the $1$-cell~$P_M$. Then $N_1$ and~$N_2$ are oriented multicurves that are contained in~$M$ and belong to~$\M_0$. If  $[N_1]\ne A$ and $[N_2]\ne A$, then the component of~$d^1y$ in the summand~$E^1_{0,2}(A)$  vanishes and hence $\btheta_A(d^1y)=0$. If $[N_1]=[N_2]=A$, then we again have $\btheta_A(d^1y)=0$, since the endpoints of~$P_M$ give equal impacts to~$\btheta_A(d^1y)$. In Section~\ref{section_4component} we study in detail the stabilizers $\I_M$ for all $M\in\M_1$ such that exactly one of the two conditions $[N_1]=A$ and~$[N_2]=A$ holds, and in particular prove that the restriction of~$\theta_A$ to~$\I_M$ vanishes for any such~$M$ (see Proposition~\ref{propos_theta_main}(b)). This implies that $\btheta_A(d^1y)=0$ in this case, too.

The proof of the vanishing of~$\btheta_A$ on the image of the differential $d^2\colon E^2_{2,1}\to E^2_{0,2}$ is much harder. The main difficulty consists in the fact that we have no good description of the kernel of the differential $d^1\colon E^1_{2,1}\to E^1_{1,1}$ and hence of the group~$E^2_{2,1}$. The reason is that the differential~$d^1$ mixes different summands in the direct sum decomposition~\eqref{eq_CL_2} for~$E^1_{2,1}$. This difficulty will be overcome as follows.

First of all, the spectral sequence~$E^*_{*,*}$ we consider is the spectral sequence associated with the filtration by columns in the double complex 
$$
B_{p,q}=C_p(\B;\Z)\otimes_{\I}\res_q, 
$$ 
where $C_*(\B;\Z)$ is the cellular chain complex of~$\B$ and $\res_*$ is the bar resolution for~$\Z$ over~$\Z\I$, see Section~\ref{section_CL} for details. Let $\partial'$ and~$\partial''$ be the differentials of this double complex induced by the differentials of~$C_*(\B;\Z)$ and~$\res_*$, respectively. The sheet $E^1_{*,*}$ is the homology of the double complex~$B_{*,*}$ with respect to~$\partial''$. We have the following analog of direct sum decomposition~\eqref{eq_CL_2}:
\begin{equation*}%\label{eq_B_dirsum}
B_{p,q}=\bigoplus_{C\in\CH_p}B_{p,q}(C),
\end{equation*}
where $B_{p,q}(C)$ is the subgroup generated by all elements~$P_M\otimes \xi$ with $[M]=C$.

To study the differential~$d^2$ we have to work directly with the double complex~$B_{*,*}$.  Namely, an element~$Y\in B_{2,1}$ represents a class  $y\in E^2_{2,1}$ if and only if there exists an element $X\in B_{1,2}$ satisfying $\partial'Y+\partial''X=0$. Further, if this condition is satisfied, then~$\partial'X$ represents the class~$d^2y$ in~$E^2_{0,2}$.

Nevertheless, working on the level of the double complex~$B_{*,*}$ does not solve our problem: we still do not know how to characterize those~$Y\in B_{2,1}$ for which there exists an $X\in B_{1,2}$ satisfying $\partial'Y+\partial''X=0$. We avoid such characterization in the following rather cumbersome way. At the moment, we just list the objects, which will be then constructed throughout the paper, and claim their properties that will be important for us. We will construct the following objects:

\begin{itemize}
\item  Subgroups  $\Gamma\subseteq B_{0,2}$ and $\Delta\subseteq B_{1,1}$. Their description will be effective  in the following sense: $\Gamma$ (respectively, $\Delta$) is the direct sum of certain explicitly defined subgroups $\Gamma_A\subseteq B_{0,2}(A)$, where $A$ runs over~$\CH_0$ (respectively, $\Delta_C\subseteq B_{1,1}(C)$, where $C$ runs over~$\CH_1$).
\item Homomorphisms $\Theta_A\colon \Gamma\to\Z/2$, where $A\in\CH_0'$, such that $\btheta_A(z)=\Theta_A(Z)$ whenever an element $Z\in\Gamma$ represents a class $z\in E^1_{0,2}$.
\item Homomorphisms $\Phi_{C,A}\colon \Delta\to\Z/2$ indexed by pairs $(C,A)$ such that $C\in\CH_1$, $A\in\CH_0'$, and $C\supset A$.
\item Homomorphisms $\sigma_D$ of~$E^1_{2,1}$ to a certain vector space over~$\Z/2$ and homomorphisms $\nu_D\colon E^1_{2,1}\to\Z$ indexed by $D\in\CH_2$.
\item Homomorphisms $\kappa_A\colon E^1_{1,1}\to \Z/2$, where $A\in\CH_0'$.
\end{itemize}

These groups and homomorphisms will have the following properties:
\begin{enumerate}
\item[(P1)] Suppose that $U\in\Delta\cap \partial''(B_{1,2})$; then there exists an element $X\in B_{1,2}$ such that $\partial'' X=U$, $\partial' X\in \Gamma$, and for all $A\in \CH_0'$,
\begin{equation}\label{eq_Delta_main_prelim}
\Theta_A(\partial' X)=\sum_{C\in\CH_1,\, C\supset A}\Phi_{C,A}(U).
\end{equation}
\item[(P2)] Suppose that $y\in E^1_{2,1}$ and $d^1y=0$; then  $\sigma_D(y)=0$ and $\nu_D(y)$ is divisible by~$4$ for all~$D\in \CH_2$.
\item[(P3)] Suppose that $y\in E^1_{2,1}$ is an element such that $\sigma_D(y)=0$ and $\nu_D(y)$ is divisible by~$4$ for all~$D\in \CH_2$. Then there exist elements $Y\in B_{2,1}$ and $X\in B_{1,2}$ satisfying the following conditions:
\begin{itemize}
\item $\partial''Y=0$ and $Y$ represents~$y$ in~$E^1_{2,1}$, 
\item $\partial'Y+\partial''X\in\Delta$, 
\item $\partial'X\in\Gamma$, 
\item for all $A\in\CH_0'$, we have $\Theta_A(\partial'X)=0$  and
\begin{equation}\label{eq_key2_prelim}
\sum_{C\in\CH_1,\, C\supset A}\Phi_{C,A}(\partial'Y+\partial''X)=\kappa_A(d^1y).
\end{equation}
\end{itemize}
\end{enumerate}

Once the groups and the homomorphisms with these properties are constructed, vanishing of~$\btheta_A$ on the image of~$d^2$ will follow immediately. Indeed, any element in~$E^2_{2,1}$ can be represented by a cycle $y\in E^1_{2,1}$ satisfying $d^1y=0$. By properties~(P2) and~(P3) there exist elements $Y\in B_{2,1}$ and~$X\in B_{1,2}$ satisfying all conditions in property~(P3). Since $d^1y=0$, we see that $\partial'Y\in\partial''(B_{1,2})$ and hence  $\partial'Y+\partial''X\in\Delta\cap\partial''(B_{1,2})$. Now, by property~(P1) there exists an element $X_1\in B_{1,2}$ such that $\partial''X_1=\partial'Y+\partial''X$, $\partial' X_1\in \Gamma$, and for all $A\in \CH_0'$,
\begin{equation*}
\Theta_A(\partial' X_1)=\sum_{C\in\CH_1,\, C\supset A}\Phi_{C,A}(\partial'Y+\partial''X)=\kappa_A(d^1y)=0.
\end{equation*}
Then the element $\partial'(X-X_1)$ represents the class~$d^2y$. Therefore, for all $A\in \CH_0'$,
$$
\btheta_A(d^2y)=\Theta_A\bigl(\partial'(X-X_1)\bigr)=0.
$$

\begin{remark}
There are several types of oriented multicurves in the set~$\M_1$, see Section~\ref{section_cells} for details. Correspondingly, there are several types of multisets $C\in\CH_1$. In fact, the homomorphisms $\Phi_{C,A}$ will be introduced only for some (not all) types of multisets $C\in\CH_1$ that contain~$A$. Similarly, the homomorphisms~$\sigma_D$ and~$\nu_D$ will be introduced only for some (not all) types of multisets $D\in\CH_2$.  
(In all other cases, we may by definition assume that these homomorphisms are trivial.) Besides, the homomorphisms~$\Phi_{C,A}$ will have different nature for two different types of~$C$. Because of that we will conveniently use different notation for them, namely, $\Psi_C$ for the first type and~$\Phi_{C,A}$ for the second type. (In the former case index~$A$ is omitted, since the homomorphism is in fact independent of~$A$.) Precise forms  of equations~\eqref{eq_Delta_main_prelim} and~\eqref{eq_key2_prelim} are equations~\eqref{eq_Delta_main} and~\eqref{eq_key2} in Propositions~\ref{propos_Delta} and~\ref{propos_2}, respectively. Note also that the homomorphism $\kappa_A$ will be made up of several different homomorphisms, so $\kappa_A(y)$ is just a short notation for the right-hand side of~\eqref{eq_key2}. In fact, we will never use notation~$\kappa_A$ in the sequel.
\end{remark}

The most difficult and intriguing part of the paper is the proof of property~(P2). The matter  is that it seems likely that there is no local obstruction to the existence of a class $y\in E^1_{2,1}$ satisfying $d^1y=0$ but violating some of the conditions $\sigma_D(y)=0$ and $\nu_D(y)\equiv 0\pmod 4$. This means that if we completed the group~$E^1_{2,1}$ by allowing infinite (but locally finite with respect to the topology in~$\B/\I$) sums of the elements of summands~$E^1_{2,1}(D)$, then in the obtained group such class~$y$ would likely exist. So the proof of the non-existence of such class~$y$ in the initial group~$E^1_{2,1}$ must use some finiteness arguments. 

To show that $\sigma_D(y)=0$ for all~$D$ whenever $y\in E^1_{2,1}$ and $d^1y=0$, we will introduce a function $w$ on~$\CH_2$ and prove the following assertion:

\smallskip
\textit{Suppose that $y\in E^1_{2,1}$, $d^1y=0$, and $\sigma_{D_0}(y)\ne 0$ for some $D_0\in\CH_2$; then there exists $D\in\CH_2$ such that $\sigma_{D}(y)\ne 0$ and $w(D)>
w(D_0)$.} 

\smallskip
This assertion implies that if the condition $\sigma_D(y)=0$ is violated for at least one $D$, then it is violated for infinitely many $D$'s, which is impossible.  The function~$w$ we use is as follows. For each $A=\{a_1,a_2,a_3\}\in\CH_0'$, we put $n(A)=n_1+n_2+n_3$, where $n_i$ are the coefficients in the decomposition $x=n_1a_1+n_2a_2+n_3a_3$. Then we take for $w(D)$ the maximum of~$n(A)$ over all subsets $A\subseteq D$ belonging to~$\CH_0'$. Note that the idea of using the function~$n$ to study the topology of the complex of cycles is not new --- the same function was used by Hatcher and Margalit to study the relative topology of the complex of homologous curves in the complex of cycles, see~\cite[Section~3]{HaMa12}. 

To show that $\nu_D(y)\equiv 0\pmod 4$ for all~$D$ whenever $y\in E^1_{2,1}$ and $d^1y=0$, we will proceed in a similar way. Nevertheless, the Hatcher--Margalit weight function $n$ does not help us any more. Instead of it we need to use a function~$F$ of completely different nature. Namely, we consider a linear function $f\colon H\to\R$ such that $f(x)=0$ and the image of~$f$ is a subgroup of~$\R$ of rank~$5$, and then roughly speaking put  $F(A)=\max|f(a_i)-f(a_j)|$, for more details see Section~\ref{subsection_b}.

This paper is organized as follows. Section~\ref{section_prelim} contains preliminary constructions and results, which then are used throughout the paper.  Namely, we recall necessary definitions and facts on the complex of cycles (Subsection~\ref{subsection_complex_cycles}), the spectral sequence associated with cellular action of a group (Subsection~\ref{section_CL}), the stabilizers of multicurves under the action of the Torelli group (Subsection~\ref{subsection_stabilizers}), the Birman--Craggs homomorphisms  (Subsection~\ref{subsection_BC}), and the fundamental cohomological exact sequence associated with a short exact sequence of groups (Subsection~\ref{subsection_fundamental}). All results in this section either are known or follow easily from known results.

In Section~\ref{section_theta} we construct infinitely many homomorphisms $\btheta_A\colon E^1_{0,2}\to \Z/2$, prove their linear independence (Proposition~\ref{propos_li}), and state our many result on their vanishing on the images of the differentials~$d^1$ and~$d^2$, Theorem~\ref{theorem_main_section}. The rest of the paper contains the proof of this theorem. In Section~\ref{section_cells} we describe various types of oriented multicurves in~$\M$ and obtain several results on decomposition of~$\M$ into $\I$-orbits. In Section~\ref{section_nu} we define several useful homomorphisms of the stabilizers~$\I_{\gamma}$ and~$\I_{\gamma\cup\gamma'}$ to~$\Z$, where $\gamma$ is a simple closed curve and~$\{\gamma,\gamma'\}$ is a bounding pair. These homomorphisms are then used throughout the whole paper. Section~\ref{section_4component} contains a detailed study of the stabilizers~$\I_M$ of $4$- and $5$-component multicurves~$M$. This serves two purposes. First, in this section we complete the proof  of Proposition~\ref{propos_theta_main}, which is the main ingredient of the proof of vanishing of~$\btheta_A$ on the image of~$d^1$. Second, the results on stabilizers obtained in Section~\ref{section_4component} are then used in Sections~\ref{section_several_hom} and~\ref{section_Delta}  to construct homomorphisms that enter formulae~\eqref{eq_Delta_main_prelim} and~\eqref{eq_key2_prelim}. Namely, in Section~\ref{section_several_hom} we construct necessary homomorphisms of~$E^1_{1,1}$ and~$E^1_{2,1}$ to abelian groups, and in Section~\ref{section_Delta} we construct the groups $\Gamma\subseteq B_{0,2}$ and~$\Delta\subseteq B_{1,1}$, the homomorphisms $\Theta_A$ of~$\Gamma$ to~$\Z/2$, and the homomorphisms~$\Phi_{C,A}$ and~$\Psi_C$ of~$\Delta$ to~$\Z/2$. Also, in this section we prove Proposition~\ref{propos_Delta}, which is a more precise version of the above property~(P1). In Section~\ref{section_scheme} we formulate Propositions~\ref{propos_1} and~\ref{propos_2}, which are more precise versions of the above properties~(P2) and~(P3), respectively. These propositions are proved in Sections~\ref{section_1} and~\ref{section_2}, respectively, which completes the proof of vanishing of~$\btheta_A$ on the image of~$d^2$.

\section{Preliminaries}\label{section_prelim}

We use the standard group-theoretic notation
$$
[g,h]=g^{-1}h^{-1}gh,\qquad g^{h}=h^{-1}gh.
$$

For an element~$g$ of a group~$G$, we denote by~$[g]$ the homology class of~$g$ in~$H_1(G;\Z)$.

Suppose that $G$ is a group and $g$ and~$h$ are commuting elements of it. Consider the homomorphism $\varphi\colon\Z\times\Z\to G$ that takes the generators of the factors~$\Z$ to~$g$ and~$h$, respectively. Then the \textit{abelian cycle}~$\CA(g,h)\in H_2(G;\Z)$ is, by definition, the image of the standard generator of~$H_2(\Z\times\Z;\Z)$ under~$\varphi_*$.

We always mean that all simple closed curves we consider are homotopically non-trivial unless otherwise stated explicitly. We denote by~$T_{\gamma}$ the left Dehn twist about a simple closed curve~$\gamma$.  A \textit{bounding pair} is a pair~$\{\gamma,\gamma'\}$  of simple closed curves such that $\gamma$ and~$\gamma'$ are non-separating, disjoint from each other, not isotopic to each other, and their union divides the surface into two components. We use notation $T_{\gamma,\gamma'}=T_{\gamma}T_{\gamma'}^{-1}$; this mapping class is called the twist about the bounding pair~$\{\gamma,\gamma'\}$. We denote by~$\Mod(S)$ the mapping class group of an oriented surface~$S$, which can be closed, or with punctures, or with boundary components. In the last case, we mean that all mapping classes fix every boundary component pointwise. For a surface~$S$ with punctures, the \textit{pure mapping class group} of~$S$ is the subgroup~$\PMod(S)$ of~$\Mod(S)$ consisting of all mapping classes that do not permute punctures.  Necessary background material on simple closed curves and mapping class groups can be found in~\cite{FaMa12}.

Suppose that the Torelli group~$\I_g$ acts on a set~$\mathcal{Y}$. Then in order to simplify notation, we will denote the stabilizer~$\Stab_{\I_g}(Y)$ of an element~$Y\in\mathcal{Y}$ by~$\I_Y$. Similar notation~$\hI_Y$,~$\CC_Y$ will be used for other groups~$\hI_g$, $\CC_g$, which will be introduced throughout the paper.

We always put $H=H_1(S_g;\Z)$ and~$H_{\Z/2}=H\otimes(\Z/2)=H_1(S_g;\Z/2)$. We denote by~$[\gamma]$ the integral homology class of an oriented simple closed curve~$\gamma$.   We denote by $x\cdot y$ the intersection form on~$H$. A basis $a_1,\ldots,a_g,b_1,\ldots,b_g$ of~$H$ is called \textit{symplectic} if it satisfies $a_i\cdot a_j=0$, $b_i\cdot b_j=0$, and~$a_i\cdot b_j=\delta_{ij}$, where $\delta_{ij}$ is the Kronecker delta. A \textit{symplectic basis} of~$H_{\Z/2}$ is defined similarly. We denote modulo~$2$ homology classes by bold letters~$\bx,\by$, etc. to distinguish them from integral homology classes.

For a sugbroup~$U$ of~$H$, we denote by~$U^{\bot}$ the orthogonal complement of~$U$ with respect to the intersection form. Under an orthogonal splitting of~$H$, we always mean a splitting into a direct sum of subgroups so that each pair of them are orthogonal to each other with respect to the intersection form.
A subgroup~$L$ of~$H$ is called \textit{isotropic} if the restriction of the  intersection form to~$L$ is trivial. A \textit{Lagrangian subgroup} is a maximal with respect to inclusion isotropic subgroup. It is easy to see that an isotropic subgroup $L$ is Lagrangian if and only if it has rank~$g$ and is a direct summand of~$H$. Besides, $L$ is Lagrangian if and only if $L^{\bot}=L$. We denote by $\langle c_1,\ldots,c_k\rangle$ the subgroup of~$H$ spanned by~$c_1,\ldots,c_k$.

\subsection{Complex of cycles}\label{subsection_complex_cycles}

Bestvina, Bux, and Margalit~\cite{BBM07} constructed a contractible cell complex called the \textit{complex of cycles} on which the Torelli group~$\I_g$ acts cellularly and \textit{without rotations}. The latter means that an element $h\in\I_g$ fixes a cell of~$\B_g$ pointwise whenever~$h$ fixes this cell setwise. In this section we recall their construction. More details can be found in~\cite{HaMa12} and~\cite{Gai18}.  

Consider a closed oriented surface~$S_g$, where $g\ge 2$.  A  class $x\in H=H_1(S_g;\Z)$ is called \textit{primitive} if it is not divisible by any integer greater than~$1$. Choose a primitive homology class $x\in H$ and fix this choice throughout the whole construction. (Bestvina, Bux, and Margalit required only that $x$ is nonzero. However,  it is convenient that $x$ be primitive.)

Let $\BC$ be the set of all isotopy classes of oriented non-separating simple closed curves in~$S_g$. Consider the infinite-dimensional vector space~$\R^{\BC}$ consisting of formal finite linear combinations of elements of~$\BC$. Note that if $\bar\gamma$ is a curve $\gamma$ with the orientation reversed, then $\gamma$ and~$\bar\gamma$ are two different basis elements of~$\R^{\BC}$, and $\bar\gamma$ is not supposed to be equal to~$-\gamma$. Hereafter, we usually say `a curve' meaning `an isotopy class of curves'. In particular, elements of~$\BC$ will be called curves, and saying `disjoint curves' we mean that the corresponding isotopy classes contain disjoint representatives. 

The \textit{complex of cycles}~$\B_g=\B_g(x)$ is, by definition, the subset of~$\R^{\BC}$ consisting of all finite linear combinations $\bgamma=\sum_{i=1}^k n_i\gamma_i$, which will be called \textit{cycles}, such that 
\begin{itemize}
\item $n_1,\ldots,n_k$ are positive reals,
\item $\gamma_1,\ldots,\gamma_k$ are pairwise disjoint and pairwise non-isotopic oriented simple closed curves,
\item $n_1[\gamma_1]+\cdots+n_k[\gamma_k]=x$,
\item the homology classes $[\gamma_1],\ldots,[\gamma_k]$ satisfy no \textit{one-sided relations} i.\,e. relations of the form 
$\sum_{i\in I}[\gamma_i]=0$, where $I$ is a nonempty subset of $\{1,\ldots,k\}$.
\end{itemize} 
Note that \textit{two-sided relations} i.\,e. relations of the form $\sum_{i\in I}[\gamma_i]=\sum_{j\in J}[\gamma_j]$ are allowed. In particular, the last condition in the definition of~$\B_g$ prohibits for two oriented curves~$\gamma_i$ and~$\gamma_j$ to have opposite homology classes. A cycle~$\bgamma\in\B_g$ is called a \textit{basic cycle} for~$x$ if the homology classes of the curves $\gamma_1,\ldots,\gamma_k$ are linearly independent. For a basic cycle, the coefficients $n_i$ are positive integers, since the class~$x$ is integral.  

A union of a finite number of pairwise disjoint and pairwise non-isotopic simple closed curves in~$S_g$ is called a \textit{multicurve}. A multicurve is said to be \textit{oriented} if all its components are endowed with orientations. In this paper we always assume that all multicurves being considered contain no separating components. As for curves, we consider multicurves up to isotopies and do not distinguish between a multicurve and its isotopy class. Let us agree that, for oriented multicurves $M$ and~$N$, notation $N\subset M$ means that $N$ is the union of several components of~$M$ with the same orientations as in~$M$. For a cycle~$\bgamma\in\B_g$, the union of oriented curves~$\gamma_i$ that enter~$\bgamma$ with nonzero coefficients is an oriented multicurve. It will be called the \textit{support} of~$\bgamma$.

Let $\M$ denote the set of all oriented multicurves $M$ in~$S_g$ such that 
\begin{enumerate}
\item for each component $\gamma$ of~$M$, there exists a basic cycle for $x$ whose support is contained in~$M$ and contains~$\gamma$,
\item the homology classes of the components of~$M$ satisfy no one-sided relations. 
\end{enumerate}
Bestvina, Bux, and Margalit showed that an oriented multicurve $M$ belongs to~$\M$ if and only if it is the support of a cycle $\bgamma\in\B_g$. Moreover, for each $M\in\M$, the set~$P_M$ consisting of all cycles $\bgamma\in\B_g$ with supports contained in~$M$ is a finite-dimensional convex polytope in~$\R^{\BC}$ whose vertices are exactly all basic cycles for~$x$ with supports contained in~$M$. Besides, they proved the following formula for the dimension of~$P_M$ (cf.~\cite[Lemma~2.1]{BBM07}):
\begin{equation}\label{eq_dimension}
\dim P_M=|M|-\rank M=|S\setminus M|-1,
\end{equation}
where $|M|$ is the number of components of~$M$, $|S\setminus M|$ is the number of components of~$S\setminus M$ and $\rank M$ is the rank of the subgroup of~$H$ spanned by the homology classes of components of~$M$. The maximum of~$\dim P_M$ is attained when $M$ is a pants decomposition of~$S_g$. Then $|M|=3g-3$ and $\rank M=g$. Hence $\dim\B_g=2g-3$, cf.~\cite[Proposition~2.2]{BBM07}. In particular, $\dim\B_3=3$.

Bestvina, Bux, and Margalit proved that the CW complex~$\B_g$ is contractible, see~\cite[Theorem~E]{BBM07}. The fact that~$\I_g$ acts on~$\B_g$ without rotations follows from a result of Ivanov~\cite[Corollary~1.8]{Iva92}.

The convex polytopes~$P_M$, where $M$ runs over~$\M$, form a regular cell decomposition of~$\B_g$. Let $C_*(\B_g;\Z)$ denote the cellular chain complex of~$\B_g$. For $M,N\in\M$, it is easy to see that $P_N$ is a face of~$P_M$ if and only if $N\subset M$. Hence the differential in~$C_*(\B_g;\Z)$ is given by
\begin{equation*}
\partial P_M=\sum_{N\subset M,\,N\in\M,\,\dim P_N=\dim P_M-1,}\pm P_N,
\end{equation*}
where signs depend on the chosen orientations of the cells.

Alongside with the Torelli group~$\I_g$, we need to consider the \textit{extended Torelli group}~$\hI_g$. By definition, $\hI_g$ is the subgroup of~$\Mod(S_g)$ consisting of all mapping classes~$h$ that act on~$H$ either trivially or by multiplication by~$-1$. Obviously, $\I_g$ is an index~$2$ subgroup of~$\hI_g$, and $\hI_g$ is generated by~$\I_g$ and a hyperelliptic involution (see~Fig.~\ref{fig_eBC} in Section~\ref{subsection_BC}).  

If we try to extend naively the action of~$\I_g$ on~$\B_g$ to the action of~$\hI_g$, we fail, since a mapping class $h\in\hI_g\setminus\I_g$ takes any cycle~$\bgamma$ for~$x$ to a cycle for~$-x$, which does not lie in~$\B_g$. Nevertheless, we may define an action of~$\hI_g$ on~$\B_g$ by artificially reversing the orientation of every curve entering the cycle whenever we act on it by a mapping class~$h\in\hI_g\setminus\I_g$. In other words, if a  mapping class~$h\in\hI_g\setminus\I_g$ takes oriented curves $\gamma_1,\ldots,\gamma_k$ to oriented curves $\delta_1,\ldots,\delta_k$, respectively, then we put 
$$
h(n_1\gamma_1+\cdots+n_k\gamma_k)=n_1\overline{\delta}_1+\cdots+n_k\overline{\delta}_k.
$$
Similarly, if $h\in\hI_g\setminus\I_g$ and~$M$ is an oriented multicurve in~$\M$, then we denote by~$h(M)$ the multicurve obtained by applying naively~$h$ to~$M$ and then reversing the orientation of every component. Then $h(M)$ again belongs to~$\M$.

For an oriented multicurve $M\in\M$, we denote by~$[M]$ the multiset of integral homology classes of components of~$M$. Generally, $[M]$ is a multiset rather than a set, since $M$ may contain homologous components. We will use square brackets~$[c_1,\ldots,c_k]$ to list elements of a multiset. However, we prefer to use a more common notation~$\{c_1,\ldots,c_k\}$ whenever a multiset is in fact a set. We denote by~$\CH$ the set consisting of all multisets~$[M]$, where $M$ runs over~$\M$. For each $p$, we denote by~$\M_p$ the subset of~$\M$ consisting of all~$M$ such that $\dim P_M=p$, and we denote by~$\CH_p$ the set consisting of all multisets~$[M]$, where $M$ runs over~$\M_p$. It follows from~\eqref{eq_dimension} that all multisets in~$\CH_0$ are sets.

\begin{propos}\label{propos_H0}
An $m$-element set $A=\{a_1,\ldots,a_m\}$ belongs to~$\CH_0$ if and only if it satisfies the following conditions:
\begin{enumerate}
\item $a_1,\ldots,a_m$ is a basis of a rank~$m$ isotropic direct summand of~$H_1(S_g;\Z)$,
\item $x=n_1a_1+\cdots+n_ma_m$ for some positive integers $n_1,\ldots,n_m$.
 \end{enumerate}
\end{propos}

\begin{proof}
It is well known that~(1) is a necessary and sufficient conditions for the existence of an oriented multicurve~$N$ with~$[N]=A$ such that $S\setminus N$ is connected. This multicurve belongs to~$\M_0$ if and only if (2) holds.  
\end{proof}

\begin{propos}\label{propos_H}
Suppose that $g=3$. Let $A=\{a_1,a_2,a_3\}$ be a three-element set belonging to~$\CH_0$ and $C$ be a multiset that contains~$A$. Then $C$ belongs to~$\CH$ if and only if it is contained in one of the following $102$ multisets: 
\begin{itemize}
\item $3$ sets $\{a_1,a_2,a_3,a_i+a_j,a_j+a_k,a_1+a_2+a_3\}$,
\item $12$ sets $\{a_1,a_2,a_3,a_i+a_j,a_k-a_j,\epsilon (a_i+a_j-a_k)\}$, where $\epsilon\in\{-1,1\}$,
\item $6$ sets $\{a_1,a_2,a_3,a_i+a_j,a_j-a_k,a_i+a_j-a_k\}$,
\item $12$ sets $\{a_1,a_2,a_3,\epsilon_1(a_i-a_k),\epsilon_2(a_j-a_k),a_i+a_j-a_k\}$, where $\epsilon_1,\epsilon_2\in\{-1,1\}$,
\item $3$ sets $\{a_1,a_2,a_3,a_k-a_i,a_k-a_j,a_k-a_i-a_j\}$,
\item $6$ sets $\{a_1,a_2,a_3,a_i+a_j,a_j+a_k,a_i-a_k\}$,
\item $6$ sets $\{a_1,a_2,a_3,a_i-a_j,a_j-a_k,a_i-a_k\}$,
\item $3$ multisets $[a_1,a_2,a_3,a_i+a_j,a_i+a_j,a_1+a_2+a_3]$,
\item $6$ multisets $[a_1,a_2,a_3,a_i+a_j,a_i+a_j,\epsilon(a_i+a_j-a_k)]$, where $\epsilon\in\{-1,1\}$,
\item $12$ multisets $[a_1,a_2,a_3,a_i-a_j,a_i-a_j,\epsilon(-a_i+a_j+a_k)]$, where $\epsilon\in\{-1,1\}$,
\item $6$ multisets $[a_1,a_2,a_3,a_i-a_j,a_i-a_j,a_i-a_j+a_k]$,
\item $3$ multisets $[a_i,a_i,a_j,a_k, a_i+a_j,a_i+a_k]$,
\item $12$ multisets $[a_i,a_i,a_j,a_k, a_i+a_j,\epsilon(a_i-a_k)]$, where $\epsilon\in\{-1,1\}$,
\item $12$ multisets $[a_i,a_i,a_j,a_k,\epsilon_1(a_i-a_j),\epsilon_2(a_i-a_k)]$, where $\epsilon_1,\epsilon_2\in\{-1,1\}$.
\end{itemize}
Here in all cases $i,j,k$ is an arbitrary permutation of~$1,2,3$.
\end{propos}

\begin{proof}
First, it can be checked by a direct case analysis that for each of the $102$ multisets~$E$ listed, there exists an oriented multicurve~$K$ in~$S_3$ with $[K]=E$.  Suppose that $C$ is a multiset contained in~$E$ and containing~$A$. Then any multicurve $M\subset K$ such that  $[M]=C$ belongs to~$\M$. Indeed, condition~(2) in the definition of~$\M$ holds, since there are no one-sided relations, and condition~(1) holds, since every element $c$ of each of the listed multisets has the following property: $x$ is a linear combination with nonnegative coefficients of $c$ and certain two classes~$a_i$ and $a_j$  such that the coefficient of~$c$ is strictly positive. Thus, $C\in\CH$.

Vice versa, suppose that $C$ is an arbitrary multiset that belongs to~$\CH$ and contains~$A$. Let $M$ be a multicurve in~$\M$ with~$[M]=C$ and~$N\subset M$ be a multicurve with~$[N]=A$. Then $S_3\setminus N$ is a six-punctured sphere. Then the multicurve~$M$ can be obtained from~$N$ by consecutive adding new oriented simple curves each of which divides one of the existing components of the surface into two parts. An easy case analysis shows that, for any muticurve that can be obtained in this way, the multiset of its components is contained in one of the~$102$ listed multisets, provided that it is not allowed to create one-sided relations. 

Just to provide an example, let us consider in more detail the case of~$M$ containing two components~$\alpha_1$ and~$\alpha_1'$ in the homology class~$a_1$. Let $\alpha_2$ and~$\alpha_3$ be the components of~$M$ in the homology classes~$a_2$ and~$a_3$, respectively. Then $N=\alpha_1\cup\alpha_2\cup\alpha_3$ is a submulticurve of~$M$ with~$[N]=A$ and $S_3\setminus N$ is a $6$-punctured sphere. The curve $\alpha_1'$ divides this $6$-punctured sphere into two $4$-punctured spheres~$S'$ and~$S''$ so that $S'$ is adjacent to~$\alpha_2$ from both sides and to each of the curves~$\alpha_1$ and~$\alpha_1'$ from one of its sides, and $S''$ is adjacent to~$\alpha_3$ from both sides and to each of the curves~$\alpha_1$ and~$\alpha_1'$ from the other of its sides. Suppose that $M$ contains one more component~$\gamma$ that divides~$S'$ into two $3$-punctured spheres. Since $\gamma$ is non-separating in~$S_3$, we see that $[\gamma]$ is one of the four classes $\pm a_1\pm a_2$. However, the case $-a_1-a_2$ must be excluded, since one-sided relations are forbidden. So $[\gamma]$ is one of the three classes $a_1+a_2$, $a_1-a_2$, and~$a_2-a_1$. Similarly, if $M$ contains a component $\delta$ that lies in~$S''$, then $[\delta]$ is one of the three classes  $a_1+a_3$, $a_1-a_3$, and~$a_3-a_1$.  Thus, $[M]$ is contained in (at least) one of the $9$ six-element multisets obtained from~$[a_1,a_1,a_2,a_3]$ by adding one of the three classes $a_1+a_2$, $a_1-a_2$, and~$a_2-a_1$ and one of the three classes  $a_1+a_3$, $a_1-a_3$, and~$a_3-a_1$. All these $9$ multisets are in our list. Moreover, they are exactly all multisets in the list that contain~$[a_1,a_1,a_2,a_3]$.

Other cases are handled similarly. 
\end{proof}

\subsection{Spectral sequence}\label{section_CL}

Suppose that a group~$G$ acts on a CW-complex~$X$ cellularly and without rotations. Denote by~$G_P$ the stabilizer of a cell~$P$.

In this section we recall the construction and basic properties of the spectral sequence
\begin{equation}\label{eq_SpSeq}
E^1_{p,q}\cong \bigoplus_{P\in\mathfrak{X}_p}H_q(G_P;\Z)\quad\Longrightarrow \quad H_{p+q}(G;\Z).
\end{equation}
Here $\CX_p$ is a set of $p$-cells of~$X$ containing exactly one representative in every $G$-orbit. 

Spectral sequence~\eqref{eq_SpSeq} is exactly spectral sequence~(7.7) in~\cite[Section~VII.7]{Bro82}. In fact, there are two spectral sequences associated with the action of a group on a cell complex.  The spectral sequence~$E^r_{p,q}$ we are interested in is called the \textit{second spectral sequence} in~\cite{Bro82}.

By definition, the tensor product $M\otimes_GN$ of two left $G$-modules~$M$ and~$N$ is the quotient of the abelian group~$M\otimes_{\Z}N$ by all relations of the form $(ga)\otimes (gb)=a\otimes b$, where $g\in G$.

The  required spectral sequence~\eqref{eq_SpSeq} is the spectral sequence~$E^r_{p,q}$ associated with the filtration by columns of the double complex 
\begin{equation}\label{eq_K_double_complex}
B_{p,q}=C_p(X;\Z)\otimes_G \res_q,
\end{equation}
where $C_*(X;\Z)$ is the cellular chain complex of~$X$ and $\res_*$ is a projective resolution for~$\Z$ over~$\Z G$. We denote by~$\partial'$ and~$\partial''$ the differentials of the double complex $B_{*,*}$ induced by the differentials in~$C_*(X;\Z)$ and~$\res_*$, respectively. The bi-degrees of~$\partial'$ and~$\partial''$ are $(-1,0)$ and~$(0,-1)$, respectively. We use the following sign convention:
$$
\partial'(\sigma\otimes\xi)=(\partial\sigma)\otimes\xi,\qquad \partial''(\sigma\otimes\xi)=(-1)^{\dim\sigma}\sigma\otimes\partial\xi.
$$
The construction of the spectral sequence of a double complex can be found in~\cite[Section~VII.3]{Bro82} or \cite[Section~7.1]{Eve91}. This is a first-quadrant spectral sequence with differentials~$d^r$ of bi-degrees $(-r,r-1)$, respectively, such that  $E^0_{p,q}=B_{p,q}$, $d^0=\partial''$, and $d^1$ is induced by~$\partial'$. We will conveniently take for~$\res_*$ the standard bar resolution. Note that, starting from page~$E^1$, the spectral sequence is independent of the choice of the projective resolution. We need the following standard facts on this spectral sequence.

\begin{fact}\label{fact_E1}
There is a canonical isomorphism 
\begin{equation}\label{eq_E1pq}
E^1_{p,q}\cong\bigoplus_{P\in \CX_p}H_q(G_P;\Z),
\end{equation}
where $\CX_p$ is a set of $p$-cells of~$X$ containing exactly one representative in every $G$-orbit. Besides, the summands in decomposition~\eqref{eq_E1pq} do not depend on the choice of representatives. Namely, if we choose another set of representatives $\CX'_p=\{g_P\cdot P\}_{P\in\CX_p}$, where $g_{P}\in G$, then the two corresponding isomorphisms~\eqref{eq_E1pq} will differ by the direct sum of the isomorphisms $H_q(G_P;\Z)\to H_q(G_{g_{P}\cdot P};\Z)$ induced by the conjugations by~$g_P$.
\end{fact}

Suppose that $P$ is an oriented $p$-cell of~$X$ and $h\in G_P$. Consider a set of representatives~$\CX_p$ so that $\sigma\in\CX_p$, and let $[h]_P\in E^1_{p,1}$ be the element~$[h]$ in the summand~$H_1(G_P;\Z)$ in decomposition~\eqref{eq_E1pq}. Note that the element~$[h]_P$ changes sign whenever we reverse the orientation of~$P$. 
By Fact~\ref{fact_E1}, we have $[ghg^{-1}]_{g\cdot P}=[h]_P$ for all $g\in G$. 

\theoremstyle{theorem}
\newtheorem*{add} {Addendum to Fact~\ref{fact_E1}}

\begin{add}
For any $p$-cell~$P$ and any $h\in G_P$, the class $[h]_P\in E^1_{p,1}$ is represented by the element $P\otimes [h]\in B_{p,1}=C_p(X;\Z)\otimes_G \res_1$, provided that the bar resolution is taken for~$\res_*$.
\end{add}

\begin{fact}\label{fact_differential}
Suppose that elements 
$$Y_{p,q}\in B_{p,q},\ Y_{p-1,q+1}\in B_{p-1,q+1},\ \ldots, \ Y_{p-r+1,q+r-1}\in B_{p-r+1,q+r-1}$$ satisfy $\partial''Y_{p,q}=0$ and
$$\partial'Y_{p-i,q+i}=-\partial''Y_{p-i-1,q+i+1},\qquad i=0,\ldots,r-2.$$
Then $Y_{p,q}$ lies in the kernels of the differentials $d^0,\ldots,d^{r-1}$ and hence represents a class~$y\in E^r_{p,q}$. Besides, $d^ry$ is the class in~$E^r_{p-r,q+r-1}$ represented by~$\partial'Y_{p-r,q+r-1}$.
\end{fact}

By definition, the equivariant homology groups $H_*^G(X;\Z)$ are the homology group of the total complex for double complex~\eqref{eq_K_double_complex} i.\,e. the chain complex consisting of the abelian groups $\Tot_s(B_{*,*})=\bigoplus_{p+q=s}B_{p,q}$ with respect to the differential $\partial_{\Tot}=\partial'+\partial''$.

 \begin{fact}\label{fact_Einfty}
The group $\bigoplus_{p+q=s}E^{\infty}_{p,q}$ is the graded group associated with certain filtration 
$$
0=\CF_{-1,s+1}\subseteq \CF_{0,s}\subseteq \CF_{1,s-1}\subseteq \cdots \subseteq \CF_{s,0}=H_s^G(X;\Z)
$$
i.\,e. $E^{\infty}_{p,q}=\CF_{p,q}/\CF_{p-1,q+1}$. 
\end{fact}

The proofs of Facts~\ref{fact_E1} and~\ref{fact_Einfty} can be found in~\cite[Section~VII.7]{Bro82}, and~\eqref{eq_E1pq} is exactly isomorphism~(7.7) in~\cite{Bro82}. The description of differentials given in Fact~\ref{fact_differential} can be found in~\cite[Section~7.1]{Eve91}.

The most important for us is the case of a contractible complex~$X$. Then the groups $H_s^G(X;\Z)$ are naturally isomorphic to the groups~$H_s(G;\Z)$.

\begin{remark}
In~\cite{BBM07} spectral sequence~\eqref{eq_SpSeq} is called the \textit{Cartan--Leray spectral sequence} for the action of~$G$ on~$X$. However, this terminology is somewhat misleading because the name `Cartan--Leray specral sequence' is  usually reserved for another spectral sequence, namely, for the spectral sequence associated with the filtration by \textit{rows} of the double complex~\eqref{eq_K_double_complex} in the case of a free action.
\end{remark}

\subsection{Stabilizers of multicurves}\label{subsection_stabilizers}

In this section we prove the following fact on generators for stabilizers of multicurves in a genus~$3$ surface~$S_3$. Though this fact follows rather easily from known results, the author has not found it in the literature. 

\begin{propos}\label{propos_stab_generate}
Suppose that $M$ is a multicurve in~$S_3$ without separating components. Then the stabilizer~$\I_M$ is generated by twists about separating simple closed curves disjoint from~$M$ and twists about bounding pairs of simple closed curves disjoint from~$M$.
\end{propos}

\begin{remark}
Saying that a bounding pair~$\{\gamma_1,\gamma_2\}$ is disjoint from~$M$, we allow that one or both of the curves~$\gamma_1$ and~$\gamma_2$ are components of~$M$, since any component of~$M$ is isotopic to a curve disjoint from~$M$. 
\end{remark}

The main ingredient we need to prove Proposition~\ref{propos_stab_generate} is theory of Torelli groups for surfaces with several boundary components, which is due to Putman~\cite{Put07}. Recall some definitions and results in this paper. 

Let $\Sigma$ be a compact oriented surface with $m$ boundary components $C_1,\ldots,C_m$ and $\CP$ be a partition of the set of boundary components of~$\Sigma$. The pair $(\Sigma,\CP)$ is called a \textit{partitioned surface}. A \textit{capping} of~$(\Sigma,\CP)$ is an embedding $i:\Sigma\hookrightarrow S_g$ such that, for each connected component~$R$ of~$S_g\setminus\Int(\Sigma)$, the set of boundary components of~$R$ is exactly an element of~$\CP$. Putman proved that the subgroup $\I(\Sigma,\CP)=i_*^{-1}\bigl(\I(S_g)\bigr)$ is independent of the choice of a capping~$i$ of~$(\Sigma,\CP)$. He called this subgroup the \textit{Torelli group} of the partitioned surface~$(\Sigma,\CP)$.

A curve in~$\Sigma$ is called \textit{$\CP$-separating} if it is separates~$S_g$ for some (and then for any) capping of~$\Sigma$. Similarly a pair of disjoint curves in~$\Sigma$ is called a \textit{$\CP$-bounding pair} if it is a bounding pair in~$S_g$ for some (and then for any) capping of~$\Sigma$. (In these definitions we allow a separating curve in~$S_g$ to be  homotopically trivial and allow curves in a bounding pair in~$S_g$ to be isotopic to each other.)  A \textit{simply intersecting pair} of curves in~$\Sigma$ is a pair~$\{\gamma_1,\gamma_2\}$ of simple closed curves whose geometric intersection number is~$2$ and whose algebraic intersection number is~$0$. It is easy to see that twists~$T_{\gamma}$ about $\CP$-separating curves~$\gamma$, twists~$T_{\gamma,\gamma'}$ about $\CP$-bounding pairs~$\{\gamma,\gamma'\}$, and commutators~$[T_{\gamma_1},T_{\gamma_2}]$ of simply intersecting pairs~$\{\gamma_1,\gamma_2\}$ lie in the Torelli group~$\I(\Sigma,\CP)$.

\begin{theorem}[{Putman~\cite[Theorems~1.3 and~1.5]{Put07}}]\label{theorem_putman} 
\textnormal{(a)} For a partitioned surface~$(\Sigma,\CP)$ of positive genus, the group~$\I(\Sigma,\CP)$ is generated by twists about $\CP$-separating curves and twists about $\CP$-bounding pairs. 

\textnormal{(b)} For a partitioned surface~$(\Sigma,\CP)$ of zero genus, the group~$\I(\Sigma,\CP)$ is generated by twists about $\CP$-separating curves,  twists about $\CP$-bounding pairs, and commutators of simply intersecting pairs.
\end{theorem}

Now, suppose that $M$ is a multicurve (without separating components) in a closed surface~$S_g$. Then there is the following commutative diagram with exact rows:
\begin{equation}\label{eq_BLM}
\begin{tikzcd}
1 \ar{r} & BP(M)\ar[hook]{d} \ar{r} &  \I_M \ar{r}{\eta} \ar[hook]{d} & \PMod(S_g\setminus M) \ar[equals]{d} & \\
1 \ar{r} & G(M) \ar{r} & \Stab_{\Mod(S_g)}\bigl(\overrightarrow{M}\bigr) \ar{r}{\eta} & \PMod(S_g\setminus M) \ar{r} & 1
\end{tikzcd}
\end{equation}
Here $\Stab_{\Mod(S)}\bigl(\overrightarrow{M}\bigr)$ is the subgroup of $\Stab_{\Mod(S)}(M)$ consisting of all mapping classes that stabilize every component of~$M$ and preserve the orientation of every component of~$M$,  $G(M)$ is the group generated by Dehn twists about components of~$M$, and $BP(M)$ is the group generated by twists about bounding pairs contained in~$M$. The lower row of~\eqref{eq_BLM} is a version of the Birman--Lubotzky--McCarthy exact sequence (see~\cite[Lemma~2.1]{BLMC83}) and the upper row is its restriction to the Torelli group (see~\cite[Section~6.2]{BBM07}).

Let  $\Sigma$ be the compact surface with boundary obtained from~$S_g\setminus M$ by replacing punctures with boundary components.
To be precise, $\Sigma = S_g\setminus\mathcal{N}$, where $\mathcal{N}$ is a small open regular neighborhood of~$M$. The inclusion $i\colon \Sigma\hookrightarrow S_g$ induces the homomorphism
\begin{equation}\label{eq_iieta}
i_*\colon \Mod(\Sigma)\to \Stab_{\Mod(S)}\bigl(\overrightarrow{M}\bigr).
\end{equation}
This homomorphism is onto, since every mapping class in~$\Stab_{\Mod(S)}\bigl(\overrightarrow{M}\bigr)$ can be represented by a homeomorphism that fixes~$\mathcal{N}$ pointwise.

Let $\Sigma_1,\ldots,\Sigma_m$ be the connected components of~$\Sigma$. The capping $i_j\colon \Sigma_j\hookrightarrow S_g$ induces a partition~$\CP_j$ for each~$\Sigma_j$. Consider the corresponding Torelli groups $\I(\Sigma_j,\CP_j)$. Let $\I(\Sigma,\CP)\subset \Mod(\Sigma)$ be the direct product of the subgroups $\I(\Sigma_j,\CP_j)\subset \Mod(\Sigma_j)$. It follows immediately from the definition that the group $i_*\bigl(\I(\Sigma,\CP)\bigr)$ is contained in~$\I_M$. Nevertheless, it may happen that this inclusion is strict.

\begin{lem}\label{lem_torelli_difference}
Suppose that $g=3$. Then the group~$\I_M$ is generated by its subgroup $i_*\bigl(\I(\Sigma,\CP)\bigr)$, twists about separating simple closed curves  disjoint from~$M$, and twists  about bounding pairs disjoint from~$M$.
\end{lem}

\begin{proof}
Assume that all connected components of~$S_3\setminus M$ are homeomorphic to a three-punctured sphere. Then the group~$\PMod(S_3\setminus M)$ is trivial.  Hence,  $\I_M=BP(M)$, which implies  the assertion of the lemma. 

So we may assume that at least one connected component of~$S_3\setminus M$ is not homeomorphic to a three-punctured sphere. Hence at least one connected component of~$\Sigma$, say~$\Sigma_1$, is not a genus~$0$ surface with $3$ boundary components. We consider the following two cases:

\textsl{Case 1: $M$ does not contain a bounding pair.} Let us prove the following claim:
\begin{itemize}
\item[$(*)$] \textit{Any non-separating simple closed curve disjoint from~$M$ is homologous to a simple closed curve contained in~$\Sigma_1$.}
\end{itemize}

Recall that~$\Sigma_1$ is a compact connected subsurface of~$S_3$ and~$\Sigma_1$ is not a genus~$0$ surface with $3$ boundary components. Since $M$ contains neither a separating curve nor a bounding pair, we see that every connected component of $S_3\setminus \Sigma_1$ either is homeomorphic to an open cylinder (i.\,e. a two-punctured sphere) or is an oriented surface with at least $3$ punctures. If $m=1$, that is, $\Sigma_1$ is a unique connected component of~$\Sigma$, then Claim~$(*)$ is obvious. 
If $m\ge 2$, then that at least one component of~$S_3\setminus \Sigma_1$ is not a cylinder. By case analysis we see that there are only $3$ possibilities: (a) $\Sigma_1$ is a genus~$0$ surface with $4$ boundary components and $S_3\setminus \Sigma_1$ is a four-punctured sphere, (b) $\Sigma_1$ is a genus~$0$ surface with $5$ boundary components and $S_3\setminus \Sigma_1$ is a disjoint union of a three-punctured sphere and a cylinder, and (c) $\Sigma_1$ is a genus~$1$ surface with $3$ boundary components and $S_3\setminus \Sigma_1$ is a three-punctured sphere. In each of these three cases Claim~$(*)$ is obvious. 

Now, let us deduce the assertion of the lemma from this claim. The group~$\Mod(\Sigma)$ is generated by Dehn twists about simple closed curves contained in~$\Sigma$. Hence any element $h\in\Stab_{\Mod(S_g)}(\overrightarrow{M})$ can be written as a product of (left and right) Dehn twists about simple closed curves contained in~$\Sigma$. Further, if $\gamma$ is a non-separating simple closed curve that is contained in~$\Sigma_2\cup\cdots\cup\Sigma_m$, then by Claim~$(*)$ there exists  a simple closed curve~$\delta$ in~$\Sigma_1$ such that $\delta$ is homologous to~$\gamma$ in~$S_3$. Then  $T_{\gamma} = T_{\gamma,\delta}T_{\delta}$. Thus, $h$ can be written as a product of twists of the following three types:
\begin{enumerate}
\item left and right twists about simple closed curves contained in~$\Sigma_1$,
\item left and right twists about separating simple closed curves that are disjoint from~$M$,
\item twists about bounding pairs that are disjoint from~$M$.
\end{enumerate}
Moreover, we may swap any pair of twists of different types replacing one of them by an appropriate conjugate of it. Hence, we may achieve that $h=h'h''$, where $h'$ is a product of type~1 twists and $h''$ is a product of type~2 and type~3 twists. 

Now, suppose that $h\in \I_M$. To prove the assertion of the lemma, we need to show that $h' \in i_*\bigl(\I(\Sigma,\CP)\bigr)$. Obviously, $h''\in\I_M$; hence, $h'\in\I_M$. On the other hand, $h'$ is a product of twists about simple closed curves contained in~$\Sigma_1$ and therefore can be represented by a homeomorphism that fixes $S_3\setminus\Sigma_1$ pointwise. Thus, $h'=i_*(q)$ for certain $q\in\Mod(\Sigma_1)$. It follows immediately from Putman's definition that $q\in \I(\Sigma_1,\CP_1)$ and hence $h'\in i_*\bigl(\I(\Sigma,\CP)\bigr)$.

\textsl{Case 2: $M$ contains a bounding pair~$\{\gamma,\gamma'\}$.} In this case we will in fact prove that  $\I_M= i_*\bigl(\I(\Sigma,\CP)\bigr)$. Suppose that $h\in\I_M$. Since homomorphism~\eqref{eq_iieta} is onto, we see that $h=i_*(q_1\cdots q_m)$ for certain $q_j\in \Mod(\Sigma_j)$. Our aim is to prove that $h\in i_*\bigl(\I(\Sigma,\CP)\bigr)$.

The multicurve~$\gamma\cup\gamma'$ divides~$S_3$ into two pieces~$R_1$ and~$R_2$ each of which is homeomorphic to a two-punctured torus. We may assume that~$\Sigma_1\subseteq R_1$. Then $\Sigma_1$ is either a genus~$1$ surface with $2$~boundary components or a genus~$0$ surface with $4$~boundary components. In both cases no other component~$\Sigma_j$ lies in~$R_1$. We have two subcases:

\textsl{Subcase 2a: $m=2$.} Then $\Sigma_2$ is also either a genus~$1$ surface with $2$~boundary components or a genus~$0$ surface with $4$~boundary components. Choose a symplectic basis $a_1,a_2,a_3,b_1,b_2,b_3$ of~$H$ so that $[\gamma]=[\gamma'] = a_3$ and $a_j,b_j,a_3$ is a basis of the subgroup $H_1(R_j;\Z)\subset H$ for $j=1,2$. Put $h_1=i_*(q_1)$ and~$h_2=i_*(q_2)$. The mapping class~$h_1$ can be represented by a homeomorphism that fixes~$R_2$ pointwise. Hence $h_1$ stabilizes the homology classes $a_2$, $b_2$, and~$a_3$ and takes $b_3$ to a homology class of the form $b_3+k_1a_3+r_1a_1+s_1b_1$. Similarly, $h_2$ stabilizes the homology classes $a_1$, $b_1$, and~$a_3$ and takes $b_3$ to a homology class of the form $b_3+k_2a_3+r_2a_2+s_2b_2$. Since the mapping class $h=h_1h_2$ belongs to the Torelli group~$\I_3$, we see that $k_2=-k_1$ and $r_1=s_1=r_2=s_2=0$. Then the mapping classes $h_1T_{\gamma}^{k_1}$ and $T_{\gamma}^{-k_1}h_2$ also belong to~$\I_3$. Let $\gamma_1$ and~$\gamma_2$ be simple closed curves homotopic to~$\gamma$ such that $\gamma_1\subset\Sigma_1$ and $\gamma_2\subset \Sigma_2$. Then the mapping classes $q_1'=q_1T_{\gamma_1}^{k_1}$ and $q_2'=T_{\gamma_2}^{-k_1}q_2$ belong to the groups~$\I(\Sigma_1,\CP_1)$ and~$\I(\Sigma_2,\CP_2)$, respectively. Thus, the mapping class $h=i_*(q_1'q_2')$ lies in~$i_*\bigl(\I(\Sigma,\CP)\bigr)$.

\textsl{Subcase 2b: $m=3$.} Then both $\Sigma_2$ and~$\Sigma_3$ are genus~$0$ surfaces with $3$~boundary components. We may assume that $\Sigma_2$ is adjacent to~$\gamma$ and~$\Sigma_3$ is adjacent to~$\gamma'$.
Let $\delta_1$ and~$\delta_2$ be the components of~$M$ that separate~$\Sigma_2$ from~$\Sigma_3$. The mapping class group of a genus~$0$ compact surface with three boundary component is the rank~$3$ free abelian group generated by twists about simple closed curves homotopic to the boundary components. Hence $i_*(q_2q_3)=T_{\gamma}^nT_{\gamma'}^{n'}T_{\delta_1}^{l_1}T_{\delta_2}^{l_2}$ for certain integers~$n$, $n'$, $l_1$, and~$l_2$. Choose a symplectic basis $a_1,a_2,a_3,b_1,b_2,b_3$ as in the previous subcase, with the additional requirement that $[\delta_1]=a_2$ and~$[\delta_2]=a_2+a_3$. Again, the mapping class $h_1=i_*(q_1)$ stabilizes the homology classes $a_2$, $b_2$, and~$a_3$ and takes $b_3$ to a homology class of the form $b_3+ka_3+ra_1+sb_1$. The element $h_1^{-1}h=T_{\gamma}^nT_{\gamma'}^{n'}T_{\delta_1}^{l_1}T_{\delta_2}^{l_2}$ stabilizes the homology classes $a_1$, $b_1$, $a_2$, and~$a_3$, and takes $b_2$ and~$b_3$ to the homology classes $b_2-(l_1+l_2)a_2-l_2a_3$ and $b_3-l_2a_2-(n+n'+l_2)a_3$, respectively. Since $h\in\I_3$, it follows that $r=s=l_1=l_2=0$ and $k=n+n'$; hence $h=h_1T_{\gamma}^{n}T_{\gamma'}^{n'}$. If $\gamma_1$ and~$\gamma_1'$ are simple closed curves in~$\Sigma_1$ homotopic to~$\gamma$ and~$\gamma'$, respectively, then $h=i_*\bigl(q_1T_{\gamma_1}^{n}T_{\gamma'_1}^{n'}\bigr)$. Therefore,  $q_1T_{\gamma_1}^{n}T_{\gamma'_1}^{n'}\in \I(\Sigma_1,\CP_1)$ and hence $h\in i_*\bigl(\I(\Sigma,\CP)\bigr)$.
\end{proof}

Theorem~\ref{theorem_putman}  and Lemma~\ref{lem_torelli_difference} immediately imply that the stabilizer~$\I_M$ is generated by twists about separating simple closed curves, twists about bounding pairs, and commutators of simply intersecting pairs (with all curves disjoint from~$M$ in all the three cases). To complete the proof of Proposition~\ref{propos_stab_generate} we need to get rid of commutators of simply intersecting pairs. So, we suffice to prove the following lemma.

\begin{lem}
Suppose that $M$ is a multicurve in~$S_3$ without separating components, and $\{\gamma_1,\gamma_2\}$ is a simply intersecting pair disjoint from~$M$. Then the commutator~$[T_{\gamma_1},T_{\gamma_2}]$ can be represented as a product of twists about separating simple closed curves disjoint from~$M$ and twists about bounding pairs of simple closed curves disjoint from~$M$.
\end{lem}

\begin{proof}
Note that if there is a simple closed curve $\varepsilon$ such that $\varepsilon$ is disjoint from~$M\cup\gamma_1\cup\gamma_2$, and  $\varepsilon$ is homologous to one of the curves~$\gamma_1$ and~$\gamma_2$, then the assertion of the lemma will follow immediately, since~$[T_{\gamma_1},T_{\gamma_2}]$ will be a product of two twists about bounding pairs that are disjoint from~$M$. Say, if  $\varepsilon$ is homologous to~$\gamma_1$, then $[T_{\gamma_1},T_{\gamma_2}]=T_{\gamma_1,\varepsilon}^{-1}T_{T_{\gamma_2}^{-1}(\gamma_1),\varepsilon}$. So we may assume that there are no such curves~$\varepsilon$.

 Let $R$ be a closed regular neighborhood of~$\gamma_1\cup\gamma_2$. Then~$R$ is a genus~$0$ surface with four boundary components.  We denote these components by~$\alpha_0,\dots, \alpha_3$ so that $\gamma_1$ separates~$\alpha_0\cup\alpha_1$ from~$\alpha_2\cup\alpha_3$ and~$\gamma_2$ separates~$\alpha_0\cup\alpha_2$ from~$\alpha_1\cup\alpha_3$. If some~$\alpha_i$ were separating in~$S_3$, then one of the other three curves~$\alpha_j$ would be homologous to~$\gamma_1$, which is not the case by the assumption made. So none of the curves~$\alpha_i$ is separating. Since the curves~$\gamma_1$ and~$\gamma_2$ are also non-separating, we see that $\alpha_0,\dots, \alpha_3$ are pairwise non-homologous and $[\alpha_0]+\cdots+[\alpha_3]=0$ for certain choice of orientations of them. It follows easily that $S_3\setminus R$ is homeomorphic to a $4$-punctured sphere.

The curves $\alpha_0,\dots, \alpha_3$ may or may not be components of~$M$. Except for them, $M$ may contain at most one component~$\delta$ that divides~$S_3\setminus R$ into two parts each of which is homeomorphic to a $3$-punctured sphere. If $M$ did not contain such a component~$\delta$, then $S_3\setminus R$ would contain a simple closed curve~$\varepsilon$ that is disjoint from~$M$ and homologous to~$\gamma_1$, which is impossible. Similarly, if $\delta$ divided~$S_3\setminus R$ so that the puncture corresponding to~$\alpha_0$ was in the same part with one of the two punctures corresponding to~$\alpha_1$ and~$\alpha_2$, then $\delta$ itself would be homologous to either~$\gamma_1$ or~$\gamma_2$, which is again impossible. Hence $\delta$ divides~$S_3\setminus R$ so that the punctures corresponding to~$\alpha_0$ and~$\alpha_3$ are in the same part and the punctures corresponding to~$\alpha_1$ and~$\alpha_2$ are in the other part. Therefore $[\delta]=[\alpha_0]+[\alpha_3]$.  We may choose a simple closed curve~$\gamma_3$ in~$R$ so that the curves~$\alpha_0$, $\alpha_1$, $\alpha_2$, $\alpha_3$, $\gamma_1$, $\gamma_2$, and~$\gamma_3$ are arranged as the curves shown in Fig.~\ref{fig_lantern}. It is well known that whenever $7$ curves are arranged in such a way, the following \textit{lantern relation} holds in the mapping class group (see~\cite[Proposition~5.1]{FaMa12}):
\begin{equation}\label{eq_lantern}
T_{\gamma_1}T_{\gamma_2}T_{\gamma_3}=T_{\alpha_0}T_{\alpha_1}T_{\alpha_2}T_{\alpha_3}.
\end{equation}
Hence, we have
$$
[T_{\gamma_1},T_{\gamma_2}]=\bigl[T_{\gamma_2}^{-1}, T_{\gamma_1}T_{\gamma_2}\bigr]=\bigl[T_{\gamma_2}^{-1}, T_{\gamma_3}^{-1}\bigr]=
T_{T_{\gamma_2}(\gamma_3),\delta}T_{\gamma_3,\delta}^{-1},
$$
which completes the proof of the lemma.
\end{proof}

\begin{figure}
\tikzset{->-/.style={decoration={
     markings,
     mark=at position #1 with {\arrow{>}}},postaction={decorate}}}

  %  \draw[->-=.5] (0,0) to [bend left] (2,4);
    %\draw[->-=.8] (0,0) to [bend right] (2,4);

\definecolor{mygreen}{rgb}{0, 0.4, 0}
\definecolor{myblue}{rgb}{0, 0, 0.7}
\definecolor{mypink}{rgb}{1, 0.3, 0.8}
\definecolor{myyellow}{rgb}{1,1, 0}

\colorlet{cola}{violet}
\colorlet{cold}{red}
\colorlet{colg}{myblue}
\colorlet{colkf}{mypink}
\colorlet{colks}{violet}
\colorlet{coll}{orange}
\colorlet{colx}{mygreen}
\colorlet{coly}{orange}
\colorlet{colm}{mygreen}
\colorlet{colz}{brown}

\begin{tikzpicture}
\small
\tikzset{every path/.append style={line width=.3mm}}

\filldraw [fill=black!4, very thick] (0,0) circle (1.9); 
%\fill [color=] (0,0) circle (2);
\filldraw [fill=white, very thick] (0,1) circle (.3);
\filldraw [fill=white, very thick] (0.866,-0.5) circle (.3);
\filldraw [fill=white, very thick] (-0.866,-0.5) circle (.3);

\draw [thick] (0.1,-.5) ellipse (1.4 and .7);
\draw [thick, rotate=120] (0.1,-.5) ellipse (1.4 and .7);
\draw [thick, rotate=240] (0.1,-.5) ellipse (1.4 and .7);
\path (0.03,1) node {$\alpha_3$};
\path (-0.836,-0.5) node {$\alpha_2$};
\path (0.896,-0.5) node {$\alpha_1$};
\path (0.13,-1.45) node {$\gamma_3$};
\path (1.27,0.7) node {$\gamma_2$};
\path (-1.27,0.7) node {$\gamma_1$};
\path (1.56,-1.56) node {$\alpha_0$};

\end{tikzpicture}
\caption{Lantern relation}\label{fig_lantern}
\end{figure}
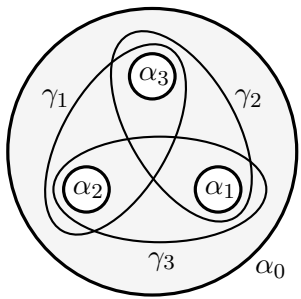

\subsection{Birman--Craggs  homomorphisms}\label{subsection_BC}

A map $\omega\colon H_{\Z/2}\to\Z/2$ is called an \textit{$\Sp$-quadratic function} if, for all $\bx,\by\in H_{\Z/2}$,
\begin{equation}\label{eq_Sp}
\omega(\bx+\by)=\omega(\bx)+\omega(\by)+\bx\cdot \by.
\end{equation}
By definition, the \textit{Arf invariant} of~$\omega$, $\Arf(\omega)\in\Z/2$, is given by
$$
\Arf(\omega)=\sum_{i=1}^g\omega(\ba_i)\omega(\bb_i),
$$
where $\ba_1,\ldots,\ba_g,\bb_1,\ldots,\bb_g$ is a symplectic basis of~$H_{\Z/2}$. It is an easy well-known fact that $\Arf(\omega)$ is independent of the choice of a symplectic basis.

To each $\Sp$-quadratic form~$\omega$ with $\Arf(\omega)=0$ is assigned a \textit{Birman--Craggs homomorphism} $\rho_{\omega}\colon\I_g\to\Z/2$. Initially homomorphisms of~$\I_g$ to~$\Z/2$ were constructed by Birman and Craggs in~\cite{BiCr78} using the Rokhlin invariant of three-dimensional homology spheres. Johnson~\cite{Joh80a} proved that these homomorphisms are indexed by $\Sp$-quadratic forms with zero Arf invariant, and wrote explicit formulae for their values on generators of~$\I_g$. These formulae are as follows.

First, suppose that $\gamma$ is a separating simple closed curve on~$S_g$. Let $S'$ and~$S''$ be the connected components of~$S_g\setminus \gamma$. Consider the splitting $H_{\Z/2}=H_1(S';\Z/2)\oplus H_1(S'';\Z/2)$. Then 
\begin{equation}\label{eq_BC1}
\rho_{\omega}(T_{\gamma})=\sum_{i=1}^{g'}\omega(\ba_i)\omega(\bb_i),
\end{equation}
where $g'$ is the genus of~$S'$ and $\ba_1,\ldots,\ba_{g'},\bb_1,\ldots,\bb_{g'}$ is a symplectic basis of~$H_1(S';\Z/2)$. Since $\Arf(\omega)=0$, the value in the right-hand side of~\eqref{eq_BC1} is independent of which of the two components of~$S_g\setminus \gamma$ is taken for~$S'$.

Second, suppose that $\{\gamma,\delta\}$ is a bounding pair in~$S_g$. Let $\bc$ be the modulo~$2$ homology class of~$\gamma$ and~$\delta$. Again let $S'$ and~$S''$ be the connected components of~$S_g\setminus \gamma$; then 
\begin{align*}
H_1(S';\Z/2)\cap H_1(S'';\Z/2)&=\langle\bc\rangle,\\
H_1(S';\Z/2)+H_1(S'';\Z/2)&=\langle\bc\rangle^{\bot}.
\end{align*}
Obviously, $S'$ is a surface of some genus $g'>0$ with two punctures. Consider a subsurface $\Sigma\subset S'$ homeomorphic to a genus~$g'$ surface with one puncture; then 
$$H_1(S';\Z/2)=H_1(\Sigma;\Z/2)\oplus\langle \bc\rangle.$$
Let $\ba_1,\ldots,\ba_{g'},\bb_1,\ldots,\bb_{g'}$ be a symplectic basis of~$H_1(\Sigma;\Z/2)$. Then
\begin{equation}\label{eq_BC2}
\rho_{\omega}\left(T_{\gamma,\delta}\right)=\bigl(\omega(\bc)+1\bigr)\sum_{i=1}^{g'}\omega(\ba_i)\omega(\bb_i).
\end{equation}
Again it is not hard to check that the value in the right-hand side is independent of all choices made.

Johnson~\cite{Joh80a} also showed that the Birman--Craggs homomorphisms $\rho_{\omega}$ can be combined into one homomorphism in the following way. Let $\Omega$ be the affine space over~$\Z/2$ consisting of all $\Sp$-quadratic forms and $\Omega_0$ be the codimension~$1$ subspace of~$\Omega$ consisting of all $\Sp$-quadratic forms with zero Arf invariant. Each homology class $\bx\in H_{\Z/2}$ determines an affine linear function $\overline{\bx}\colon \Omega\to\Z/2$ by $\overline{\bx}(\omega)=\omega(\bx)$. Formula~\eqref{eq_Sp} reads as 
$$
\overline{\bx+\by}=\overline{\bx}+\overline{\by}+\bx\cdot\by.
$$
Now, let $\BB$ be the algebra of Boolean polynomials on~$\Omega$, that is, the quotient of the algebra of polynomials on~$\Omega$ with coefficients in~$\Z/2$ by all relations~$q^2=q$. Note that, for each basis $\bc_1,\ldots,\bc_{2g}$ of~$H$, $\Omega$ is the free Boolean algebra in variables $\overline{\bc}_1,\ldots,\overline{\bc}_{2g}$.
Further, let $\BB_k\subset\BB$ be the subspace consisting of all Boolean polynomials of degrees not exceeding~$k$. The Arf invariant is a quadratic Boolean polynomial; for each symplectic basis $\ba_1,\ldots,\ba_{g},\bb_1,\ldots,\bb_{g}$, it is given by
$\Arf = \sum_{i=1}^g\overline{\ba}_i\overline{\bb}_i.$ So $\BB'=\BB/(\Arf)$ is the algebra of Boolean polynomials on~$\Omega_0$. Similarly to~$\BB_k$, we denote by $\BB_k'$ the subset of~$\BB'$ consisting of all cosets by the ideal~$(\Arf)$ that contain representatives of degree $\le k$. 

Now, the homomorphisms~$\rho_{\omega}$ can be combined into one $\Sp(2g,\Z)$-equivariant homomorphism $\sigma\colon \I_g\to\BB_3'$, which is called the \textit{Birman--Craggs--Johnson homomorphism}, so that $\sigma(h)(\omega)=\rho_{\omega}(h)$ for all $h\in\I_g$ and all $\omega\in\Omega_0$. Formulae~\eqref{eq_BC1} and~\eqref{eq_BC2} read as 
\begin{align}
\sigma(T_{\gamma})&=\sum_{i=1}^{g'}\overline{\ba}_i\overline{\bb}_i&&\text{if $\gamma$ is separating,}\label{eq_BCJ_sep}\\
\sigma\left(T_{\gamma,\delta}\right)&=\bigl(\overline{\bc}+1\bigr)\sum_{i=1}^{g'}\overline{\ba}_i\overline{\bb}_i&& \text{if $\{\gamma,\delta\}$ is a bounding pair.}\label{eq_BCJ_BP}
\end{align}
It follows easily that the homomorphism $\sigma\colon\I_g\to\BB_3'$ is surjective. Also it is easy to compute that $\dim\BB_3'=\binom{2g}{3}+\binom{2g}{2}=g(4g^2-1)/3$. 

Let $\CC_g$ be the kernel of~$\sigma$; then $\CC_g\triangleleft\I_g$ is a normal subgroup of index~$2^{g(4g^2-1)/3}$. In particular, $\CC_3\triangleleft\I_3$ is a normal subgroup of index~$2^{35}$.

In~\cite{Gai17}, the author showed that the Birman--Craggs--Johnson homomorphism $\sigma\colon\I_g\to \BB_3'$ can be extended to a $\Mod(S_g)$-equivariant homomorphism $\hsigma\colon\hI_g\to \BB_4'$. This extension has the following property. 

If we arrange~$S_g$ as the surface shown in Fig.~\ref{fig_eBC} and consider the corresponding hyperelliptic involution~$\iota$, i.\,e., the rotation of~$S_g$ by angle~$\pi$ around the horizontal axis, then
 \begin{equation}\label{eq_eBC}
 \hsigma(\iota)=\sum_{1\le i<j\le g}\overline{\ba}_i\overline{\bb}_i\bigl(\overline{\ba}_j+1\bigr)\overline{\bb}_j,
 \end{equation}
where $\ba_1,\ldots,\ba_g,\bb_1,\ldots,\bb_g$ are the homology classes of the curves $\alpha_1,\ldots,\alpha_g,\beta_1,\ldots,\beta_g$  in Fig.~\ref{fig_eBC}, respectively.

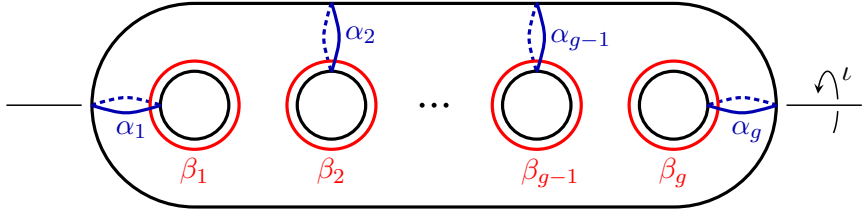
\begin{figure}
\begin{tikzpicture}[scale=.45]
\small
\definecolor{myblue}{rgb}{0, 0, 0.7}
\definecolor{mygreen}{rgb}{0, 0.4, 0}
\tikzset{every path/.append style={line width=.2mm}}
\tikzset{my dash/.style={dash pattern=on 2pt off 1.5pt}}

\begin{scope}
% Invisible lines

%\draw[color=myblue,my dash, very thick] (0,-3.03) .. controls (-0.3,-2) .. (0,-.97); % e_0
\draw[color=myblue,my dash, very thick] (-3,3.03) .. controls (-3.3,2) .. (-3,.97); % e_0'
\draw[color=myblue,my dash, very thick] (3,3.03) .. controls (2.7,2) .. (3,.97); % e_0'
\draw[color=myblue,my dash, very thick] (-10.03,0) .. controls (-9,0.3) .. (-7.97,0); %e_1
\draw[color=myblue,my dash, very thick] (10.03,0) .. controls (9,0.3) .. (7.97,0); %e_2
%\draw[color=myblue,my dash, very thick] (-3.03,0) .. controls (-2,0.3) .. (-.97,0); % e_0+e_1
%\draw[color=red,my dash ] (3.03,0) .. controls (2,0.3) .. (.97,0); % e_0+e_2
%\draw[color=red, my dash] (2, -3.03) .. controls (1.7,-2) and (1.7,2) .. (2, 3.03);
%\draw[color=myblue, very thick,my dash] (-3.293,-.707) .. controls (-1.5,-2.2) and (1.5,-2.2) .. (3.293,-.707);
%\draw[color=red,my dash] (-3.293,.707) .. controls (-1.5,2.8) and (1.5,2.8) .. (3.293,.707);

% Border lines

\draw [fill] (0,0) circle (.05);
\draw [fill] (0.35,0) circle (.05);
\draw [fill] (-0.35,0) circle (.05);

\draw [very thick] (3,0) circle (1);
\draw [very thick] (-3,0) circle (1);
\draw [very thick] (7,0) circle (1);
\draw [very thick] (-7,0) circle (1);
\draw [very thick]  (-7,3)--(7,3) arc (90:-90:3) -- (-7,-3) arc (270:90:3);

\draw [] (-12.5,0) -- (-10.3,0);
\draw [] (12.5,0) -- (10.3,0);

\draw (11.5,0) + (-55: .3 and 1) arc (-55: -10: .3 and 1);
\draw [-stealth] (11.5,0) + (10: .3 and 1) arc (10: 160: .3 and 1) node [pos=.4, right] {$\iota$};

% Visible lines 

\draw [color=red, very thick] (-7,0) + (1.3,0) arc (0:360:1.3) node [pos=.75, below=-1pt] {$\beta_1$};

\draw [color=red, very thick] (-3,0) + (1.3,0) arc (0:360:1.3) node [pos=.75, below=-1pt] {$\beta_2$};

\draw [color=red, very thick] (3,0) + (1.3,0) arc (0:360:1.3) node [pos=.8, below=0pt] {$\beta_{g-1}$};

\draw [color=red, very thick] (7,0) + (1.3,0) arc (0:360:1.3) node [pos=.75, below=-1pt] {$\beta_g$};

\draw[color=myblue, very thick] (3,.97) .. controls (3.3,2) .. (3,3.03) node [pos=0.45, right=-2pt] {$\alpha_{g-1}$};  %e_0
\draw[color=myblue, very thick] (-3,3.03) .. controls (-2.7,2) .. (-3,.97) node [pos=0.45, right=-2pt]{$\alpha_2$};  %e_0'

\draw[color=myblue, very thick] (-7.97,0) .. controls (-9,-0.3) .. (-10.03,0) node [pos=.4,below=-1pt] {$\alpha_1$};%e_1
%\draw[color=red,->-=.55] (4.97,0) .. controls (6,-0.4) .. (7.03,0) node [pos=.3,below=-1pt] {$\alpha_2$}; %e_2
%\draw[color=red,->-=.55] (-7.03,0) .. controls (-6,-0.3) .. (-4.97,0)node [pos=0.4,below] {$\alpha_1$};%e_1 -
\draw[color=myblue, very thick] (10.03,0) .. controls (9,-0.3) .. (7.97,0) node [pos=.4,below=-1pt] {$\alpha_g$};%e_2 -

%\draw[color=myblue, very thick] (-3.03,0) .. controls (-2,-0.3) .. (-.97,0);%e_0+e_1
%\draw[color=red] (3.03,0) node [right=-2pt] {$\delta$} .. controls (2,-0.3) .. (.97,0) ;%e_0+e_2

%\draw[color=red] (2, -3.03) .. controls (2.3,-2) and (2.3,2) .. (2, 3.03) node [pos=0.35, right=-1pt] {$\delta$};  %e_0

%\draw[color=red] (-2,-3.03) .. controls (-1.7,-1.5) and (-1.7,1.5) .. (-2,3.03) node [pos=.25,right=-2pt] {$u$};
%\draw[color=red,my dash] (-2,-3.03) .. controls (-2.3,-1.5) and (-2.3,1.5) .. (-2,3.03);

%\draw[color=red] (-3.293,-.707) .. controls (-1,-2.8) and (1,-2.8) .. (3.293,-.707) node [pos=.8,below=1pt] {$\delta_2$};
%\draw[color=red] (-3.293,.707) .. controls (-1,2.2) and (1,2.2) .. (3.293,.707)node [pos=.8,below] {$\delta$};
\end{scope}

\end{tikzpicture}
\caption{Curves $\alpha_1,\ldots,\alpha_g,\beta_1,\ldots,\beta_g$ and hyperelliptic involution~$\iota$}\label{fig_eBC}
\end{figure}

We shall refer to~$\hsigma$ as to the \textit{extended Birman--Craggs--Johnson homomorphism}. Now, to each $\Sp$-quadratic form~$\omega$ on~$H_{\Z/2}$ with zero Arf invariant is assigned an \textit{extended Birman--Craggs homomorphism} $\hrho_{\omega}\colon\hI_g\to\Z/2$ given by $\hrho_{\omega}(h)=\hsigma(h)(\omega)$.

\begin{propos}\label{propos_eBC}
Suppose that $g\ge 3$. Then the kernel of~$\hsigma$ coincides with $\ker\sigma=\CC_g$.
\end{propos}

\begin{proof}
Since $\I_g$ is an index~$2$ subgroup of~$\hI_g$, we just need to show that the Boolean polynomial in the right-hand side of~\eqref{eq_eBC} does not lie in~$\BB_3'$. Equivalently, we need to prove that in~$\BB$ there are no relations of the form
\begin{equation}\label{eq_Bool_rel}
\sum_{1\le i<j\le g}\overline{\ba}_i\overline{\bb}_i\overline{\ba}_j\overline{\bb}_j=q\sum_{i=1}^g\overline{\ba}_i\overline{\bb}_i+r,\qquad q\in\BB,\, r\in\BB_3.
\end{equation}

Since the variables~$\overline{\ba}_1,\ldots,\overline{\ba}_g,\overline{\bb}_1,\ldots,\overline{\bb}_g$ are  independent, it is sufficient to prove that there are no relations of form~\eqref{eq_Bool_rel} in case $g=3$. Indeed, if such relation existed for some $g>3$, then, substituting zeroes for all variables~$\overline{\ba}_i$ and~$\overline{\bb}_i$ with $i>3$, we would obtain a relation of form~\eqref{eq_Bool_rel} for $g=3$.

If a relation of form~\eqref{eq_Bool_rel} existed for $g=3$, then, multiplying both sides of it by $\sum_{i=1}^3\overline{\ba}_i\overline{\bb}_i+1$, we would obtain that
\begin{equation*}
\overline{\ba}_1\overline{\bb}_1\overline{\ba}_2\overline{\bb}_2\overline{\ba}_3\overline{\bb}_3+\sum_{1\le i<j\le 3}\overline{\ba}_i\overline{\bb}_i\overline{\ba}_j\overline{\bb}_j=\left(\sum_{i=1}^3\overline{\ba}_i\overline{\bb}_i+1\right)r,
\end{equation*}
which is impossible, since a product of~$6$ independent variables cannot be written as a Boolean polynomial of degree less than~$6$.
\end{proof}

Note that in case $g=3$, which is the only important for us in this paper, formula~\eqref{eq_eBC} reads as
\begin{equation}\label{eq_eBC3}
\hsigma(\iota)=\overline{\ba}_1\overline{\bb}_1\bigl(\overline{\ba}_2+1\bigr)\overline{\bb}_2.
\end{equation}
Besides, in this case the algebra~$\BB'$ is additively generated by~$\BB_3'$ and the monomial $\overline{\ba}_1\overline{\bb}_1\overline{\ba}_2\overline{\bb}_2$. Thus, $\BB'=\BB'_4$ and the homomorphism $\hsigma\colon\hI_3\to\BB'$ is surjective.

\subsection{Fundamental exact sequence.}\label{subsection_fundamental}
 Suppose that 
\begin{equation}\label{eq_short}
1\to N\xrightarrow{i} G\xrightarrow{j} Q\to 1
\end{equation}
is a short exact sequence of groups. Then there is the following $5$-term exact sequence of cohomology groups, see~\cite[Corollary~7.2.3]{Eve91}:
\begin{equation}\label{eq_fund}
0\to H^1(Q;\Z/2)\xrightarrow{j^*} H^1(G;\Z/2)\xrightarrow{i^*}H^1(N;\Z/2)^Q\xrightarrow{d_2} H^2(Q;\Z/2)\xrightarrow{j^*} H^2(G;\Z/2).
\end{equation}  
Here $H^1(N;\Z/2)^Q$ is the $Q$-invariant part of~$H^1(N;\Z/2)$.
The homomorphism~$d_2$ is called the \textit{transgression}. It is denoted by~$d_2$, since in fact it is the differential of the Lyndon--Hochschild--Serre spectral sequence.

We will deal with exact sequence~\eqref{eq_fund} in the following special case. Suppose that $Q=H_1(G;\Z/2)$ and hence $N$ is the kernel of the natural surjective homomorphism $G\to  H_1(G;\Z/2)$. (Nevertheless, we conveniently use multiplicative notation for~$Q$.) It is a standard fact that~$N$ is generated by commutators of elements of~$G$ and squares of elements of~$G$. Obviously, in this case the pullback homomorphism $$j^*\colon H^1(Q;\Z/2)\to H^1(G;\Z/2)$$ is an isomorphism, so the exact sequence reads as
\begin{equation}\label{eq_fund2}
0\to H^1(N;\Z/2)^Q\xrightarrow{d_2} H^2(Q;\Z/2)\xrightarrow{j^*} H^2(G;\Z/2).
\end{equation} 

Suppose that $\kappa_1,\ldots,\kappa_m$ and~$\lambda_1,\ldots,\lambda_m$ are  homomorphisms of~$G$ to~$\Z/2$ satisfying 
$$
\sum_{s=1}^m\kappa_s\lambda_s=0
$$ 
in $H^2(G;\Z/2)$, where multiplication is the cup-product in cohomology. Let $k_1,\ldots,k_m$ and $l_1,\ldots,l_m$ be the corresponding homomorphisms of~$Q$ to~$\Z/2$ so that $\kappa_s=k_s\circ j$ and $\lambda_s=l_s\circ j$, $s=1,\ldots,m$.  Consider the cohomology class 
$$
\varphi=\sum_{s=1}^mk_sl_s\in H^2(Q;\Z/2).
$$
Then $j^*\varphi=0$. Since sequence~\eqref{eq_fund2} is exact, we obtain that there is a unique $Q$-invariant homomorphism $\psi\colon N\to\Z/2$ such that $d_2\psi=\varphi$.  

\begin{propos}\label{propos_d2_LHS}
For all $g_1,g_2\in G$,
\begin{equation}\label{eq_d2_gen1}
\psi\bigl([g_1,g_2]\bigr)=\sum_{s=1}^m\bigl(\kappa_s(g_1)\lambda_s(g_2)+\kappa_s(g_2)\lambda_s(g_1)\bigr),
\end{equation}
and for all $g\in G$,
\begin{equation}\label{eq_d2_gen2}
\psi\bigl(g^2\bigr)=\sum_{s=1}^m \kappa_s(g)\lambda_s(g).
\end{equation}
\end{propos}

\begin{proof}
First, recall that the transgression~$d_2$ has the following description, see~\cite[Theorem~7.3.1]{Eve91}. Let $N'$ be the commutator subgroup of~$N$. Exact sequence~\eqref{eq_short} gives rise to an action of~$Q$ on~$N/N'$. Recall that cohomology classes in~$H^2(Q;N/N')$ are in one-to-one correspondence with extensions of~$Q$ by~$N/N'$ giving rise to the same action of~$Q$ on~$N/N'$. Let $\varepsilon\in H^2(Q;N/N')$ be the class of the extension 
\begin{equation*}
1\to N/N'\xrightarrow{\tilde\imath}G/N'\xrightarrow{\tilde\jmath} Q\to 1,
\end{equation*}
where $\tilde\imath$ and~$\tilde\jmath$ are induced by~$i$ and~$j$, respectively.
Then $d_2\psi$ is the image of~$\varepsilon$ under the homomorphism 
$$
\psi_*\colon H^2(Q;N/N')\to H^2(Q;\Z/2)
$$
induced by~$\psi \in\Hom(N,\Z/2)=\Hom(N/N',\Z/2)$.

Second, recall that if $\tilde{f}\colon Q\to G/N'$ is a set-theoretic section of~$\tilde\jmath$, then the function $e\colon Q\times Q\to N/N'$ given by
$$
\tilde\imath\bigl(e(q_1,q_2)\bigr)=\tilde{f}(q_1)\tilde{f}(q_2)\tilde{f}(q_1q_2)^{-1}
$$
is a cocycle representing~$\varepsilon$ in the bar resolution, see~\cite[Section~2.3]{Eve91} and~\cite[Section~IV.3]{Bro82}. Hence, if $f\colon Q\to G$ is a set-theoretic section of~$j$, then the function $w\colon Q\times Q\to\Z/2$ given by
\begin{equation*}
w(q_1,q_2)=\psi\left(f(q_1)f(q_2)f(q_1q_2)^{-1}\right)
\end{equation*}
is a cocycle representing the class $d_2\psi=\varphi$. On the other hand, by the definition of the cup-product (see~\cite[Section~V.3]{Bro82}), the class~$\varphi$ can be also represented by the cocycle 
$$
v(q_1,q_2)=\sum_{s=1}^m k_s(q_1)l_s(q_2).
$$

Cocycles~$v$ and~$w$ are homologous and hence take equal to each other values on every cycle, i.\,e. every element in the kernel of the differential of the complex~$(\Z/2)\otimes_{Q}R_*$, where $R_*$ is the bar resolution for~$Q$. Since $Q$ is an abelian group consisting of elements of order~$2$, all chains  $[q_1|q_2]+[q_2|q_1]$ and $[q|q]+[1|1]$ lie in the kernel of~$\partial$. Putting $g_1=f(q_1)$, $g_2=f(q_2)$, and~$g=f(q)$ and equating with each other the values of~$w$ and~$v$ on the cycles $[q_1|q_2]+[q_2|q_1]$ and $[q|q]+[1|1]$, we obtain that
\begin{multline*}
\psi\bigl([g_1,g_2]\bigr)=\psi\bigl(g_1^{-1}g_2^{-1}g_1g_2\bigr)=\psi\left(g_1g_2f(q_1q_2)^{-1}\right)+\psi\left(g_2g_1f(q_1q_2)^{-1}\right)\\
{}=w(q_1,q_2)+w(q_2,q_1)=v(q_1,q_2)+v(q_2,q_1)=\sum_{s=1}^m \kappa_s(g_1)\lambda_s(g_2)
\end{multline*}
and
\begin{multline*}
\psi\bigl(g^2\bigr)=\psi\bigl(g^2s(1)^{-1}\bigr)+\psi\bigl(s(1)^{-1}\bigr)=w(q,q)+w(1,1)\\
{}=v(q,q)+v(1,1)=\sum_{s=1}^m\kappa_s(g)\lambda_s(g),
\end{multline*}
which are exactly the required formulae~\eqref{eq_d2_gen1} and~\eqref{eq_d2_gen2} for those particular~$g_1$, $g_2$, and~$g$.

Now, for each~$g\in G$, we could choose~$f$ so that $f(j(g))=g$. Hence, equality~\eqref{eq_d2_gen2} holds for all $g\in G$. Similarly, for any~$g_1,g_2\in G$ satisfying $j(g_1)\ne j(g_2)$, we could choose~$f$ so that $f(j(g_1))=g_1$ and~$f(j(g_2))=g_2$. Hence equality~\eqref{eq_d2_gen1} holds for all $g_1,g_2\in G$ such that $j(g_1)\ne j(g_2)$. If $j(g_1)=j(g_2)\ne 1$, then equality~\eqref{eq_d2_gen1} for the pair~$\{g_1,g_2\}$ follows immediately from  equality~\eqref{eq_d2_gen1} for the pair~$\{g_1,g_1g_2\}$. Finally, if $j(g_1)=j(g_2)= 1$, i.\,e. $g_1,g_2\in N$, then both sides of~\eqref{eq_d2_gen1} obviously vanish.
\end{proof}

\section{Cohomology classes $\theta_{\fA}$ and homomorphisms~$\btheta_A$}\label{section_theta}

Starting from this point, we always consider  a genus~$3$ surface~$S=S_3$ and the Torelli group~$\I=\I_3$ of it, and conveniently omit the index~$3$ in notation. (Similarly for $\CC=\CC_3$, $\hI=\hI_3$, $\B=\B_3$). Let $x\in H$ be the chosen primitive homology class that was used in the construction of the complex of cycles. We fix it throughout the whole paper.

Suppose that $\fA=\{\ba_1,\ba_2,\ba_3\}$ is a three-element subset of~$H_{\Z/2}$ that satisfies the following condition:
\begin{itemize}
\item[$(*)$] $\ba_1,\ba_2,\ba_3$ are linearly independent over~$\Z/2$ and have zero pairwise intersection numbers. 
\end{itemize}
It is easy to see that there are exactly four $\Sp$-quadratic functions~$\omega$ satisfying $\Arf(\omega)=0$ and 
$$\omega(\ba_1)=\omega(\ba_2)=\omega(\ba_3)=1.$$ 
If we extend the set~$\fA$ to a symplectic basis $\ba_1,\ba_2,\ba_3,\bb_1,\bb_2,\bb_3$ of~$H_{\Z/2}$, then these four $\Sp$-quadratic functions~$\omega_0,\ldots,\omega_3$ are given by
\begin{equation}\label{eq_omega_numeration}
\omega_i(\bb_j)=\left\{
\begin{aligned}
&0,&&\text{if $i=0$ or $i=j$,}\\
&1,&&\text{if $i\ne 0$ and~$i\ne j$.}
\end{aligned}
\right.
\end{equation} 
Though the numeration of~$\omega_i$'s depends on the numeration of the elements of~$\fA$ and on the choice of~$\bb_1,\bb_2,\bb_3$, the set of these four functions depend only on~$\fA$.

We consider the Birman--Craggs homomorphism $\rho_i=\rho_{\omega_i}$, $i=0,1,2,3$, as elements of~$H^1(\I;\Z/2)$, and define a cohomology class $\theta_{\fA}\in H^2(\I;\Z/2)$ by
\begin{equation*}
\theta_{\fA}=\sum_{0\le i<j\le 3}\rho_i\rho_j,
\end{equation*}
where multiplication is the cup-product. In the sequel, we will write many formulae involving classes~$\rho_i$. We agree that their indices always run from~$0$ to~$3$, and usually do not indicate this explicitly.

Since the multiplication in $H^*(\I;\Z/2)$ is commutative, we see that the class~$\theta_{\fA}$ is independent of the chosen numeration of $\omega_i$'s. Hence, $\theta_{\fA}$ is independent of the numeration of elements of~$\fA$ and of the choice of $\bb_1$, $\bb_2$, and~$\bb_3$. Besides, since $\omega_i(\ba_1+\ba_2+\ba_3)=1$, $i=0,1,2,3$, we have
\begin{equation}\label{eq_theta_equal}
\theta_{\{\ba_1+\ba_2+\ba_3,\ba_1,\ba_2\}}=
\theta_{\{\ba_1+\ba_2+\ba_3,\ba_2,\ba_3\}}=
\theta_{\{\ba_1+\ba_2+\ba_3,\ba_3,\ba_1\}}=
\theta_{\{\ba_1,\ba_2,\ba_3\}}.
\end{equation}

\begin{propos}\label{propos_theta_main}
\textnormal{(a)} Suppose that $N$ is a three-component multicurve whose components have modulo~$2$ homology classes~$\ba_1$, $\ba_2$, and~$\ba_3$. Let $\delta_1$ and~$\delta_2$ be two non-isotopic separating simple  closed curves that are disjoint from~$N$ and from each other, see Fig.~\ref{fig_2separate}. Then $\langle \theta_{\fA},\CA(T_{\delta_1},T_{\delta_2})\rangle=1$.

\begin{figure}
\begin{tikzpicture}[scale=.45]
\small
\definecolor{myblue}{rgb}{0, 0, 0.7}
\definecolor{mygreen}{rgb}{0, 0.4, 0}
\tikzset{every path/.append style={line width=.2mm}}
\tikzset{my dash/.style={dash pattern=on 2pt off 1.5pt}}

\begin{scope}
% Invisible lines

\draw[color=myblue,my dash, very thick] (0,-3.03) .. controls (-0.3,-2) .. (0,-.97); % e_0
%\draw[color=myblue,my dash, very thick] (0,3.03) .. controls (-0.3,2) .. (0,.97); % e_0'
\draw[color=myblue,my dash, very thick] (-7.03,0) .. controls (-6,0.3) .. (-4.97,0); %e_1
\draw[color=myblue,my dash, very thick] (7.03,0) .. controls (6,0.3) .. (4.97,0); %e_2
%\draw[color=myblue,my dash, very thick] (-3.03,0) .. controls (-2,0.3) .. (-.97,0); % e_0+e_1
%\draw[color=red,my dash ] (3.03,0) .. controls (2,0.3) .. (.97,0); % e_0+e_2
\draw[color=red, my dash] (2, -3.03) .. controls (1.7,-2) and (1.7,2) .. (2, 3.03);
\draw[color=red, my dash] (-2, -3.03) .. controls (-2.3,-2) and (-2.3,2) .. (-2, 3.03);
%\draw[color=myblue, very thick,my dash] (-3.293,-.707) .. controls (-1.5,-2.2) and (1.5,-2.2) .. (3.293,-.707);
%\draw[color=red,my dash] (-3.293,.707) .. controls (-1.5,2.8) and (1.5,2.8) .. (3.293,.707);

% Border lines

\draw [very thick] (0,0) circle (1);
\draw [very thick] (4,0) circle (1);
\draw [very thick] (-4,0) circle (1);
\draw [very thick]  (-4,3)--(4,3);
\draw [very thick]  (-4,-3)--(4,-3);
\draw [very thick] (-4,3) arc (90:270:3); 
\draw [very thick] (4,3) arc (90:-90:3); 

% Visible lines 

\draw[color=myblue, very thick] (0,-3.03) .. controls (0.3,-2) .. (0,-.97);  %e_0
%\draw[color=myblue, very thick] (0,3.03) .. controls (0.3,2) .. (0,.97) node [pos=0.5, right=-2pt]{$\alpha_1$};  %e_0'

\draw[color=myblue, very thick] (-4.97,0) .. controls (-6,-0.3) .. (-7.03,0);%e_1
%\draw[color=red,->-=.55] (4.97,0) .. controls (6,-0.4) .. (7.03,0) node [pos=.3,below=-1pt] {$\alpha_2$}; %e_2
%\draw[color=red,->-=.55] (-7.03,0) .. controls (-6,-0.3) .. (-4.97,0)node [pos=0.4,below] {$\alpha_1$};%e_1 -
\draw[color=myblue, very thick] (7.03,0) .. controls (6,-0.3) .. (4.97,0);%e_2 -

%\draw[color=myblue, very thick] (-3.03,0) .. controls (-2,-0.3) .. (-.97,0);%e_0+e_1
%\draw[color=red] (3.03,0) node [right=-2pt] {$\delta$} .. controls (2,-0.3) .. (.97,0) ;%e_0+e_2

\draw[color=red] (2, -3.03) .. controls (2.3,-2) and (2.3,2) .. (2, 3.03) node [pos=0.25, right=-2pt] {$\delta_2$};  %e_0

\draw[color=red] (-2, -3.03) .. controls (-1.7,-2) and (-1.7,2) .. (-2, 3.03) node [pos=0.25, right=-2pt] {$\delta_1$}; 

%\draw[color=red] (-2,-3.03) .. controls (-1.7,-1.5) and (-1.7,1.5) .. (-2,3.03) node [pos=.25,right=-2pt] {$u$};
%\draw[color=red,my dash] (-2,-3.03) .. controls (-2.3,-1.5) and (-2.3,1.5) .. (-2,3.03);

%\draw[color=red] (-3.293,-.707) .. controls (-1,-2.8) and (1,-2.8) .. (3.293,-.707) node [pos=.8,below=1pt] {$\delta_2$};
%\draw[color=red] (-3.293,.707) .. controls (-1,2.2) and (1,2.2) .. (3.293,.707)node [pos=.8,below] {$\delta$};

\end{scope}

\end{tikzpicture}
\caption{Abelian cycle~$\CA(T_{\delta_1},T_{\delta_2})$}\label{fig_2separate}
\end{figure}
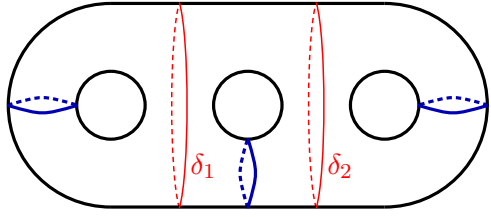

\textnormal{(b)} Suppose that $M$ is a four-component multicurve without a pair of homologous components such that some three components of~$M$ have modulo~$2$ homology classes~$\ba_1$, $\ba_2$, and $\ba_3$. Then the restriction of~$\theta_{\fA}$ to the stabilizer~$\I_M$ vanishes.
\end{propos}

\begin{proof}[Proof of part~\textnormal{(a)}  of Proposition~\ref{propos_theta_main}]
The curves~$\delta_1$ and~$\delta_2$ divide~$S$ into three parts; we denote them by~$\Sigma_1$, $\Sigma_2$, and~$\Sigma_3$ so that $\Sigma_1$ and~$\Sigma_2$ are one-punctured tori bounded by~$\delta_1$ and~$\delta_2$, respectively, and~$\Sigma_3$ is a two-punctured torus between~$\delta_1$ and~$\delta_2$. Put $V_i=H_1(\Sigma_i;\Z/2)$, $i=1,2,3$; then $H_{\Z/2}=V_1\oplus V_2\oplus V_3$. The multicurve $N$ has exactly one component in every~$\Sigma_i$; renumbering $\ba_i$'s, we may achieve that the component of~$N$ lying in~$\Sigma_i$ has modulo~$2$ homology class~$\ba_i$, $i=1,2,3$.
Then $\ba_i\in V_i$. For every $i=1,2,3$, extend~$\ba_i$ to a basis $\ba_i,\bb_i$ of~$V_i$. Now, number the functions~$\omega_i$ so that~\eqref{eq_omega_numeration} holds. By~\eqref{eq_BC1} we have
$$
\rho_i(T_{\delta_k})=\omega_i(\ba_k)\omega_i(\bb_k)=
\left\{
\begin{aligned}
&0,&&\text{if $i=0$ or $i=k$,}\\
&1,&&\text{if $i\ne 0$ and~$i\ne k$.}
\end{aligned}
\right.
$$
The standard generator of the group $H_2(\Z\times\Z;\Z)$ can be represented in the bar resolution by the cycle $[e_1|e_2]-[e_2|e_1]$, where $e_1$ and~$e_2$ are the standard generators of~$\Z\times\Z$. Hence, the abelian cycle~$\CA(T_{\delta_1},T_{\delta_2})$ can be represented in the bar resolution by the cycle $[T_{\delta_1}|T_{\delta_2}]-[T_{\delta_2}|T_{\delta_1}]$. Therefore 
$$
\bigl\langle\theta_{\fA},\CA(T_{\delta_1},T_{\delta_2})\bigr\rangle=\sum_{i<j}\bigl\langle \rho_i\rho_j,\CA(T_{\delta_1},T_{\delta_2})\bigr\rangle = \sum_{i\ne j}\rho_i(T_{\delta_1})\rho_j(T_{\delta_2})=1,
$$
which is exactly the required equality.
\end{proof}

\begin{remark}
The idea to distinguish abelian cycles in Torelli groups by means of cup-products of Birman--Craggs homomorphisms is due to Brendle and Farb~\cite{BrFa07}. In a sense, we use this idea in the opposite direction: we use abelian cycles to prove non-triviality of the quadratic polynomials~$\theta_{\fA}$ in Birman--Craggs homomorphisms. 
\end{remark}

We postpone the proof of part~(b) to Section~\ref{section_4component}.

Now, let us consider spectral sequence~\eqref{eq_SpSeq} for the action of~$\I$ on the complex of cycles~$\B$ with coefficients in~$\Z$:
\begin{equation}\label{eq_CL_final}
E^1_{p,q}=\bigoplus_{M\in\mathfrak{M}_p}H_q(\I_M;\Z) \  \Longrightarrow \ H_{p+q}(\I;\Z),
\end{equation}
where $\mathfrak{M}_p$ is a set of representatives of all $\I$-orbits of multicurves in~$\M_p$.

For $M\in\M_p$ and $z\in H_q(\I_M;\Z)$, we denote by~$z_M$ the element of~$E^1_{p,q}$ equal to the homology class $z$ in the summand~$H_q(\I_M;\Z)$. (Note that $z_M$ changes its sign whenever we reverse the orientation of the cell~$P_M$.) Recall that by Fact~\ref{fact_E1}, the summands in decomposition~\eqref{eq_CL_final} are independent of the choice of the set of representatives~$\mathfrak{M}_p$, and $(g_*z)_{g(M)}=z_M$, provided that the cells~$P_M$ and~$P_{g(M)}$ are endowed with orientations that are taken one into the other by the action of~$g$. Here $g_*\colon H_q(\I_M;\Z)\to  H_q(\I_{g(M)};\Z)$ is the homomorphism induced by the conjugation by~$g$. 

By Proposition~\ref{propos_H0} each $A\in\CH_0$ is a set consisting of~$1$, $2$, or~$3$ linearly independent elements that form a basis of a direct summand of~$H$. We denote by~$\CH_0'$ the subset of~$\CH_0$ consisting of all $3$-element sets. Then~$\CH_0'$ is exactly the set of all subsets $A=\{a_1,a_2,a_3\}$ of~$H$ that satisfy the following conditions:
\begin{itemize}
\item $a_1$, $a_2$, and~$a_3$ is a basis of a Lagrangian subgroup of~$H$,
\item $x=n_1a_1+n_2a_2+n_3a_3$ for some positive integers~$n_1$, $n_2$, and~$n_3$.  
\end{itemize}
Obviously, the set $\CH'_0$ is infinite.

Take $A\in\CH'_0$ and let $\fA=\{\ba_1,\ba_2,\ba_3\}\subset H_{\Z/2}$ be the modulo~$2$ reduction of~$A$. Then~$\fA$ satisfies condition~$(*)$, so the homology class~$\theta_{\fA}\in H^2(\I;\Z/2)$ is well defined. 

We define a homomorphism $\btheta_A\colon E^1_{0,2}\to\Z/2$ as follows. For any $N\in\M_0$ and any $z\in H_2(\I_N;\Z)$, we put
\begin{equation*}
\btheta_A(z_N)=\left\{
\begin{aligned}
&\langle \theta_{\fA}, i_{N*}z\rangle&&\text{if }[N]=A,\\
&0&&\text{if }[N]\ne A,\\
\end{aligned}
\right.
\end{equation*}
where $i_N\colon \I_N\to\I$ is the inclusion. By Fact~\ref{fact_E1} the homomorphism~$\btheta_A$ is independent of the choice of the set of representatives~$\mathfrak{M}_0$ and hence well defined.

Note that there are infinitely many homomorphisms~$\btheta_A$, though there are only finitely many cohomology classes~$\theta_{\fA}$. The main result of the present paper is as follows.

\begin{theorem}\label{theorem_main_section}
The homomorphisms~$\btheta_A$, where $A\in\CH_0'$, vanish on the images of the differentials~$d^1$ and $d^2$ and hence yield well-defined homomorphisms $\btheta_A\colon E^3_{0,2}\to\Z/2$. Moreover, these resulting homomorphisms are linearly independent over~$\Z/2$, so the group~$E^3_{0,2}$ is not finitely generated.
\end{theorem}

A hard part of this theorem is vanishing on the images of differentials. The proof of the linear independence is easy; let us start with it.

\begin{propos}\label{propos_li}
The homomorphisms~$\btheta_A\colon E^1_{0,2}\to\Z/2$, where $A\in\CH_0'$, are linearly independent over~$\Z/2$.
\end{propos}

\begin{proof}
Each homomorphism~$\btheta_A$ may take nonzero values on only one summand of  decomposition~\eqref{eq_CL_final} of~$E^1_{0,2}$, and these summands are pairwise different for different sets~$A$. Therefore, to settle the required linear independence it is sufficient to show that every homomorphism~$\btheta_A$ is non-trivial.

Consider an oriented multicurve $N\in \M_0$ such that $[N]=A$, and choose two non-isotopic separating simple closed curves $\delta_1$ and~$\delta_2$ that are disjoint from~$N$ and from each other. It follows from Proposition~\ref{propos_theta_main}(a) that 
$\btheta_A\bigl(\CA(T_{\delta_1},T_{\delta_2})_N\bigr)=1$, 
so the homomorphism~$\btheta_A$ is non-trivial. 
\end{proof}

\begin{propos}\label{propos_d1}
For each $A\in\CH_0'$, the restriction of the homomorphism~$\btheta_A$ to the image of the differential $d^1\colon E^1_{1,2}\to E^1_{0,2}$ is trivial. 
\end{propos}

\begin{proof}
We need to prove that $\btheta_A(d^1z_M)=0$ for any $M\in\M_1$ and any $z\in H_2(\I_M;\Z)$. Let $P_{N_1}$ and~$P_{N_2}$ be the endpoints of~$P_M$, then $N_1$ and~$N_2$ are the two oriented multicurves that are contained in~$M$ and belong to~$\M_0$. We have $\partial P_M=P_{N_2}-P_{N_1}$, provided that $P_M$ is oriented from~$P_{N_1}$ to~$P_{N_2}$. Since the differential~$d^1$ is induced by the differential~$\partial$ of the chain complex~$C_*(\B;\Z)$, we obtain that $d^1z_M=(i_{2*}z)_{N_2}-(i_{1*}z)_{N_1}$, where $i_k\colon\I_M\to\I_{N_k}$ are the inclusions, $k=1,2$. If the multiset~$[M]$ does not contain~$A$, then neither $[N_1]$ nor $[N_2]$ coincides with~$A$; hence $\btheta_A\bigl((i_{k*}z)_{N_k}\bigr)=0$ for $k=1,2$ and therefore $\btheta_A(d^1z_M)=0$. Suppose that $[M]\supset A$. If $M$ does not contain a pair of homologous components, then $\btheta_A(d^1z_M)=0$, since $\langle \theta_{\fA},i_{M*}z\rangle=\langle i_M^*\theta_{\fA},z\rangle=0$  by Proposition~\ref{propos_theta_main}(b). Finally, if $M$ contains a pair of homologous components, then the multicurves~$N_1$ and~$N_2$ are obtained from~$M$ by removing one of the two components in this pair. Hence $[N_1]=[N_2]=A$. Therefore $\btheta_A\bigl((i_{k*}z)_{N_k}\bigr)=\langle \theta_{\fA},i_{M*}z\rangle$ for $k=1,2$. Thus, $\btheta_A(d^1z_M)=0$. 
\end{proof}

By Proposition~\ref{propos_d1} every homomorphism~$\btheta_A$ induces a well-defined homomorphism $E^2_{0,2}\to\Z/2$, which we conveniently denote by~$\btheta_A$, too. Our  aim for the rest of the present paper is to prove the following proposition.

\begin{propos}\label{propos_d2}
For each $A\in\CH_0'$, the restriction of the homomorphism~$\btheta_A$ to the image of the differential $d^2\colon E^2_{2,1}\to E^2_{0,2}$ is trivial. 
\end{propos}

Theorem~\ref{theorem_main_section} follows immediately from Propositions~\ref{propos_li}, \ref{propos_d1}, and~\ref{propos_d2}.

\section{Cells of~$\B$ and their orientations}\label{section_cells}

Note that a multicurve in a closed genus~$3$ surface may contain at most one bounding pair.

It follows easily from~\eqref{eq_dimension} that every multicurve~$M\in\M_p$ has at most~$p+3$ components and decomposes~$S$ into exactly $p+1$ parts.
We denote by $\M_p'$ the subset of~$\M_p$ consisting of all~$M$ that satisfy the following two conditions:
\begin{itemize}
\item $M$ consists of exactly $p+3$ components,
\item $M$ does not contain a bounding pair,
\end{itemize}
and we denote by~$\CH_p'$ the set of all sets~$[M]$ for $M\in\M_p'$. Note that for $p=0$ this definition agrees with the definition in Section~\ref{section_theta}, since a multicurve in~$\M_0$ never contains a bounding pair. 

If we forget about orientation, and consider multicurves up to the action of the whole group~$\Mod(S)$, then there is only one type of multicurves in~$\M_0'$ and only one type of multicurves in~$\M_2'$, see Figs.~\ref{fig_types}(a) and~(d), respectively.  However, there are two different types of multicurves in~$\M_1'$. For a multicurve~$S$ of \textit{type}~1, $S\setminus M$ is the disjoint union of two $4$-punctured spheres, see Fig.~\ref{fig_types}(b), and for a multicurve~$S$ of \textit{type}~2, $S\setminus M$ is the disjoint union of a $3$-punctured sphere and a $5$-punctured spheres, see Fig.~\ref{fig_types}(c). We denote by~$\M_1^{(1)}$ and~$\M_1^{(2)}$ the subsets of~$\M_1'$ consisting of all type~1 and type~2 multicurves, respectively. Similarly, we denote by~$\CH_1^{(1)}$ and~$\CH_1^{(2)}$ the corresponding subsets of~$\CH_1'$.

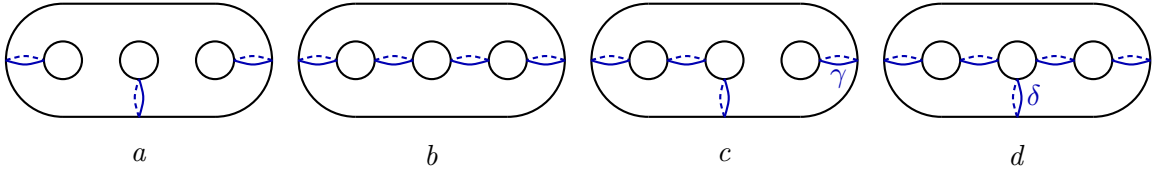
\begin{figure}
\begin{tikzpicture}[scale=.25]
\small
\definecolor{myblue}{rgb}{0, 0, 0.7}
\definecolor{mygreen}{rgb}{0, 0.4, 0}
\tikzset{e path/.append style={line width=.2mm}}
\tikzset{my dash/.style={dash pattern=on 2pt off 1.5pt}}

\begin{scope}[shift={(-23.1,0)}]
% Invisible lines

\draw[color=myblue,my dash,  thick] (0,-3.03) .. controls (-0.3,-2) .. (0,-.97); % e_0
%\draw[red,my dash] (0,3.03) .. controls (-0.3,2) .. (0,.97); % e_0'
\draw[color=myblue,my dash,  thick] (-7.03,0) .. controls (-6,0.3) .. (-4.97,0); %e_1
\draw[color=myblue,my dash,  thick] (7.03,0) .. controls (6,0.3) .. (4.97,0); %e_2
%\draw[color=myblue,my dash,  thick] (-3.03,0) .. controls (-2,0.3) .. (-.97,0); % e_0+e_1
%\draw[color=myblue,my dash,  thick] (3.03,0) .. controls (2,0.3) .. (.97,0); % e_0+e_2
%\draw[color=myblue,  thick,my dash] (-3.293,-.707) .. controls (-1.5,-2.2) and (1.5,-2.2) .. (3.293,-.707);
%\draw[color=red,my dash] (-3.293,.707) .. controls (-1.5,2.8) and (1.5,2.8) .. (3.293,.707);

% Border lines

\draw [ thick] (0,0) circle (1);
\draw [ thick] (4,0) circle (1);
\draw [ thick] (-4,0) circle (1);
\draw [ thick]  (-4,3)--(4,3);
\draw [ thick]  (-4,-3)--(4,-3);
\draw [ thick] (-4,3) arc (90:270:3); 
\draw [ thick] (4,3) arc (90:-90:3); 

% Visible lines 

\draw[color=myblue,  thick] (0,-3.03) .. controls (0.3,-2) .. (0,-.97) ;  %e_0
%\draw[color=red] (0,3.03) .. controls (0.3,2) .. (0,.97);  %e_0'

\draw[color=myblue,  thick] (-4.97,0) .. controls (-6,-0.3) .. (-7.03,0);%e_1
%\draw[color=red,->-=.55] (4.97,0) .. controls (6,-0.4) .. (7.03,0) node [pos=.3,below=-1pt] {$\alpha_2$}; %e_2
%\draw[color=red,->-=.55] (-7.03,0) .. controls (-6,-0.3) .. (-4.97,0)node [pos=0.4,below] {$\alpha_1$};%e_1 -
\draw[color=myblue,  thick] (7.03,0) .. controls (6,-0.3) .. (4.97,0);%e_2 -

%\draw[color=myblue,  thick] (-3.03,0) .. controls (-2,-0.3) .. (-.97,0);%e_0+e_1
%\draw[color=myblue,  thick] (3.03,0) .. controls (2,-0.3) .. (.97,0);%e_0+e_2

%\draw[color=red] (-2,-3.03) .. controls (-1.7,-1.5) and (-1.7,1.5) .. (-2,3.03) node [pos=.25,right=-2pt] {$u$};
%\draw[color=red,my dash] (-2,-3.03) .. controls (-2.3,-1.5) and (-2.3,1.5) .. (-2,3.03);

%\draw[color=red] (-3.293,-.707) .. controls (-1,-2.8) and (1,-2.8) .. (3.293,-.707) node [pos=.8,below=1pt] {$\delta_2$};
%\draw[color=red] (-3.293,.707) .. controls (-1,2.2) and (1,2.2) .. (3.293,.707);

\node[] at (0,-5) {\textit{a}};
\end{scope}

\begin{scope}[shift={(-7.7,0)}]
% Invisible lines

%\draw[color=myblue,my dash,  thick] (0,-3.03) .. controls (-0.3,-2) .. (0,-.97); % e_0
%\draw[red,my dash] (0,3.03) .. controls (-0.3,2) .. (0,.97); % e_0'
\draw[color=myblue,my dash,  thick] (-7.03,0) .. controls (-6,0.3) .. (-4.97,0); %e_1
\draw[color=myblue,my dash,  thick] (7.03,0) .. controls (6,0.3) .. (4.97,0); %e_2
\draw[color=myblue,my dash,  thick] (-3.03,0) .. controls (-2,0.3) .. (-.97,0); % e_0+e_1
\draw[color=myblue,my dash,  thick] (3.03,0) .. controls (2,0.3) .. (.97,0); % e_0+e_2
%\draw[color=myblue,  thick,my dash] (-3.293,-.707) .. controls (-1.5,-2.2) and (1.5,-2.2) .. (3.293,-.707);
%\draw[color=red,my dash] (-3.293,.707) .. controls (-1.5,2.8) and (1.5,2.8) .. (3.293,.707);

% Border lines

\draw [ thick] (0,0) circle (1);
\draw [ thick] (4,0) circle (1);
\draw [ thick] (-4,0) circle (1);
\draw [ thick]  (-4,3)--(4,3);
\draw [ thick]  (-4,-3)--(4,-3);
\draw [ thick] (-4,3) arc (90:270:3); 
\draw [ thick] (4,3) arc (90:-90:3); 

% Visible lines 

%\draw[color=myblue,  thick] (0,-3.03) .. controls (0.3,-2) .. (0,-.97) ;  %e_0
%\draw[color=red] (0,3.03) .. controls (0.3,2) .. (0,.97);  %e_0'

\draw[color=myblue,  thick] (-4.97,0) .. controls (-6,-0.3) .. (-7.03,0);%e_1
%\draw[color=red,->-=.55] (4.97,0) .. controls (6,-0.4) .. (7.03,0) node [pos=.3,below=-1pt] {$\alpha_2$}; %e_2
%\draw[color=red,->-=.55] (-7.03,0) .. controls (-6,-0.3) .. (-4.97,0)node [pos=0.4,below] {$\alpha_1$};%e_1 -
\draw[color=myblue,  thick] (7.03,0) .. controls (6,-0.3) .. (4.97,0);%e_2 -

\draw[color=myblue,  thick] (-3.03,0) .. controls (-2,-0.3) .. (-.97,0);%e_0+e_1
\draw[color=myblue,  thick] (3.03,0) .. controls (2,-0.3) .. (.97,0);%e_0+e_2

%\draw[color=red] (-2,-3.03) .. controls (-1.7,-1.5) and (-1.7,1.5) .. (-2,3.03) node [pos=.25,right=-2pt] {$u$};
%\draw[color=red,my dash] (-2,-3.03) .. controls (-2.3,-1.5) and (-2.3,1.5) .. (-2,3.03);

%\draw[color=red] (-3.293,-.707) .. controls (-1,-2.8) and (1,-2.8) .. (3.293,-.707) node [pos=.8,below=1pt] {$\delta_2$};
%\draw[color=red] (-3.293,.707) .. controls (-1,2.2) and (1,2.2) .. (3.293,.707);

\node[] at (0,-5) {\textit{b}};
\end{scope}

\begin{scope}[shift={(7.7,0)}]
% Invisible lines

\draw[color=myblue,my dash,  thick] (0,-3.03) .. controls (-0.3,-2) .. (0,-.97); % e_0
%\draw[red,my dash] (0,3.03) .. controls (-0.3,2) .. (0,.97); % e_0'
\draw[color=myblue,my dash,  thick] (-7.03,0) .. controls (-6,0.3) .. (-4.97,0); %e_1
\draw[color=myblue,my dash,  thick] (7.03,0) .. controls (6,0.3) .. (4.97,0); %e_2
\draw[color=myblue,my dash,  thick] (-3.03,0) .. controls (-2,0.3) .. (-.97,0); % e_0+e_1
%\draw[color=myblue,my dash,  thick] (3.03,0) .. controls (2,0.3) .. (.97,0); % e_0+e_2
%\draw[color=myblue,  thick,my dash] (-3.293,-.707) .. controls (-1.5,-2.2) and (1.5,-2.2) .. (3.293,-.707);
%\draw[color=red,my dash] (-3.293,.707) .. controls (-1.5,2.8) and (1.5,2.8) .. (3.293,.707);

% Border lines

\draw [ thick] (0,0) circle (1);
\draw [ thick] (4,0) circle (1);
\draw [ thick] (-4,0) circle (1);
\draw [ thick]  (-4,3)--(4,3);
\draw [ thick]  (-4,-3)--(4,-3);
\draw [ thick] (-4,3) arc (90:270:3); 
\draw [ thick] (4,3) arc (90:-90:3); 

% Visible lines 

\draw[color=myblue,  thick] (0,-3.03) .. controls (0.3,-2) .. (0,-.97) ;  %e_0
%\draw[color=red] (0,3.03) .. controls (0.3,2) .. (0,.97);  %e_0'

\draw[color=myblue,  thick] (-4.97,0) .. controls (-6,-0.3) .. (-7.03,0);%e_1
%\draw[color=red,->-=.55] (4.97,0) .. controls (6,-0.4) .. (7.03,0) node [pos=.3,below=-1pt] {$\alpha_2$}; %e_2
%\draw[color=red,->-=.55] (-7.03,0) .. controls (-6,-0.3) .. (-4.97,0)node [pos=0.4,below] {$\alpha_1$};%e_1 -
\draw[color=myblue,  thick] (7.03,0) .. controls (6,-0.3) .. (4.97,0) node [pos=.5, below=-2pt] {$\gamma$};%e_2 -

\draw[color=myblue,  thick] (-3.03,0) .. controls (-2,-0.3) .. (-.97,0);%e_0+e_1
%\draw[color=myblue,  thick] (3.03,0) .. controls (2,-0.3) .. (.97,0);%e_0+e_2

%\draw[color=red] (-2,-3.03) .. controls (-1.7,-1.5) and (-1.7,1.5) .. (-2,3.03) node [pos=.25,right=-2pt] {$u$};
%\draw[color=red,my dash] (-2,-3.03) .. controls (-2.3,-1.5) and (-2.3,1.5) .. (-2,3.03);

%\draw[color=red] (-3.293,-.707) .. controls (-1,-2.8) and (1,-2.8) .. (3.293,-.707) node [pos=.8,below=1pt] {$\delta_2$};
%\draw[color=red] (-3.293,.707) .. controls (-1,2.2) and (1,2.2) .. (3.293,.707);

\node[] at (0,-5) {\textit{c}};
\end{scope}

\begin{scope}[shift={(23.1,0)}]
% Invisible lines

\draw[color=myblue,my dash,  thick] (0,-3.03) .. controls (-0.3,-2) .. (0,-.97); % e_0
%\draw[red,my dash] (0,3.03) .. controls (-0.3,2) .. (0,.97); % e_0'
\draw[color=myblue,my dash,  thick] (-7.03,0) .. controls (-6,0.3) .. (-4.97,0); %e_1
\draw[color=myblue,my dash,  thick] (7.03,0) .. controls (6,0.3) .. (4.97,0); %e_2
\draw[color=myblue,my dash,  thick] (-3.03,0) .. controls (-2,0.3) .. (-.97,0); % e_0+e_1
\draw[color=myblue,my dash,  thick] (3.03,0) .. controls (2,0.3) .. (.97,0); % e_0+e_2
%\draw[color=myblue,  thick,my dash] (-3.293,-.707) .. controls (-1.5,-2.2) and (1.5,-2.2) .. (3.293,-.707);
%\draw[color=red,my dash] (-3.293,.707) .. controls (-1.5,2.8) and (1.5,2.8) .. (3.293,.707);

% Border lines

\draw [ thick] (0,0) circle (1);
\draw [ thick] (4,0) circle (1);
\draw [ thick] (-4,0) circle (1);
\draw [ thick]  (-4,3)--(4,3);
\draw [ thick]  (-4,-3)--(4,-3);
\draw [ thick] (-4,3) arc (90:270:3); 
\draw [ thick] (4,3) arc (90:-90:3); 

% Visible lines 

\draw[color=myblue,  thick] (0,-3.03) .. controls (0.3,-2) .. (0,-.97) node [pos=.6, right=-2pt] {$\delta$};  %e_0
%\draw[color=red] (0,3.03) .. controls (0.3,2) .. (0,.97);  %e_0'

\draw[color=myblue,  thick] (-4.97,0) .. controls (-6,-0.3) .. (-7.03,0);%e_1
%\draw[color=red,->-=.55] (4.97,0) .. controls (6,-0.4) .. (7.03,0) node [pos=.3,below=-1pt] {$\alpha_2$}; %e_2
%\draw[color=red,->-=.55] (-7.03,0) .. controls (-6,-0.3) .. (-4.97,0)node [pos=0.4,below] {$\alpha_1$};%e_1 -
\draw[color=myblue,  thick] (7.03,0) .. controls (6,-0.3) .. (4.97,0);%e_2 -

\draw[color=myblue,  thick] (-3.03,0) .. controls (-2,-0.3) .. (-.97,0);%e_0+e_1
\draw[color=myblue,  thick] (3.03,0) .. controls (2,-0.3) .. (.97,0);%e_0+e_2

%\draw[color=red] (-2,-3.03) .. controls (-1.7,-1.5) and (-1.7,1.5) .. (-2,3.03) node [pos=.25,right=-2pt] {$u$};
%\draw[color=red,my dash] (-2,-3.03) .. controls (-2.3,-1.5) and (-2.3,1.5) .. (-2,3.03);

%\draw[color=red] (-3.293,-.707) .. controls (-1,-2.8) and (1,-2.8) .. (3.293,-.707) node [pos=.8,below=1pt] {$\delta_2$};
%\draw[color=red] (-3.293,.707) .. controls (-1,2.2) and (1,2.2) .. (3.293,.707);

\node[] at (0,-5) {\textit{d}};
\end{scope}

\end{tikzpicture}
\caption{Multicurves in the sets (a) $\M_0'$, (b) $\M_1^{(1)}$, (c) $\M_1^{(2)}$, and~(d) $\M_2'$}\label{fig_types}
\end{figure}

A component of a multicurve~$M\in\M_1^{(2)}$ is called \textit{special} if it is not adjacent to the connected component of~$S\setminus M$ that is homeomorphic to a sphere with three punctures, like~$\gamma$ in~Fig.~\ref{fig_types}(c). Three other components of~$M$ are called \textit{non-special}. A component of a multicurve~$M\in\M_2'$ is called \textit{principal} if it is adjacent to both connected components of~$S\setminus M$ that are homeomorphic to a sphere with three punctures, like~$\delta$ in~Fig.~\ref{fig_types}(d). Four other components of~$M$ are called \textit{non-principal}. We will use the same terminology for elements of sets in~$\CH_1^{(2)}$ and~$\CH_2'$.

We will need the following easy fact.

\begin{propos}\label{propos_M2no'}
Any oriented multicurve in $\M_2\setminus \M_2'$ contains a bounding pair.
\end{propos}

\begin{proof}
It is easy to check that it is impossible to decompose a closed genus~$3$ surface into~$3$ parts by a multicurve that contains no separating components and no bounding pairs and has less than $5$ components.
\end{proof}

Note that similar assertions for multicurves in $\M_0\setminus\M_0'$ and~$\M_1\setminus\M_1'$ are false.

\begin{propos}\label{propos_orbits}
\textnormal{(a)} Suppose that $C$ belongs to either $\CH_0'$ or~$\CH_1^{(1)}$. Then all oriented multicurves~$M$ with $[M]=C$ lie in the same $\I$-orbit.

\textnormal{(b)} Suppose that $C$ belongs to either~$\CH_1^{(2)}$ or~$\CH_2'$. Then all oriented multicurves~$M$ with $[M]=C$ lie in the same $\hI$-orbit. This $\hI$-orbit consists of exactly two $\I$-orbits.
\end{propos}

\begin{proof}
First, suppose that $C\in\CH_1^{(2)}\cup\CH_2'$. Let $c$ be an element of~$C$ that is not special in case $C\in\CH_1^{(2)}$ and is not principal in case $C\in\CH_2'$. Suppose that $M$ is an oriented multicurve with~$[M]=C$, and $\gamma$ is the component of~$M$ with~$[\gamma]=c$. Consider the components of~$S\setminus M$ on both sides of~$\gamma$. Then exactly one of them is a three-punctured sphere, and the other is either five- or four-punctured sphere. We can say whether $M$ is \textit{right} or \textit{left} with respect to~$c$ depending of on which side of~$\gamma$ the component homeomorphic to a three-punctured sphere lies. Obviously, an element of the Torelli group cannot take a right multicurve with~$[M]=C$ to a left one, and vice versa. On the other hand, any element $h\in\hI\setminus\I$ transforms right multicurves to left ones and left multicurves to right ones, since we artificially reverse the orientations of components while acting by~$h$. 

So to prove the proposition we need to show that two multicurves~$M_1$ and~$M_2$ with $[M_1]=[M_2]=C\in\CH_p'$, $p=0,1,2$, lie in the same $\I$-orbit, provided that in cases $C\in\CH_1^{(2)}$ and $C\in\CH_2'$ we additionally know that they are either both right or both left with respect to some~$c$. 

Let $L\subset H$ be the Lagrangian subgroup generated by elements of~$C$ and $\CL\subset\Mod(S)$ be the corresponding \textit{Lagrangian mapping class group} consisting of all mapping classes that act trivially on~$L$. It is easy to prove that under the condition stated above~$M_1$ can be taken to~$M_2$ by a mapping class~$f\in\CL$. Indeed, one can just map every components of~$M_1$ onto the corresponding component of~$M_2$ (preserving the orientation), and then extend this map to every connected component of~$S\setminus M_1$. Now, the image~$\mathbb{L}$ of~$\CL$ under the natural surjective homomorphism $\Mod(S_3)\to\Sp(6,\Z)$ is a free abelian group of rank~$6$. If we write transformations in~$\Sp(6,\Z)$ in a symplectic basis $a_1,a_2,a_3,b_1,b_2,b_3$ such that $a_1,a_2,a_3$ is a basis of~$L$, then $\mathbb{L}$ will consist of all matrices of the form
$
\begin{pmatrix}
I&P\\
0&I
\end{pmatrix},
$
where $I$ is the unit $3\times 3$ matrix and
$P$
is a symmetric $3\times 3$ integral matrix. 

Now, let us make two easy observations. First, suppose that $a_1,a_2,a_3$ is a symplectic basis of~$L$ and $\gamma_1,\ldots,\gamma_6$ are simple closed curves in the homology classes $a_1$, $a_2$, $a_3$, $a_1+a_2$, $a_1+a_3$, and~$a_1+a_2+a_3$, respectively; then the images of the twists~$T_{\gamma_1},\ldots,T_{\gamma_6}$ form a basis of~$\mathbb{L}$. This can be checked immediately by writing down the corresponding matrices~$P$. Second, there exist $6$ curves $\gamma_1,\ldots,\gamma_6$ that are disjoint from~$M_1$ (where components of~$M_1$ are also allowed) so that their homology classes are exactly the $6$ classes listed above for some basis $a_1,a_2,a_3$ of~$L$. This can be checked by cases for all four types of multicurves in~$\M_p'$, $p=0,1,2$. Now, we obtain that every mapping class $hT_{\gamma_1}^{k_1}\cdots T_{\gamma_6}^{k_6}$, where $k_1,\ldots,k_6\in\Z$, takes $M_1$ to~$M_2$, and exactly one of those mapping classes belongs to~$\I$.
\end{proof}

For each $C\in \CH_0'\cup\CH_1^{(1)}$, we denote by~$\orb_C$ the $\I$-orbit consisting of oriented multicurves~$M$ with~$[M]=C$. For each $C\in \CH_1^{(2)}\cup\CH_2'$, we denote by~$\orb_C^+$ and~$\orb_C^-$ the two $\I$-orbits consisting of oriented multicurves~$M$ with~$[M]=C$. The choice of which of the two orbits is~$\orb^+_C$ is not canonical. We suppose that some choice is made.

Now, suppose that $C\in\CH_1^{(2)}$, $D\in\CH_2'$ and $C\subset D$. Then we put $\varepsilon_{D,C}=1$ if every multicurve in~$\orb_{D}^+$ contains a multicurve in~$\orb_C^+$ (and then every multicurve in~$\orb_{D}^-$ contains a multicurve in~$\orb_C^-$),  and $\varepsilon_{D,C}=-1$ if every multicurve in~$\orb_{D}^+$ contains a multicurve in~$\orb_C^-$ (and then every multicurve in~$\orb_{D}^-$ contains a multicurve in~$\orb_C^+$).

\begin{remark}\label{remark_infinite_orbits}
It is not hard to show that whenever an oriented multicurve~$M_0$ contains a bounding pair of components, there are infinitely many $\I$-orbits of oriented multicurves~$M$ with $[M]=[M_0]$. We are not going to prove this assertion, since we never use it in the sequel.
\end{remark}

To work with the chain complex~$C_*(\B;\Z)$, we need to endow positive-dimensional cells of~$\B$ with orientations. We do not know any canonical way to do this, and we have no need in it. However, we need to introduce an agreement on orientations, which will guarantee that the orientations of any two cells~$M$ with the same~$[M]$ agree each other properly. This agreement will be different depending on whether $M$ contain a bounding pair or not. Consider those two cases separately.

1. Suppose that $C\in\M_p$, $p>0$, and no element has multiplicity~$2$ in~$C$. Then, for any two oriented multicurves~$M_1$ and~$M_2$ with $[M_1]=[M_2]=C$, there is a unique bijection between components of~$M_1$ and components of~$M_2$ that takes every component to a component in the same homology class. This bijection induces an isomorphism $P_{M_1}\to P_{M_2}$. Our agreement is that the chosen orientation of~$P_{M_1}$ must go to the chosen  orientation of~$P_{M_2}$ under this isomorphism. 

2. Suppose that $C\in\M_p$, $p>0$, and an element~$c$ has multiplicity~$2$ in~$C$. In fact, we would like to leave out a very special case~$C=[x,x]$. Oriented multicurves~$M$ with $[M]=[x,x]$  will never occur in this paper, so we do not have to take care of their orientations. Assume that $C\ne [x,x]$. Let $E$ be the set consisting of all elements of~$C$ that are different from~$c$. Then $E$ is non-empty. For each multicurve $M$ with $[M]=C$, the bounding pair~$\gamma\cup\gamma'$ contained in~$M$ separates~$S$ into two connected components, so yields a partition of~$E$ into two subsets. It is easy to see that this partition is independent of the choice of~$M$. Indeed, two elements $e_1,e_2\in E$ lie in the same part if and only if one of the three homology classes $e_1+e_2$, $e_1-e_2$, and~$e_2-e_1$ is equal to~$c$. So this partition is intrinsically determined by the initial multiset~$C$. (Here it is important that the genus of~$S$ is $3$.) Let us choose which part is the first and which is the second, and denote them by~$E_1$ and~$E_2$, respectively. Then, for an oriented multicurve~$M$ with $[M]$, we can decide which of the two homologous components of it is the \textit{first} one and which is the \textit{second} one by claiming that $E_1$ lies on the right-hand side of the first component and on the left-hand side of the second component.  Now, for any two oriented multicurves~$M_1$ and~$M_2$ with $[M_1]=[M_2]=C$, there is a unique bijection between sets of their components that takes every component of~$M_1$ to a component of~$M_2$ in the same homology class, and besides, takes the first and the second components of the bounding pair in~$M_1$ to the first and the second components of the bounding pair in~$M_2$, respectively. If we swap the parts~$E_1$ and~$E_2$, then the first and the second components will interchange simultaneously for all~$M$ with $[M]=C$; hence, the constructed bijection will not change. This bijection induces an isomorphism $P_{M_1}\to P_{M_2}$. Our agreement is that the chosen orientation of~$P_{M_1}$ must go to the chosen  orientation of~$P_{M_2}$ under this isomorphism. 

In both cases, the agreement made means that, once we have chosen an orientation of a cell~$P_M$ (with $[M]\ne [x,x]$), we immediately orient accordingly all other cells~$P_{M'}$ with $[M']=[M]$. Throughout the rest of the present paper, we assume that orientations of cells of~$\B$ are chosen arbitrarily according to this agreement. Sometimes, we will conveniently reverse the orientations of some cells~$P_M$. Nevertheless, this always means that we simultaneously reverse the orientations of all cells~$P_{M'}$ with $[M']=[M]$.

Also, hereafter we assume that the pair of homologous components of each multicurve $M$ (that contains such a pair) is ordered so that the orderings for different~$M$ with the same~$[M]$ are in coordination with each other, as above.

If $P_M$ is a codimension~$1$ face of~$P_K$, we denote by~$[P_K\colon P_M]$ the incidence coefficient of the cells~$P_K$ and~$P_M$. Since $\B$ is a regular CW complex, we always have $[P_K\colon P_M]=\pm 1$.

If $C\in\CH_p$ and $D\in\CH_{p+1}$ are sets without multiple elements such that $C\subset D$, then the incidence coefficient~$[P_K\colon P_M]$ is the same for every pair of oriented multicurves $M\subset K$ satisfying $[M]=C$ and $[K]=D$. We denote this incidence coefficient by~$[D\colon C]$.

Finally, suppose that $K\in\M_2$ is a $5$-component oriented multicurve containing a bounding pair of curves, and~$M^{\pm}$ are the two multicurves obtained from~$K$ by removing one of the two components of the bounding pair. Put $D=[K]$ and~$C=[M^{\pm}]$. Then $M^{\pm}$ belong to~$\M_1^{(2)}$ and lie in different $\I$-orbits, so we may swap them to obtain $M^+\in\orb_C^+$ and $M^-\in\orb_C^-$. The cell $P_K$ is isomorphic to the cylinder $P_{M^+}\times [0,1]$ with bases~$P_{M^+}$ and~$P_{M^-}$. Hence $[P_K\colon P_{M^+}]=-[P_K\colon P_{M^-}]$. On the other hand, the incidence coefficient $[P_K\colon P_{M^+}]$ will not change if we replace~$K$ with another multicurve~$K'$ such that $[K']=D$. We denote this incidence coefficient $[P_K\colon P_{M^+}]$ by~$[D\colon C]$.

\section{Homomorphisms~$\nu_{\gamma}$, $\nu_{\gamma,\CW}$, and~$\mu_{\gamma,\gamma'}$}
\label{section_nu}

Consider a closed oriented surface~$S_2$ of genus~$2$ and the corresponding Torelli group~$\I_2$. Recall a well-known fact that there exists a unique homomorphism $d_0\colon\I_2\to \Z$ satisfying $d_0(T_{\delta})=1$ for all separating simple closed curves~$\delta$ in~$S_2$. Generally, Morita~\cite{Mor91} constructed a homomorphism $d_0\colon \K_g\to\Z$ satisfying $d_0(T_{\delta})=g'(g-g')$ whenever $\delta$ separates~$S_g$ into two parts of genera~$g'$ and~$g-g'$, and proved that up to proportionality $d_0$ is the only  $\Mod_g$-invariant homomorphism~$\K_g\to\Z$. Here $\K_g\subset \I_g$ is the subgroup generated by Dehn twists about separating curves. ($\K_g$ is also called the \textit{Johnson kernel}, since it is the kernel of the \textit{Johnson homomorphism} $\tau\colon\I_g\to\Lambda^3H/H$, cf.~\cite{Joh85a}.) If $g=2$, then $\I_2=\K_2$, hence, Morita's construction gives the required homomorphism $d_0\colon\I_2\to \Z$.

Now, suppose that $\gamma$ is a non-separating simple closed curve in a closed oriented genus~$3$ surface~$S=S_3$. Then $S\setminus\gamma$ is a genus two surface with two punctures. Compactifying~$S\setminus\gamma$ by adding two points, we obtain a closed surface~$S_2$ of genus~$2$. Let~$\lambda_{\gamma}$ be  the composition of the natural homomorphisms
$$
\I_{\gamma}\hookrightarrow\PMod(S\setminus\gamma)\to\Mod(S_2).
$$
It is easy to see that the image of~$\lambda_{\gamma}$ is contained in the Torelli group $\I_2=\I(S_2)$. The homomorphism $\nu_{\gamma}\colon\I_{\gamma}\to\Z$ is, by definition, the composition
$$
\I_{\gamma}\xrightarrow{\lambda_{\gamma}}\I_2\xrightarrow{d_0}\Z.
$$

Next proposition follows immediately from the definition of~$\nu_{\gamma}$. 

\begin{propos}\label{propos_phi_values}
\textnormal{(a)} Suppose that $\delta$ is a separating simple closed curve disjoint from~$\gamma$. Then $\nu_{\gamma}(T_{\delta})=1$ whenever $\gamma$ lies in the genus~$2$ component of~$S\setminus\delta$ and~$\nu_{\gamma}(T_{\delta})=0$ whenever $\gamma$ lies in the genus~$1$ component of~$S\setminus\delta$.

\textnormal{(b)} If $\{\gamma,\gamma'\}$ is a bounding pair, then $\nu_{\gamma}(T_{\gamma,\gamma'})=-1$.

\textnormal{(c)} If $\{\delta,\delta'\}$ is a bounding pair such that $\delta$ and~$\delta'$ are disjoint from~$\gamma$ and neither~$\delta$ nor~$\delta'$ is homotopic to~$\gamma$, then $\nu_{\gamma}(T_{\delta,\delta'})=0$.
\end{propos}

\begin{propos}\label{propos_triv_restr1}
Suppose that $\{\gamma,\gamma'\}$ is a bounding pair. Then, for all $h\in\I_{\gamma\cup\gamma'}$,
$$
\nu_{\gamma}(h)\equiv\nu_{\gamma'}(h)\pmod 2.
$$
\end{propos}

\begin{proof}
By Proposition~\ref{propos_stab_generate}, it is enough to check the required equality for $h=T_{\gamma,\gamma'}$ and $h=T_{\delta}$, where $\delta$ is a  separating simple closed curve disjoint from~$\gamma$ and~$\gamma'$. In both cases the equality follows from Proposition~\ref{propos_phi_values}.
\end{proof}

Proposition~\ref{propos_triv_restr1} allows us to assign to a bounding pair~$\{\gamma,\gamma'\}$ the homomorphism $\mu_{\gamma,\gamma'}\colon \I_{\gamma\cup\gamma'}\to\Z$ defined by
\begin{equation}\label{eq_mu_defin}
\mu_{\gamma,\gamma'}(h)=\frac{\nu_{\gamma}(h)+\nu_{\gamma'}(h)}2\,.
\end{equation}
Obviously $\mu_{\gamma',\gamma}=\mu_{\gamma,\gamma'}$. Next proposition follows immediately from Proposition~\ref{propos_phi_values}.

\begin{propos}\label{propos_psi_values}
Suppose that $\{\gamma,\gamma'\}$ is a bounding pair. Then $\mu_{\gamma,\gamma'}(T_{\gamma,\gamma'})=0$. Besides, $\mu_{\gamma,\gamma'}(T_{\delta})=1$ whenever $\delta$ is a separating simple closed curve disjoint from~$\gamma$ and~$\gamma'$.
\end{propos}

Consider orthogonal splittings $H_1(S_2;\Z)=U\oplus V$ into two rank~$2$ subgroups. We regard any such splitting as a unordered pair of subgroups, that is, do not distinguish between splittings~$U\oplus V$ and~$V\oplus U$.
By a theorem of Mess~\cite{Mes92} 
\begin{itemize}
\item the Torelli group~$\I_2$ is an infinitely generated free group,
\item for free generators of~$\I_2$ one can take the Dehn twists~$T_{\delta_i}$ about certain separating curves~$\delta_i$ giving all pairwise different orthogonal splittings of~$H_1(S_2;\Z)$.  
\end{itemize}
Hence, for each orthogonal splitting~$\CW$  of~$H_1(S_2;\Z)$, there is a unique homomorphism $d_{\CW}\colon \I_2\to \Z$ such that $d_{\CW}(T_{\delta})=1$ whenever $\delta$ is a separating curve providing the splitting~$\CW$ in homology and $d_{\CW}(T_{\delta})=0$ whenever $\delta$ provides any other splitting. Obviously, Morita's homomorphism~$d_0$ is the sum of all~$d_{\CW}$ in the sense that, for all~$h\in\I_2$,
$$
d_0(h)=\sum d_{\CW}(h).
$$
(The sum in the right-hand side is finite for each~$h$.)

Using homomorphisms~$d_{\CW}$ instead of~$d_0$, we may refine our construction of homomorphisms~$\nu_{\gamma}$ as follows.
Suppose that $\gamma$ is a non-separating curve in~$S=S_3$, and $\CW$ is an orthogonal splitting of the rank~$4$ abelian group $H_{\gamma}=\langle [\gamma]\rangle^{\bot}/\langle [\gamma]\rangle$. As above, consider the genus~$2$ surface~$S_2$ obtained by  compactifying~$S\setminus\gamma$. Then the group~$H_1(S_2;\Z)$ is canonically isomorphic to~$H_{\gamma}$, so $\CW$ can be regarded as a splitting of it. The homomorphism $\nu_{\gamma,\CW}\colon\I_{\gamma}\to\Z$ is, by definition, the composition
$$
\I_{\gamma}\xrightarrow{\lambda_{\gamma}}\I_2\xrightarrow{d_{\CW}}\Z.
$$

Next proposition is an analogue of Proposition~\ref{propos_phi_values} in this situation. 

\begin{propos}\label{propos_phiW_values}
\textnormal{(a)} Suppose that $\delta$ is a separating simple closed curve that is disjoint from~$\gamma$. Then $\nu_{\gamma,\CW}(T_{\delta})=1$ if~$\delta$ yields the splitting~$\CW$ for~$H_{\gamma}$, and $\nu_{\gamma,\CW}(T_{\delta})=0$ otherwise.

\textnormal{(b)} Suppose that $\{\gamma,\gamma'\}$ is a bounding pair. Then  $\nu_{\gamma,\CW}(T_{\gamma,\gamma'})=-1$ if this pair yields the splitting~$\CW$ for~$H_{\gamma}$, and $\nu_{\gamma,\CW}(T_{\gamma,\gamma'})=0$ otherwise.

\textnormal{(c)} Suppose that  $\{\delta,\delta'\}$ is a bounding pair  such that $\delta$ and~$\delta'$ are disjoint from~$\gamma$ and neither~$\delta$ nor~$\delta'$ is homotopic to~$\gamma$, then $\nu_{\gamma,\CW}(T_{\delta,\delta'})=0$.
\end{propos}

\section{Stabilizers of four- and five-component multicurves without separating components and bounding pairs}\label{section_4component}

In this section we study the stabilizers~$\I_M$ for $4$-component multicurves in~$\M_1^{(1)}$ and~$\M_1^{(2)}$ and $5$-components multicurves in~$\M_2'$. Orientations of components of multicurves will not play any role, so we temporarily forget about them. We will conveniently study type~1 four-component multicurves first, then five-component multicurves, and only then type~2 four-component multicurves. In particularly, in Subsections~\ref{subsection_type1} and~\ref{subsection_type2} we prove Proposition~\ref{propos_theta_main}(b) for the two types of four-component multicurves, see Propositions~\ref{propos_theta_main1} and~\ref{propos_theta_main2}.

\subsection{Type~1 four-component multicurves}\label{subsection_type1}

Suppose that $M=\alpha_0\cup\alpha_1\cup\alpha_2\cup\alpha_3$ is a type~1 four-component multicurve in~$S$. Then $M$ divides $S$ into two connected components~$\Sigma_1$ and~$\Sigma_2$ each of which is homeomorphic to a four-punctured sphere, see Fig.~\ref{fig_psi}. We denote by $\ba_1$, $\ba_2$, and~$\ba_3$ the modulo~$2$ homology classes of~$\alpha_1$, $\alpha_2$, and~$\alpha_3$, respectively, and put $\fA=\{\ba_1,\ba_2,\ba_3\}$; then the modulo~$2$ homology class of~$\alpha_0$ is $\ba_1+\ba_2+\ba_3$. By~\eqref{eq_theta_equal} the class~$\theta_{\fA}$ is independent on the numeration of the components of~$M$. 

\begin{figure}
\begin{tikzpicture}[scale=.45]
\small
\definecolor{myblue}{rgb}{0, 0, 0.7}
\definecolor{mygreen}{rgb}{0, 0.4, 0}
\tikzset{every path/.append style={line width=.2mm}}
\tikzset{my dash/.style={dash pattern=on 2pt off 1.5pt}}

\begin{scope}[shift={(-9,0)}]
% Invisible lines

\draw[red,my dash] (0,-3.03) .. controls (-0.3,-2) .. (0,-.97); % e_0
\draw[red,my dash] (0,3.03) .. controls (-0.3,2) .. (0,.97); % e_0'
\draw[color=myblue,my dash, very thick] (-7.03,0) .. controls (-6,0.3) .. (-4.97,0); %e_1
\draw[color=myblue,my dash, very thick] (7.03,0) .. controls (6,0.3) .. (4.97,0); %e_2
\draw[color=myblue,my dash, very thick] (-3.03,0) .. controls (-2,0.3) .. (-.97,0); % e_0+e_1
\draw[color=myblue,my dash, very thick] (3.03,0) .. controls (2,0.3) .. (.97,0); % e_0+e_2
\draw[color=red,my dash] (-3.293,-.707) .. controls (-1.5,-2.2) and (1.5,-2.2) .. (3.293,-.707);
\draw[color=red,my dash] (-3.293,.707) .. controls (-1.5,2.8) and (1.5,2.8) .. (3.293,.707);

% Border lines

\draw [very thick] (0,0) circle (1);
\draw [very thick] (4,0) circle (1);
\draw [very thick] (-4,0) circle (1);
\draw [very thick]  (-4,3)--(4,3);
\draw [very thick]  (-4,-3)--(4,-3);
\draw [very thick] (-4,3) arc (90:270:3); 
\draw [very thick] (4,3) arc (90:-90:3) node [pos=.25, above right] {$\Sigma_1$} node [pos=.75, below right] {$\Sigma_2$}; 
\draw [] (-9,0) -- (-7.3,0);
\draw [] (9,0) -- (7.3,0);

\draw (8.3,0) + (-55: .3 and 1) arc (-55: -10: .3 and 1);
\draw [-stealth] (8.3,0) + (10: .3 and 1) arc (10: 160: .3 and 1) node [pos=.4, right] {$\iota$};

% Visible lines 

\draw[color=red] (0,-3.03) .. controls (0.3,-2) .. (0,-.97) node [above=-2pt] {$\gamma_2$};  %e_0
\draw[color=red] (0,3.03) .. controls (0.3,2) .. (0,.97) node [below=-1pt] {$\gamma_1$};  %e_0'

\draw[color=myblue, very thick] (-4.97,0) .. controls (-6,-0.3) .. (-7.03,0) node [pos=0.4,below=-1pt] {$\alpha_1$};%e_1
%\draw[color=red,->-=.55] (4.97,0) .. controls (6,-0.4) .. (7.03,0) node [pos=.3,below=-1pt] {$\alpha_2$}; %e_2
%\draw[color=red,->-=.55] (-7.03,0) .. controls (-6,-0.3) .. (-4.97,0)node [pos=0.4,below] {$\alpha_1$};%e_1 -
\draw[color=myblue, very thick] (7.03,0) .. controls (6,-0.3) .. (4.97,0) node [pos=.4,below=-1pt] {$\alpha_3$};%e_2 -

\draw[color=myblue, very thick] (-3.03,0) .. controls (-2,-0.3) .. (-.97,0) node [pos=.6,below=-1pt] {$\alpha_0$};%e_0+e_1
\draw[color=myblue, very thick] (3.03,0) .. controls (2,-0.3) .. (.97,0) node [pos=.4,below=-1pt] {$\alpha_2$};%e_0+e_2

%\draw[color=red] (-2,-3.03) .. controls (-1.7,-1.5) and (-1.7,1.5) .. (-2,3.03) node [pos=.25,right=-2pt] {$u$};
%\draw[color=red,my dash] (-2,-3.03) .. controls (-2.3,-1.5) and (-2.3,1.5) .. (-2,3.03);

\draw[color=red] (-3.293,-.707) .. controls (-1,-2.8) and (1,-2.8) .. (3.293,-.707) node [pos=.8,below=0pt] {$\delta_2$};
\draw[color=red] (-3.293,.707) .. controls (-1,2.2) and (1,2.2) .. (3.293,.707)node [pos=.8,below=-1pt] {$\delta_1$};

\node[] at (0,-5) {\textit{a}};
\end{scope}

\begin{scope}[shift={(9,0)}]
% Invisible lines

\draw[color=myblue,my dash, very thick] (-7.03,0) .. controls (-6,0.3) .. (-4.97,0); %e_1
\draw[color=myblue,my dash, very thick] (7.03,0) .. controls (6,0.3) .. (4.97,0); %e_2
\draw[color=myblue,my dash, very thick] (-3.03,0) .. controls (-2,0.3) .. (-.97,0); % e_0+e_1
\draw[color=myblue,my dash, very thick] (3.03,0) .. controls (2,0.3) .. (.97,0); % e_0+e_2

% Border lines

\draw [very thick] (0,0) circle (1);
\draw [very thick] (4,0) circle (1);
\draw [very thick] (-4,0) circle (1);
\draw [very thick]  (-4,3)--(4,3);
\draw [very thick]  (-4,-3)--(4,-3);
\draw [very thick] (-4,3) arc (90:270:3); 
\draw [very thick] (4,3) arc (90:-90:3); 

% Visible lines 

\draw [color=red] (-2.5,0) arc (0:360:1.5)node [pos=.75, below=-1pt] {$\beta_1$} ;

\draw [color=red] (1.5,0) arc (0:360:1.5) node [pos=.85, right=-1pt] {$\beta_2$} ;
\draw [color=red] (0,1.8) arc (90:270:1.8) -- (4,-1.8) node [pos=.75,below=-1pt] {$\beta_3$} arc (-90:90:1.8) -- (0,1.8) ;

\draw[color=myblue, very thick] (-4.97,0) .. controls (-6,-0.3) .. (-7.03,0) node [pos=0.4,below=-1pt] {$\alpha_1$};%e_1
%\draw[color=red,->-=.55] (4.97,0) .. controls (6,-0.4) .. (7.03,0) node [pos=.3,below=-1pt] {$\alpha_2$}; %e_2
%\draw[color=red,->-=.55] (-7.03,0) .. controls (-6,-0.3) .. (-4.97,0)node [pos=0.4,below] {$\alpha_1$};%e_1 -
\draw[color=myblue, very thick] (7.03,0) .. controls (6,-0.3) .. (4.97,0) node [pos=.3,below=-1pt] {$\alpha_3$};%e_2 -

\draw[color=myblue, very thick] (-3.03,0) .. controls (-2,-0.3) .. (-.97,0) node [pos=.4,below=-1pt] {$\alpha_0$};%e_0+e_1
\draw[color=myblue, very thick] (3.03,0) .. controls (2,-0.3) .. (.97,0) node [pos=.35,below=-1pt] {$\alpha_2$};%e_0+e_2

%\draw[color=red] (-2,-3.03) .. controls (-1.7,-1.5) and (-1.7,1.5) .. (-2,3.03) node [pos=.25,right=-2pt] {$u$};
%\draw[color=red,my dash] (-2,-3.03) .. controls (-2.3,-1.5) and (-2.3,1.5) .. (-2,3.03);
\node[] at (0,-5) {\textit{b}};

\end{scope}

\end{tikzpicture}
\caption{Type 1 multicurve}\label{fig_psi}
\end{figure}
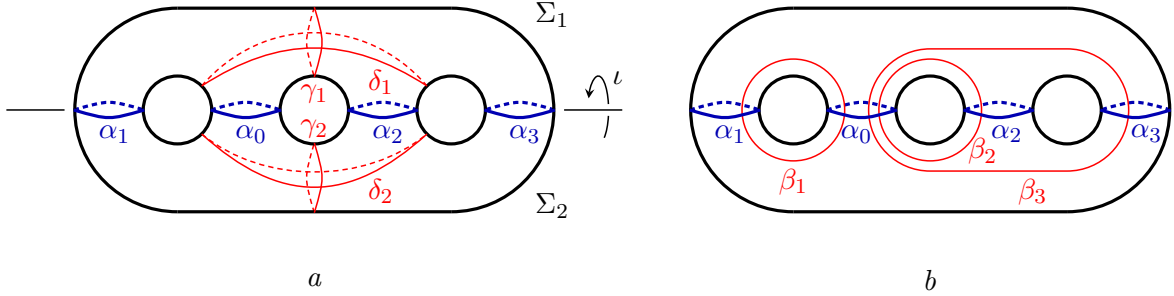

We denote by~$\theta_M$ the restriction of~$\theta_{\fA}$ to the stabilizer~$\I_M$. In this section we first prove that $\theta_M=0$, and then construct a homomorphism $\psi_M\colon \CC_M\to\Z/2$ originating in vanishing of~$\theta_M$. (Recall that we denote by~$\CC$ the kernel of the Birman--Craggs--Johnson homomorphism~$\sigma$, see Section~\ref{subsection_BC}.)

For each $s=1,2$, choose simple closed curves $\gamma_s$ and~$\delta_s$ in~$\Sigma_s$ so that the modulo~$2$ homology classes of~$\gamma_s$ and~$\delta_s$ are $\ba_2+\ba_3$ and~$\ba_1+\ba_3$, respectively, and the geometric intersection number of~$\gamma_s$ and~$\delta_s$ is~$2$, see Fig.~\ref{fig_psi}(a). (The reason why the components of~$M$ in Fig.~\ref{fig_psi} are arranged in a rather strange order is to achieve that $\gamma_s$ and~$\delta_s$ look nice.) Let $\bb_1$, $\bb_2$, and~$\bb_3$ be the modulo~$2$ homology classes of the curves~$\beta_1$, $\beta_2$, and~$\beta_3$ in Fig.~\ref{fig_psi}(b), respectively. Then   $\ba_1,\ba_2,\ba_3,\bb_1,\bb_2,\bb_3$ is a symplectic basis of~$H_{\Z/2}$. Number the quadratic forms~$\omega_i$ and the homomorphisms~$\rho_i=\rho_{\omega_i}$ so that equations~\eqref{eq_omega_numeration} hold.

Either of the surfaces $\Sigma_s$ is homeomorphic to  a sphere with four punctures. Hence $\PMod(\Sigma_s)$ is the free group~$\F_2$ with two generators $u_s=T_{\gamma_s}$ and $v_s=T_{\delta_s}$. Therefore $\PMod(S\setminus M)=\F_2\times\F_2$, where generators of the first factor are $u_1$ and~$v_1$ and generators of the second factor are $u_2$ and~$v_2$. Since $M$ does not contain a bounding pair, the natural homomorphism 
$\eta \colon \I_M\to\F_2\times\F_2$ in the Birman--Lubotzky--McCarthy exact sequence~\eqref{eq_BLM} is injective.

{\sloppy
\begin{propos}\label{propos_StabIM12}
The following three subgroups of~$\F_2\times\F_2$ coincide with each other:
\begin{itemize}
\item $\eta(\I_M)$,
\item $\ker f$, where   $f\colon \F_2\times\F_2\to\Z\times\Z$  is the homomorphism sending the generators~$u_1$ and~$u_2$ to~$(1,0)$ and the generators~$v_1$ and~$v_2$ to~$(0,1)$,
\item the normal closure of the two elements~$z_1=u_1u_2^{-1}$ and~$z_2=v_1v_2^{-1}$,
\item the subgroup generated by the three elements $z_1=u_1u_2^{-1}$, $z_2=v_1v_2^{-1}$, and $z_3=[u_1,v_1]$.
\end{itemize}
\end{propos}
}

\begin{proof}
The proof of the fact that $\ker f$ coincides with the normal closure of~$z_1$ and~$z_2$ is straightforward. 

Let $K$ be the subgroup of~$\F_2\times\F_2$ generated by~$z_1$, $z_2$, and~$z_3$. Obviously, $K\subseteq \ker f$. Let us prove the opposite inclusion. We put 
$z_4=[u_2,v_2]$. Then $z_4\in K$, since  
\begin{equation}\label{eq_z4}
z_4=z_1z_2z_3^{-1}z_1^{-1}z_2^{-1}.
\end{equation}

Consider an element $w\in \ker f$ and write it as a word~$W$ in letters~$u_1^{\pm 1}$, $v_1^{\pm 1}$, $u_2^{\pm 1}$, and~$v_2^{\pm 1}$. Let $k$ be the algebraic number of occurrences of~$u_1$ in~$W$, i.\,e. the number of occurrences of~$u_1$ minus the number of occurrences of~$u_1^{-1}$. Similarly, let $l$ be the algebraic number of occurrences of~$v_1$ in~$W$. Since $f(w)=0$, we see that the algebraic numbers of occurrences of~$u_2$ and~$v_2$ in~$W$ are equal to~$-k$ and~$-l$, respectively. Obviously, the numbers~$k$ and~$l$ are independent of the choice of a word~$W$ representing~$w$. It is easy to see that $w$ can be written in the form
$$
w=z_1^kz_2^lw_1w_2,
$$
where $w_1$ lies in the commutator subgroup of the first factor~$\F_2$ and $w_2$ lies in the commutator subgroup of the second factor~$\F_2$. It is a standard fact that these commutator subgroups are generated by the commutators~$[u_1^m,v_1]$ and~$[u_2^m,v_2]$, $m\in\Z$, respectively. So to show that $\ker f\subseteq K$ it is sufficient to prove that all commutators~$[u_1^m,v_1]$ and~$[u_2^m,v_2]$ lie in~$K$. We have
$$
[u_1^m,v_1]=u_1^{-1}[u_1^{m-1},v_1]u_1[u_1,v_1]=z_1^{-1}[u_1^{m-1},v_1]z_1z_3.
$$
Therefore, $[u_1^m,v_1]=z_1^{-m}(z_1z_3)^m$ and, similarly,  $[u_2^m,v_2]=z_2^{-m}(z_2z_4)^m$. Thus, $\ker f=K$.

Finally, let us prove that $\eta(\I_M)=K$.

The elements $z_1=T_{\gamma_1,\gamma_2}$ and $z_2=T_{\delta_1,\delta_2}$ are twists about bounding pairs disjoint from~$M$, and the element $z_3=[T_{\gamma_1},T_{\delta_1}]$ is the commutator of a simply intersecting pair. Hence $z_1,z_2,z_3\in\eta(\I_M)$ and therefore $K\subseteq\eta(\I_M)$.

Let us prove the opposite inclusion. Let $a_1,a_2,a_3,b_1,b_2,b_3$ be the integral homology classes of the curves $\alpha_1,\alpha_2,\alpha_3,\beta_1,\beta_2,\beta_3$, respectively, with orientations chosen so that $[\gamma_1]=[\gamma_2]=a_2+a_3$, $[\delta_1]=[\delta_2]=a_1+a_3$ and $a_i\cdot b_i=1$, $i=1,2,3$. 

The Birman--Lubotzky--McCarthy short exact sequence~\eqref{eq_BLM} reads as 
$$
1\to G(M)\to \Stab_{\Mod(S)}\bigl(\overrightarrow{M}\bigr)\xrightarrow{\eta}\F_2\times\F_2\to 1.
$$
The action of an element~$h\in\Stab_{\Mod(S)}\bigl(\overrightarrow{M}\bigr)$ on~$H$ is written in the  basis $a_1,a_2,a_3,b_1,b_2,b_3$ by a matrix 
$$
\begin{pmatrix}
I&-P\\
0&\phantom{-}I
\end{pmatrix},
$$
where $I$ is the unit $3\times 3$ matrix and $P=P(h)=(p_{ij})_{1\le i,j\le 3}$
is a symmetric $3\times 3$ integer matrix. Therefore, we obtain well-defined homomorphisms 
$$p_{ij}\colon  \Stab_{\Mod(S)}\bigl(\overrightarrow{M}\bigr)\to\Z, \qquad 1\le i\le j\le 3.$$
It is easy to see that the matrices $P$ for the twists $T_{\alpha_0}$, $T_{\alpha_1}$, $T_{\alpha_2}$, $T_{\alpha_3}$ are 
\begin{align*}
&
\begin{pmatrix}
1&1&1\\
1&1&1\\
1&1&1
\end{pmatrix},&
&
\begin{pmatrix}
1&0&0\\
0&0&0\\
0&0&0
\end{pmatrix},&
&
\begin{pmatrix}
0&0&0\\
0&1&0\\
0&0&0
\end{pmatrix},&
&
\begin{pmatrix}
0&0&0\\
0&0&0\\
0&0&1
\end{pmatrix},
\end{align*}
respectively. It follows that the homomorphisms $p_{23}-p_{12}$ and~$p_{13}-p_{12}$ vanish on~$G(M)$ and hence determine well-defined homomorphisms $\F_2\times\F_2\to\Z$. Further, it is easy to see that 
\begin{align*}
P(T_{\gamma_1})=P(T_{\gamma_2})&=
\begin{pmatrix}
0&0&0\\
0&1&1\\
0&1&1
\end{pmatrix},&
P(T_{\delta_1})=P(T_{\delta_2})&=
\begin{pmatrix}
1&0&1\\
0&0&0\\
1&0&1
\end{pmatrix}.
\end{align*}
Therefore, $f_1\circ\eta=p_{23}-p_{12}$ and $f_2\circ\eta=p_{13}-p_{12}$, where $f_1$ and~$f_2$ are the coordinates of the homomorphism~$f$. Since all homomorphisms $p_{ij}$ obviously vanish on $\I_M$, we obtain that $f_1$ and~$f_2$ vanish on~$\eta(\I_M)$. Thus, $\eta(\I_M)\subseteq K$.
\end{proof}

In the sequel, we conveniently identify~$\I_M$ with its image in~$\F_2\times\F_2$ under~$\eta$.

Before proceeding with studying restrictions of the Birman--Craggs homomorphism to~$\I_M$, let us mention the following easy consequence of Propositions~\ref{propos_StabIM12} and~\ref{propos_phi_values}.

\begin{propos}\label{propos_triv_restr2}
The restrictions of the homomorphisms~$\nu_{\alpha_i}$ to the group~$\I_M$ are trivial, $i=0,1,2,3$.
\end{propos}

\begin{proof}
By Proposition~\ref{propos_phi_values} every homomorphism~$\nu_{\alpha_i}$ takes zero value on each of the bounding pair twists~$z_1$, $z_2$, and~$z_1z_3=T_{\gamma_1,\gamma_2}^{v_1}$, which generate~$\I_M$.
\end{proof}

\begin{propos}
We have, 
\begin{gather}
\rho_0(z_1)=\rho_1(z_1)=0,\qquad \rho_2(z_1)=\rho_3(z_1)=1,\label{eq_rho_z1}\\
\rho_0(z_2)=\rho_2(z_2)=0,\qquad \rho_1(z_2)=\rho_3(z_2)=1,\label{eq_rho_z2}\\
\rho_0(z_3)=\rho_1(z_3)=\rho_2(z_3)=\rho_3(z_3)=1.\label{eq_rho_z3}
\end{gather}
\end{propos}

\begin{proof}
Applying formula~\eqref{eq_BCJ_BP} to the bounding pair twists~$T_{\gamma_1,\gamma_2}$, $T_{\delta_1,\delta_2}$, and~$T_{\gamma_1',\gamma_2}$, where $\gamma_1'=T_{\delta_1}^{-1}(\gamma_1)$, we get:
\begin{align*}
\sigma(z_1)&=\sigma(T_{\gamma_1,\gamma_2})=\left(
\overline{\ba}_2+\overline{\ba}_3+1
\right)\overline{\ba}_1\overline{\bb}_1,\\
\sigma(z_2)&=\sigma(T_{\delta_1,\delta_2})=\left(
\overline{\ba}_1+\overline{\ba}_3+1
\right)\overline{\ba}_2\overline{\bb}_2,\\
\sigma(z_1z_3)&=\sigma(T_{\gamma_1',\gamma_2})=\left(
\overline{\ba}_2+\overline{\ba}_3+1
\right)\overline{\ba}_1\left(\overline{\bb}_1+\overline{\ba}_3\right).
\end{align*}
Therefore,
\begin{equation*}
\sigma(z_3)=\overline{\ba}_1\overline{\ba}_2\overline{\ba}_3.
\end{equation*}
Combining this with~\eqref{eq_omega_numeration}, we obtain the required formulae~\eqref{eq_rho_z1}--\eqref{eq_rho_z3}.
\end{proof}

For $i=0,1,2,3$, denote by~$\rho_{i,M}$ the restriction of the homomorphism~$\rho_i$ to~$\I_M$.

\begin{cor}\label{cor_StabC}
The homomorphisms $\rho_{i,M}\colon\I_M\to\Z/2$ satisfy $\rho_{0,M}+\rho_{1,M}+\rho_{2,M}+\rho_{3,M}=0$ and satisfy no other linear relations over~$\Z/2$. The subgroup $\CC_M\subset\I_M$ has index~$8$ and coincides with the intersection of kernels of the homomorphisms~$\rho_{0,M}$, $\rho_{1,M}$, $\rho_{2,M}$, and~$\rho_{3,M}$. Besides, $\CC_M$ coincides with the kernel of the natural homomorphism $\I_M\to H_1(\I_M;\Z/2)$.
\end{cor}
\begin{proof}
By Proposition~\ref{propos_StabIM12} the group~$\I_M$ is generated by $z_1$, $z_2$, and~$z_3$. So it follows from~\eqref{eq_rho_z1}--\eqref{eq_rho_z3} that the homomorphisms~$\rho_{0,M}$, $\rho_{1,M}$, and~$\rho_{2,M}$ are linearly independent and $\rho_{0,M}+\rho_{1,M}+\rho_{2,M}+\rho_{3,M}=0$. Hence, the intersection of kernels of~$\rho_{0,M}$, $\rho_{1,M}$, $\rho_{2,M}$, and~$\rho_{3,M}$ has index~$8$. On the other hand, we have
\begin{equation}\label{eq_two_incl}
\bigcap_{i=0}^3\ker\rho_{i,M}\supseteq\CC_M\supseteq\ker\left[
\I_M\to H_1(\I_M;\Z/2)
\right].
\end{equation}
However, since the group~$\I_M$ is generated by $3$ elements, we have $\dim_{\Z/2} H_1(\I_M;\Z/2)\le 3$. Therefore, the index of the kernel of the natural homomorphism $\I_M\to H_1(\I_M;\Z/2)$ does not exceed~$8$. Thus, this index is equal to~$8$ and both inclusions in~\eqref{eq_two_incl} are equalities.
\end{proof}

By Proposition~\ref{propos_StabIM12}, we have a short exact sequence
\begin{equation*}
1\to \I_M\xrightarrow{\eta}\F_2\times\F_2\xrightarrow{f}\Z\times\Z\to 1.
\end{equation*}
For $i=1,2$, let $\xi_i\colon \I_M\to\Z$ be the composition of homomorphisms
$$
\I_M\xrightarrow{\eta}\F_2\times\F_2\xrightarrow{\pr_1}\F_2\xrightarrow{\zeta_i}\Z,
$$
where $\pr_1$ is the projection onto the first factor, $\zeta_1$ sends the generators~$u_1$ and~$v_1$ to~$1$ and~$0$, respectively, and $\zeta_2$ sends~$u_1$ and~$v_1$ to~$0$ and~$1$, respectively. Obviously, 
\begin{equation}\label{eq_xi_z}
\begin{aligned}
\xi_1(z_1)&=1,&
\xi_1(z_2)&=0,&
\xi_1(z_3)&=0,\\
\xi_2(z_1)&=0,&
\xi_2(z_2)&=1,&
\xi_2(z_3)&=0.
\end{aligned}
\end{equation}

Before proceeding with studying homomorphisms~$\rho_{i,M}$, let us make the following easy observation, which will be used in Section~\ref{subsection_b}.

\begin{propos}\label{propos_xi_eps}
Suppose that $\{\varepsilon_1,\varepsilon_2\}$ is a bounding pair that is disjoint from~$M$, and besides, $\varepsilon_1\subset\Sigma_1$ and $\varepsilon_2\subset\Sigma_2$. Then 
\begin{itemize}
\item $\xi_1(T_{\varepsilon_1,\varepsilon_2})=1$ and $\xi_2(T_{\varepsilon_1,\varepsilon_2})=0$, provided that $\varepsilon_1\cup\varepsilon_2$ separates $\alpha_0\cup\alpha_1$ from~$\alpha_2\cup\alpha_3$,
\item $\xi_1(T_{\varepsilon_1,\varepsilon_2})=0$ and $\xi_2(T_{\varepsilon_1,\varepsilon_2})=1$, provided that $\varepsilon_1\cup\varepsilon_2$ separates $\alpha_0\cup\alpha_2$ from~$\alpha_1\cup\alpha_3$,
\item $\xi_1(T_{\varepsilon_1,\varepsilon_2})=\xi_2(T_{\varepsilon_1,\varepsilon_2})=-1$, provided that $\varepsilon_1\cup\varepsilon_2$ separates $\alpha_0\cup\alpha_3$ from~$\alpha_1\cup\alpha_2$.
\end{itemize}
\end{propos}

\begin{proof}
The first two assertions follow from~\eqref{eq_xi_z}, since every twist~$T_{\varepsilon_1,\varepsilon_2}$ such that $\varepsilon_1\cup\varepsilon_2$ separates $\alpha_0\cup\alpha_1$ from~$\alpha_2\cup\alpha_3$ (respectively, $\alpha_0\cup\alpha_2$ from~$\alpha_1\cup\alpha_3$) is conjugate in~$\F_2\times\F_2$ to~$z_1$ (respectively, to~$z_2$). Similarly, it is sufficient to prove the third assertion for one particular twist~$T_{\varepsilon_1,\varepsilon_2}$ such that $\varepsilon_1\cup\varepsilon_2$ separates $\alpha_0\cup\alpha_1$ from~$\alpha_2\cup\alpha_3$, since all such twists are conjugate two each other. Let us consider the particular curves~$\varepsilon_1$ and~$\varepsilon_2$ shown in Fig.~\ref{fig_lantern_type1}. 

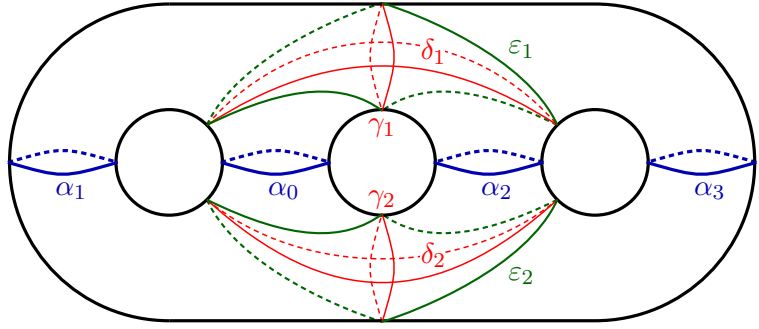
\begin{figure}
\begin{tikzpicture}[scale=.7]
\small
\definecolor{myblue}{rgb}{0, 0, 0.7}
\definecolor{mygreen}{rgb}{0, 0.4, 0}
\tikzset{every path/.append style={line width=.2mm}}
\tikzset{my dash/.style={dash pattern=on 2pt off 1.5pt}}

\begin{scope}
% Invisible lines

\draw[red,my dash] (0,-3.03) .. controls (-0.3,-2) .. (0,-.97); % e_0
\draw[red,my dash] (0,3.03) .. controls (-0.3,2) .. (0,.97); 

\begin{scope}
 {
  \tikzset{every path/.style={}}
  \clip (-4,0) rectangle (.7,3);
  }
  \draw[color=red,my dash] (-3.293,.707) .. controls (-1.5,2.8) and (1.5,2.8) .. (3.293,.707);
\end{scope}

\begin{scope}
 {
  \tikzset{every path/.style={}}
  \clip (1.2,3) rectangle (4,0);
  }
  \draw[color=red,my dash] (-3.293,.707) .. controls (-1.5,2.8) and (1.5,2.8) .. (3.293,.707);
\end{scope}

\begin{scope}
 {
  \tikzset{every path/.style={}}
  \clip (-4,0) rectangle (.7,-3);
  }
  \draw[color=red,my dash] (-3.293,-.707) .. controls (-1.5,-2.2) and (1.5,-2.2) .. (3.293,-.707);
\end{scope}

\begin{scope}
 {
  \tikzset{every path/.style={}}
  \clip (1.2,-3) rectangle (4,0);
  }
  \draw[color=red,my dash] (-3.293,-.707) .. controls (-1.5,-2.2) and (1.5,-2.2) .. (3.293,-.707);

\end{scope}

\draw[color=myblue,my dash, very thick] (-7.03,0) .. controls (-6,0.3) .. (-4.97,0); %e_1
\draw[color=myblue,my dash, very thick] (7.03,0) .. controls (6,0.3) .. (4.97,0); %e_2
\draw[color=myblue,my dash, very thick] (-3.03,0) .. controls (-2,0.3) .. (-.97,0); % e_0+e_1
\draw[color=myblue,my dash, very thick] (3.03,0) .. controls (2,0.3) .. (.97,0); % e_0+e_2

\draw[color=mygreen, thick, my dash] (3.293,.707) .. controls (1.7,1.5) and (.7,1.5) .. (0,.98);

\draw[color=mygreen, thick, my dash] (-3.293,.707) .. controls (-2.7,2) and (-.7,2.8) .. (0,3.02);

\draw[color=mygreen, thick, my dash] (3.293,-.707) .. controls (1.7,-1.5) and (.7,-1.5) .. (0,-.98);

\draw[color=mygreen, thick, my dash] (-3.293,-.707) .. controls (-2.7,-2) and (-.7,-2.8) .. (0,-3.02);

% Border lines

\draw [very thick] (0,0) circle (1);
\draw [very thick] (4,0) circle (1);
\draw [very thick] (-4,0) circle (1);
\draw [very thick]  (-4,3)--(4,3);
\draw [very thick]  (-4,-3)--(4,-3);
\draw [very thick] (-4,3) arc (90:270:3); 
\draw [very thick] (4,3) arc (90:-90:3); 
%\draw [] (-9,0) -- (-7.3,0);
%\draw [] (9,0) -- (7.3,0);

%\draw (8.3,0) + (-55: .3 and 1) arc (-55: -10: .3 and 1);
%\draw [-stealth] (8.3,0) + (10: .3 and 1) arc (10: 160: .3 and 1) node [pos=.4, right] {$\iota$};

% Visible lines 

\draw[color=red] (0,-3.03) .. controls (0.3,-2) .. (0,-.97) node [above=-2pt] {$\gamma_2$};  %e_0
\draw[color=red] (0,3.03) .. controls (0.3,2) .. (0,.97) node [below=-1pt] {$\gamma_1$};  %e_0'

\draw[color=myblue, very thick] (-4.97,0) .. controls (-6,-0.3) .. (-7.03,0) node [pos=0.4,below=-1pt] {$\alpha_1$};%e_1
%\draw[color=red,->-=.55] (4.97,0) .. controls (6,-0.4) .. (7.03,0) node [pos=.3,below=-1pt] {$\alpha_2$}; %e_2
%\draw[color=red,->-=.55] (-7.03,0) .. controls (-6,-0.3) .. (-4.97,0)node [pos=0.4,below] {$\alpha_1$};%e_1 -
\draw[color=myblue, very thick] (7.03,0) .. controls (6,-0.3) .. (4.97,0) node [pos=.4,below=-1pt] {$\alpha_3$};%e_2 -

\draw[color=myblue, very thick] (-3.03,0) .. controls (-2,-0.3) .. (-.97,0) node [pos=.6,below=-1pt] {$\alpha_0$};%e_0+e_1
\draw[color=myblue, very thick] (3.03,0) .. controls (2,-0.3) .. (.97,0) node [pos=.4,below=-1pt] {$\alpha_2$};%e_0+e_2

%\draw[color=red] (-2,-3.03) .. controls (-1.7,-1.5) and (-1.7,1.5) .. (-2,3.03) node [pos=.25,right=-2pt] {$u$};
%\draw[color=red,my dash] (-2,-3.03) .. controls (-2.3,-1.5) and (-2.3,1.5) .. (-2,3.03);

\draw[color=red] (-3.293,-.707) .. controls (-1,-2.8) and (1,-2.8) .. (3.293,-.707) node [pos=.65,above=-1pt] {$\delta_2$};
\draw[color=red] (-3.293,.707) .. controls (-1,2.2) and (1,2.2) .. (3.293,.707)node [pos=.65,above = -1pt] {$\delta_1$};

\draw[color=mygreen, thick] (-3.293,.707) .. controls (-1.7,1.5) and (-.7,1.5) .. (0,.98);

\draw[color=mygreen, thick] (3.293,.707) .. controls (2.7,2) and (.7,2.8) .. (0,3.02) node [pos=.35, above right = -2pt] {$\varepsilon_1$};

\draw[color=mygreen, thick] (-3.293,-.707) .. controls (-1.7,-1.5) and (-.7,-1.5) .. (0,-.98);

\draw[color=mygreen, thick] (3.293,-.707) .. controls (2.7,-2) and (.7,-2.8) .. (0,-3.02) node [pos=.35, below right = -2pt] {$\varepsilon_2$};

\end{scope}

\end{tikzpicture}
\caption{Bounding pair $\{\varepsilon_1,\varepsilon_2\}$}\label{fig_lantern_type1}

\end{figure}

Note that the $7$-tuples $\{\gamma_s,\delta_s,\varepsilon_s,\alpha_0,\alpha_1,\alpha_2,\alpha_3\}$, $i=1,2$, are lantern configurations (with opposite orientations), cf.~Fig.~\ref{fig_lantern}. The corresponding lantern relations read as
$$
T_{\gamma_1}T_{\delta_1}T_{\varepsilon_1}=T_{\alpha_0}T_{\alpha_1}T_{\alpha_2}T_{\alpha_3}=T_{\delta_2}T_{\gamma_2}T_{\varepsilon_2}.
$$ 
Therefore $T_{\varepsilon_1,\varepsilon_2}=z_2^{-1}z_1^{-1}$, hence the required equalities.
\end{proof}

Combining~\eqref{eq_xi_z} and~\eqref{eq_rho_z1}--\eqref{eq_rho_z3}, we obtain the following proposition.

\begin{propos}\label{propos_liftrho}
For all $h\in\I_M$, 
\begin{align*}
\rho_0(h)+\rho_1(h)&=\xi_2(h) \mod 2,\\
\rho_0(h)+\rho_2(h)&=\xi_1(h) \mod 2.
\end{align*}
\end{propos}

\begin{cor}\label{cor_rho_sq}
$\rho_{0,M}^2=\rho_{1,M}^2=\rho_{2,M}^2=\rho_{3,M}^2$ in $H^2(\I_M;\Z/2)$.
\end{cor}

\begin{proof}
Recall that, for any $1$-dimensional modulo~$2$ cohomology class~$z$ of any space or group, we have $z^2=\beta z$, where $\beta$ is the Bockstein homomorphism associated with the short exact sequence of coefficient groups $0\to\Z/2\to\Z/4\to\Z/2\to 0$. Hence, $z^2=0$ whenever $z$ lifts to a cohomology class with coefficients in~$\Z/4$, so in particular whenever it lifts to an integral cohomology class. Therefore, by Proposition~\ref{propos_liftrho} we have $(\rho_{0,M}+\rho_{i,M})^2=0$ for $i=1,2$. Thus, $\rho_{0,M}^2=\rho_{1,M}^2=\rho_{2,M}^2$. By Corollary~\ref{cor_StabC} we finally get $\rho_{3,M}^2=\rho_{0,M}^2$.
\end{proof}

Now, we are ready to prove Proposition~\ref{propos_theta_main}(b) for type~1 multicurves. 

\begin{propos}\label{propos_theta_main1} 
$\theta_M=0$.
\end{propos}

\begin{proof}
Using Corollaries~\ref{cor_StabC} and~\ref{cor_rho_sq} and Proposition~\ref{propos_liftrho}, we obtain that
\begin{align*}
\theta_M=\sum_{i<j}\rho_{i,M}\rho_{j,M}&=\rho_{0,M}\rho_{1,M} + \rho_{0,M}\rho_{2,M} + \rho_{1,M}\rho_{2,M} + \rho_{3,M}^2\\
{}&=\rho_{0,M}\rho_{1,M} + \rho_{0,M}\rho_{2,M} + \rho_{1,M}\rho_{2,M} + \rho_{0,M}^2\phantom{\sum_{i<j}}\\
{}&=(\rho_{0,M}+\rho_{1,M}) (\rho_{0,M}+\rho_{2,M})= \xi_1\xi_2\mod 2. 
\end{align*}
But
$$
\xi_1\xi_2=(\pr_1\circ\eta)^*(\zeta_1\zeta_2)=0,
$$
since $H^2(\F_2;\Z)=0$. Thus, $\theta_M=0$.
\end{proof}

\begin{propos}\label{propos_psi1}
There exists a unique homomorphism $\psi_M\colon \CC_M\to\Z/2$ such that, for all $h_1,h_2\in\I_M$,
\begin{equation}
\label{eq_psi1}\psi_M\bigl([h_1,h_2]\bigr)=\sum_{i\ne j}\rho_i(h_1)\rho_j(h_2),
\end{equation}
and, for all $h\in\I_M$,
\begin{equation}
\label{eq_psi2}\psi_M(h^2)=\sum_{i<j}\rho_i(h)\rho_j(h).
\end{equation}
Besides, $\psi_M$ is invariant under the action of~$\Stab_{\Mod(S)}(M)$ on~$\CC_M$ by conjugations.
\end{propos}
 
\begin{proof}
 By Corollary~\ref{cor_StabC} the subgroup~$\CC_M$ coincides with the kernel of the natural homomorphism $j\colon \I_M\to H_1(\I_M;\Z/2)$; hence, $\CC_M$ is generated by commutators~$[h_1,h_2]$, $h_1,h_2\in\I_M$, and squares~$h^2$ of elements $h\in\I_M$. Therefore, a homomorphism~$\psi_M$ satisfying~\eqref{eq_psi1} and~\eqref{eq_psi2} is unique if exists. Moreover, the invariance of~$\psi_M$ under the action of~$\Stab_{\Mod(S)}(M)$ follows immediately from the uniqueness of it. By Proposition~\ref{propos_d2_LHS}, the existence of~$\psi_M$ satisfying~\eqref{eq_psi1} and~\eqref{eq_psi2} follows from vanishing of~$\theta_M$.
\end{proof}

Let $\iota$ be the rotation of the surface shown in Fig.~\ref{fig_psi} by angle~$\pi$ around the horizontal axis. Then $\iota\in\hI_M\setminus\I_M$ and $\iota^2=1$. 
Hence, $\I_M$ is an index two subgroup of~$\hI_M$. By Proposition~\ref{propos_eBC},  $[h_1,h_2]\in\CC_M$ for all $h_1,h_2\in\hI_M$ and $h^2\in\CC_M$ for all $h\in\hI_M$. 
In the sequel, we will need the following improvement of Proposition~\ref{propos_psi1}.

\begin{propos}\label{propos_psi2}
For all $h_1,h_2\in\hI_M$,
\begin{equation}
\label{eq_psi1h}\psi_M\bigl([h_1,h_2]\bigr)=\sum_{i\ne j}\hrho_i(h_1)\hrho_j(h_2),
\end{equation}
and, for all $h\in\hI_M$,
\begin{equation}
\label{eq_psi2h}\psi_M(h^2)=\sum_{i<j}\hrho_i(h)\hrho_j(h).
\end{equation}
\end{propos}

\begin{proof}
For the elements $z_1$, $z_2$, and $z_3$ in Proposition~\ref{propos_StabIM12} and the  element $z_4$ in the proof of those proposition, we have
$$
\iota z_1\iota=z_1^{-1},\qquad
\iota z_2\iota=z_2^{-1},\qquad
\iota z_3\iota=z_4.
$$
Using~\eqref{eq_eBC3}, we obtain that
$$
\hsigma(\iota)=\overline{\ba}_1\bigl(\overline{\ba}_2+\overline{\ba}_3+1\bigr)\overline{\bb}_1\overline{\bb}_2
$$
and hence
\begin{equation}\label{eq_hrho_iota}
\hrho_0(\iota)=\hrho_1(\iota)=\hrho_2(\iota)=0,\qquad \hrho_3(\iota)=1.
\end{equation}
Since the group~$\I_M$ is generated by~$z_1$, $z_2$, and~$z_3$, we see that the group~$\hI_M$ is generated by~$z_1$, $z_2$, $z_3$, and~$\iota$.

Suppose that $h_1=h'h''$, where $h',h''\in\hI_M$; then
$$
[h_1,h_2]=[h',h_2]^{h''}[h'',h_2].
$$ 
Since the homomorphism~$\psi_M$ is $\hI_M$-invariant,  equality~\eqref{eq_psi1h} for $[h_1,h_2]$ would follow once the same equalities were proved for $[h',h_2]$ and~$[h'',h_2]$. Besides, $[h_1,h_2]^{-1}=[h_1^{-1},h_2]^{h_1}$ and hence equality~\eqref{eq_psi1h} for $[h_1,h_2]$ is equivalent to the same equality for~$[h_1^{-1},h_2]$. Thus, it is sufficient to prove the required equality~\eqref{eq_psi1h} for some set of elements~$h_1$ that generate the group~$\hI_M$, for instance, for $h_1\in\{z_1,z_2,z_3,\iota\}$. Repeating similar arguments for~$h_2$, we can further reduce the proof of~\eqref{eq_psi1h} to the case when both $h_1$ and~$h_2$ are in the set~$\{z_1,z_2,z_3,\iota\}$. If $h_1=h_2$, then equality~\eqref{eq_psi1h} is obvious. If both $h_1$ and~$h_2$ are in~$\{z_1,z_2,z_3\}$, then  equality~\eqref{eq_psi1h} holds by Proposition~\ref{propos_psi1}. Besides, equality~\eqref{eq_psi1h} for $[h_1,h_2]$ is equivalent to the same equality for $[h_2,h_1]$. So we need only to prove equality~\eqref{eq_psi1h} for $h_1=\iota$ and $h_2=z_s$, $s=1,2,3$. Combining~\eqref{eq_rho_z1}--\eqref{eq_rho_z3} and~\eqref{eq_hrho_iota}, we see that, for each commutator~$[\iota,z_s]$, $s=1,2,3$, the right-hand side of~\eqref{eq_psi1h} is equal to~$1$. So we need to prove that $\psi_M\bigl([\iota,z_s]\bigr)=1$.

If $s$ is either~$1$ or~$2$, then $[\iota,z_s]=z_s^2$. Combining~\eqref{eq_psi2}, \eqref{eq_rho_z1}, and~\eqref{eq_rho_z2}, we obtain that
$$
\psi_M\bigl([\iota,z_s]\bigr)=\sum_{i<j}\rho_i(z_s)\rho_j(z_s)=1.
$$

If $s=3$, then, using~\eqref{eq_z4}, we obtain that
$$
[\iota,z_3]=z_4^{-1}z_3=z_2z_1z_3z_2^{-1}z_1^{-1}z_3=\bigl[z_1^{-1}z_2^{-1},z_3^{-1}\bigr]z_3^2\bigl[z_2^{-1},z_1^{-1}\bigr]^{z_3}.
$$
Combining~\eqref{eq_psi1}, \eqref{eq_psi2} and~\eqref{eq_rho_z1}--\eqref{eq_rho_z3}, we compute that
$$
\psi_M\bigl([\iota,z_3]\bigr)=
\psi_M\bigl(\bigl[z_1^{-1}z_2^{-1},z_3^{-1}\bigr]\bigr)
+\psi_M\bigl(z_3^2\bigr)
+\psi_M\bigl(\bigl[z_2^{-1},z_1^{-1}\bigr]\bigr)=1,
$$
which completes the proof of~\eqref{eq_psi1h}.

Now, let us prove~\eqref{eq_psi2h}. If $h\in\I_M$, then this is exactly equality~\eqref{eq_psi2}. Suppose that $h\in\hI_M\setminus\I_M$. Then $h\iota\in\I_M$ and 
$$
h^2=(h\iota)^2[\iota,h].
$$
Using~\eqref{eq_psi2} and~\eqref{eq_psi1h}, we obtain that
\begin{multline*}
\psi_M\left(h^2\right)=\psi_M\left((h\iota)^2\right)+\psi_M\bigl( [\iota,h]\bigr)=\sum_{i<j}\hrho_i(h\iota)\hrho_j(h\iota)+\sum_{i\ne j}\hrho_i(\iota)\hrho_j(h)\\{}=\sum_{i<j}\hrho_i(h)\hrho_j(h)+\sum_{i<j}\hrho_i(\iota)\hrho_j(\iota)=\sum_{i<j}\hrho_i(h)\hrho_j(h),
\end{multline*}
which is exactly the required equality~\eqref{eq_psi2h}.
\end{proof}

\subsection{Five-component multicurves} 
In this section we study the stabilizer~$\I_K$ for a five-component multicurve~$K$ that belongs to~$\M_2'$. In order to use notation in the previous subsection, we conveniently arrange~$K$ as the multicurve $\alpha_0\cup\alpha_1\cup\alpha_2\cup\alpha_3\cup\gamma_2$ in Fig.~\ref{fig_psi}.  As in the previous section, we denote by~$M$ the four-component multicurve $\alpha_0\cup\alpha_1\cup\alpha_2\cup\alpha_3$.

\begin{propos}\label{propos_Stab_5comp}
The stabilizer~$\I_K$ is the free group with infinitely many generators
$$
w_k=T_{\delta_1}^{k}T_{\gamma_1,\gamma_2}T_{\delta_1}^{-k},\qquad k\in\Z.
$$
\end{propos}

\begin{proof}
Since one connected components of~$S\setminus K$ is a 4-punctured sphere, and two others are 3-punctured spheres, we see that $\PMod(S\setminus K)\cong\F_2$ is the free group with two generators~$T_{\gamma_1}$ and~$T_{\delta_1}$. Similarly to Proposition~\ref{propos_StabIM12}, we derive that the image of the homomorphism~$\eta\colon \I_K\to \F_2$ in the corresponding Birman--Lubotzky--McCarthy exact sequence  is the kernel  of the homomorphism $\F_2\to\Z$ that sends~$T_{\gamma_1}$ and~$T_{\delta_1}$ to~$0$ and~$1$, respectively. The proposition follows easily.
\end{proof}
 
\begin{propos}
The group~$\CC_K$ coincides with the normal closure in~$\PMod(S\setminus K)$ of the three elements~$w_0^2$, $w_2w_0^{-1}$, and~$[w_0,w_1]$.
\end{propos}
 
\begin{proof}
Using~\eqref{eq_BCJ_BP}, we compute that 
\begin{equation}\label{eq_sigma_w_value}
\sigma(w_k)=\bigl(\overline{\ba}_2+\overline{\ba}_3+1\bigr)\overline{\ba}_1\overline{\bb}_1+k\,\overline{\ba}_1\overline{\ba}_2\overline{\ba}_3.
\end{equation}
Since the elements $\bigl(\overline{\ba}_2+\overline{\ba}_3+1\bigr)\overline{\ba}_1\overline{\bb}_1$ and~$\overline{\ba}_1\overline{\ba}_2\overline{\ba}_3$ are linearly independent in~$\BB_3'$, we obtain that $\CC_K$ has index~$4$ in~$\I_K$. Besides, $\{1, w_0, w_1, w_0w_1\}$ is the set of representatives for the cosets of~$\I_K$ by~$\CC_K$. Using Reidemeister rewriting process (cf.~\cite[Theorem~2.7]{MKS}), we obtain that the group~$\CC_K$ is generated by the following elements:
\begin{align*}
&w_{2k}w_0^{-1}, &
&w_0w_{2k}, &
&w_1w_{2k}w_1^{-1}w_0^{-1}, &
&w_0w_1w_{2k}w_1^{-1}, \\
&w_{2k+1}w_1^{-1},& 
&w_0w_{2k+1}w_1^{-1}w_0^{-1},&
&w_1w_{2k+1},&
&w_0w_1w_{2k+1}w_0^{-1},
\end{align*}
where $k$ runs over~$\Z$. Each of these elements can be easily expressed through the elements~$w_0^2$, $w_1^2$, $[w_0,w_1]$, $w_{m}w_{m-2}^{-1}$, where $m\in\Z$, and their conjugates. Since in the group~$\PMod(S\setminus K)$ the element~$w_1^2$ is conjugate to~$w_0^2$ and every element~$w_{m}w_{m-2}^{-1}$ is conjugate to~$w_2w_0^{-1}$, the proposition follows.
\end{proof}
 
Now, we are able to prove the following result, which relates the homomorphism~$\psi_M$ to the homomorphisms~$\nu_{\gamma}$ introduced in Section~\ref{section_nu}. 

\begin{propos}\label{propos_psi_nu}
Suppose that $h\in\CC_K$. Then the number $\nu_{\gamma_2}(h)$ is even, and
\begin{equation}\label{eq_psi_nu}
\psi_M(h)=\frac{\nu_{\gamma_2}(h)}{2}\mod 2.
\end{equation}
\end{propos}

\begin{proof}
Since both homomorphisms~$\psi_M$ and~$\nu_{\gamma_2}$ are $\PMod(S\setminus K)$--invariant, we suffice to check equality~\eqref{eq_psi_nu} for the three elements~$w_0^2$, $w_2w_0^{-1}$, and~$[w_0,w_1]$. 

By Proposition~\ref{propos_phi_values} we have $\nu_{\gamma_2}(w_k)=1$ for all~$k\in\Z$. Hence
$$
\nu_{\gamma_2}\bigl(w_0^2\bigr)=2,\qquad \nu_{\gamma_2}\bigl(w_{2}w_0^{-1}\bigr)=\nu_{\gamma_2}\bigl([w_0,w_1]\bigr)=0.
$$

Combining~\eqref{eq_psi1},~\eqref{eq_psi2}, and~\eqref{eq_sigma_w_value}, we obtain that 
\begin{align}\label{eq_psi_wk}
\psi_M(w_0^2)&=1,\\
\psi_M\bigl([w_0,w_1]\bigr)&=0.\nonumber
\end{align}
To compute the value~$\psi_M(w_{2}w_0^{-1})$, we proceed as follows:
$$
w_{2}w_0^{-1}=T_{\delta_1}^{2}T_{\gamma_1,\gamma_2}T_{\delta_1}^{-2}T_{\gamma_1,\gamma_2}^{-1}=T_{\delta_1}^{2}T_{\gamma_1}T_{\delta_1}^{-2}T_{\gamma_1}^{-1}=T_{\delta_1,\delta_2}^{2}T_{\gamma_1}T_{\delta_1,\delta_2}^{-2}T_{\gamma_1}^{-1}.
$$
Similarly to~\eqref{eq_psi_wk}, we have
$$
\psi_M\bigl(T_{\delta_1,\delta_2}^2\bigr)=\psi_M\bigl(T_{\gamma_1}T_{\delta_1,\delta_2}^2T_{\gamma_1}^{-1}\bigr)=1.
$$
Hence $\psi_M(w_2w_0^{-1})=0$. Therefore, the required equality~\eqref{eq_psi_nu} holds for the elements~$w_0^2$, $w_2w_0^{-1}$, and~$[w_0,w_1]$; thus, for all elements of~$\CC_K$.
\end{proof}

\begin{propos}\label{propos_nu_W_w}
There are orthogonal splittings~$\CW_k$, $k\in\Z$, of the group $\langle [\gamma_2]\rangle^{\bot}/\langle [\gamma_2]\rangle$ such that 
\begin{equation}\label{eq_nu_W_w}
\nu_{\gamma_2,\CW_k}(w_m)=\left\{
\begin{aligned}
&1,&k&=m,\\
&0,&k&\ne m.
\end{aligned}
\right.
\end{equation}
Moreover, $\CW_k$ are exactly all orthogonal splittings~$U\oplus V$ of~$\langle [\gamma_2]\rangle^{\bot}/\langle [\gamma_2]\rangle$ satisfying $[\alpha_i]\in U\cup V$ for $i=0,1,2,3$. 
\end{propos}

\begin{remark}
Here with some abuse of notation we denote the image of~$[\alpha_i]$ after taking quotient by~$[\gamma_2]$ again by~$[\alpha_i]$.
The homology classes $[\gamma_2]$, $[\alpha_0],\ldots,[\alpha_3]$ are defined up to sign. After taking quotient by~$\langle [\gamma_2]\rangle$, we have $[\alpha_0]=\pm [\alpha_1]$ and~$[\alpha_2]=\pm[\alpha_3]$. So in fact the condition $[\alpha_i]\in U\cup V$ for $i=0,1,2,3$ means that $[\alpha_0]$ belongs to one of the summands and $[\alpha_2]$ belongs to the other.
\end{remark}

\begin{proof}[Proof of Proposition~\ref{propos_nu_W_w}]
Choose a symplectic basis $e_1,e_2,e_3,f_1,f_2,f_3$ of~$H$ so that 
$$e_1=\pm[\gamma_2],\quad e_2=\pm[\alpha_0],\quad e_3=\pm[\alpha_2],\quad e_2+e_3=\pm [\delta_1],$$ 
and there exist curves in homology classes~$f_2$ and~$f_3$ disjoint from $\gamma_1\cup\gamma_2$. Denote the images of the classes $e_2,e_3,f_2,f_3$ after taking quotient by~$\langle e_1\rangle$ again by the same letters. Denote by~$\CW_k$ the splitting
$$
\langle e_1\rangle^{\bot}/\langle e_1\rangle=\langle e_2,f_2-ke_3\rangle\oplus\langle e_3,f_3-ke_2\rangle.
$$
It is easy to see that $\CW_k$ are exactly all orthogonal splittings for which $e_2$ and~$e_3$ lie in the summands. Moreover, every $\CW_k$ coincides with the splitting provided by the bounding pair~$\bigl(T_{\delta_1}^k(\gamma_1),\gamma_2\bigr)$, hence formula~\eqref{eq_nu_W_w}.
\end{proof}

\begin{cor}\label{cor_injection_H1}
The inclusion of groups $\I_K\subset \I_{\gamma_2}$ induces an injective homomorphism $H_1(\I_K;\Z)\to H_1(\I_{\gamma_2};\Z)$.
\end{cor}

\begin{proof}
Since $\I_K$ is the free group with generators~$w_k$, $k\in\Z$, we see that $H_1(\I_K;\Z)$ is the free abelian group with generators~$[w_k]$. The required assertion follows immediately from Proposition~\ref{propos_nu_W_w}, since homomorphisms~$\nu_{\gamma_2,\CW_k}$ are defined on~$\I_{\gamma_2}$.
\end{proof}

\subsection{Type~2 four-component multicurves}\label{subsection_type2}

Suppose that $M$ is a type~$2$ four-component multicurve, see Fig.~\ref{fig_types}(c). Suppose that $\ba_1,\ba_2,\ba_3\in H_{\Z/2}$ are elements such that the modulo~$2$ homology classes of components of~$M$ are $\ba_1$, $\ba_2$, $\ba_1+\ba_2$, and~$\ba_3$. Put $\fA=\{\ba_1,\ba_2,\ba_3\}$, extend the set~$\fA$ to a symplectic basis $\ba_1,\ba_2,\ba_3,\bb_1,\bb_2,\bb_3$  of~$H_{\Z/2}$, and number the quadratic forms~$\omega_i$ and the homomorphisms~$\rho_i=\rho_{\omega_i}$ so that equations~\eqref{eq_omega_numeration} hold. Denote the components of~$M$ by $\alpha_0,\ldots,\alpha_3$ so that the homology class of~$\alpha_i$ is~$\ba_i$ for $i=1,2,3$, and the homology class of~$\alpha_0$ is $\ba_1+\ba_2$. 

\begin{propos}\label{propos_generate_type2}
The group~$\I_M$ is generated by twists~$T_{\alpha',\alpha}$, where $\alpha$ is one of the three curves~$\alpha_0$, $\alpha_1$, and~$\alpha_2$ and~$\alpha'$ is homologous to~$\alpha$ and disjoint from~$M$.
\end{propos}

\begin{proof}
It follows from Proposition~\ref{propos_stab_generate} that~$\I_M$ is generated by twists~$T_{\alpha',\alpha}$ described in the formulation of the proposition and twists~$T_{\delta}$ about separating curves disjoint from~$M$. For any of the latter twists~$T_{\delta}$, the curves~$\alpha_0$, $\alpha_1$, $\alpha_2$, and~$\delta$ bound a subsurface of~$S$ that is homeomorphic to a sphere with four boundary components. In this subsurface we can find simple closed curves~$\alpha_0'$, $\alpha_1'$, and~$\alpha_2'$ that together with~$\alpha_0$, $\alpha_1$, $\alpha_2$, and~$\delta$ form a lantern configuration,  see Fig.~\ref{fig_lantern_type2}. Then $T_{\alpha_0'}T_{\alpha_1'}T_{\alpha_2'}=T_{\alpha_0}T_{\alpha_1}T_{\alpha_2}T_{\delta}$. Since $\delta$ is separating, every~$\alpha_i'$ is homologous to~$\alpha_i$. Thus,
$T_{\delta}=T_{\alpha_0',\alpha_0}T_{\alpha_1',\alpha_1}T_{\alpha_2',\alpha_2}$
and the proposition follows.
\end{proof}

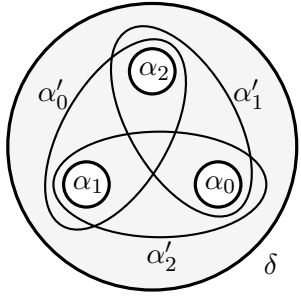
\begin{figure}
\tikzset{->-/.style={decoration={
     markings,
     mark=at position #1 with {\arrow{>}}},postaction={decorate}}}

  %  \draw[->-=.5] (0,0) to [bend left] (2,4);
    %\draw[->-=.8] (0,0) to [bend right] (2,4);

\definecolor{mygreen}{rgb}{0, 0.4, 0}
\definecolor{myblue}{rgb}{0, 0, 0.7}
\definecolor{mypink}{rgb}{1, 0.3, 0.8}
\definecolor{myyellow}{rgb}{1,1, 0}

\colorlet{cola}{violet}
\colorlet{cold}{red}
\colorlet{colg}{myblue}
\colorlet{colkf}{mypink}
\colorlet{colks}{violet}
\colorlet{coll}{orange}
\colorlet{colx}{mygreen}
\colorlet{coly}{orange}
\colorlet{colm}{mygreen}
\colorlet{colz}{brown}

\begin{tikzpicture}
\small
\tikzset{every path/.append style={line width=.3mm}}

\filldraw [fill=black!4, very thick] (0,0) circle (1.9); 
%\fill [color=] (0,0) circle (2);
\filldraw [fill=white, very thick] (0,1) circle (.3);
\filldraw [fill=white, very thick] (0.866,-0.5) circle (.3);
\filldraw [fill=white, very thick] (-0.866,-0.5) circle (.3);

\draw [thick] (0.1,-.5) ellipse (1.4 and .7);
\draw [thick, rotate=120] (0.1,-.5) ellipse (1.4 and .7);
\draw [thick, rotate=240] (0.1,-.5) ellipse (1.4 and .7);
\path (0.03,1) node {$\alpha_2$};
\path (-0.836,-0.5) node {$\alpha_1$};
\path (0.896,-0.5) node {$\alpha_0$};
\path (0.13,-1.45) node {$\alpha_2'$};
\path (1.27,0.7) node {$\alpha_1'$};
\path (-1.3,0.7) node {$\alpha_0'$};
\path (1.56,-1.56) node {$\delta$};

\end{tikzpicture}
\caption{Simple closed curves~$\alpha_0'$, $\alpha_1'$, and~$\alpha_2'$}\label{fig_lantern_type2}
\end{figure}

Note that, unlike in the case of type~1 multicurves, the cohomology class $\theta_{\fA}$ depends on which two of the three homology classes of components of~$M$ bounding a three-punctured sphere are taken for~$\ba_1$ and~$\ba_2$.
We denote by~$\rho_{0,M},\ldots,\rho_{3,M}$ the restrictions of $\rho_0,\ldots,\rho_3$ to~$\I_M$, respectively, and by~$\theta_{\fA,M}$ the restriction of~$\theta_{\fA}$ to~$\I_M$.

\begin{propos}\label{propos_rho_coincide}
We have $\rho_{0,M}=\rho_{3,M}$ and~$\rho_{1,M}=\rho_{2,M}$.
\end{propos}

\begin{proof}
 The Dehn twist $T_{\alpha_0}$  acts on the group~$H_{\Z/2}$ by stabilizing the basis elements~$\ba_1$, $\ba_2$, $\ba_3$, and~$\bb_3$, and sending $\bb_1$ and~$\bb_2$ to $\bb_1+\ba_1+\ba_2$ and $\bb_2+\ba_1+\ba_2$, respectively. It follows easily that the conjugation by~$T_{\alpha_0}$ permutes the quadratic forms~$\omega_i$ by $\omega_0\leftrightarrow \omega_3$ and~$\omega_1\leftrightarrow \omega_2$ and hence the Birman--Craggs homomorphisms~$\rho_i$ by~$\rho_0\leftrightarrow \rho_3$ and~$\rho_1\leftrightarrow \rho_2$. On the other hand, $T_{\alpha_0}$ commutes with all elements of $\I_M$ and therefore acts trivially on~$H^1(\I_M;\Z/2)$. Thus, $\rho_{0,M}=\rho_{3,M}$ and~$\rho_{1,M}=\rho_{2,M}$.
\end{proof}

\begin{cor}\label{cor_rho_rho_zero}
Suppose that $h_1,h_2\in\I_M$; then 
$$
\sum_{i\ne j}\rho_i(h_1)\rho_j(h_2)=0.
$$
\end{cor}

Now, we prove the following analogue of Proposition~\ref{propos_liftrho}.

\begin{propos}\label{propos_liftrho2}
For all $h\in\I_M$, 
\begin{equation}\label{eq_rho_nu}
\rho_0(h)+\rho_1(h)=\nu_{\alpha_0}(h) \mod 2.
\end{equation}
\end{propos}

\begin{proof}
By Proposition~\ref{propos_generate_type2} it is sufficient to prove the required equality for  $h=T_{\alpha',\alpha_i}$, where $i\in\{0,1,2\}$ and~$\alpha'$ is a curve homologous to~$\alpha_i$ and disjoint from~$M$.

First, suppose that $i$ is either~$1$ or~$2$. Since $\omega_0(\ba_i)=\omega_1(\ba_i)=1$, formula~\eqref{eq_BC1} easily implies that $\rho_0(T_{\alpha',\alpha_i})=\rho_1(T_{\alpha',\alpha_i})=0$. The required equality follows, since $\nu_{\alpha_0}(T_{\alpha',\alpha_i})=0$ by Proposition~\ref{propos_phi_values}.

Second, suppose that  $i=0$. Consider a separating simple closed curve~$\delta$ that separates $\alpha_0\cup\alpha'$ from~$\alpha_3$. The modulo~$2$ homology group of the genus~1 subsurface of~$S$ bounded by~$\delta$ has basis $\ba_3$, $\bb_3+m_1\ba_1+m_2\ba_2$ for certain $m_1,m_2\in\Z/2$.  Using  formula~\eqref{eq_BC1}, we easily compute that 
\begin{align*}
\rho_0\bigl(T_{\alpha',\alpha_0}\bigr) &=\bigl(\omega_0(\ba_1+\ba_2)+1\bigr)\omega_0(\ba_3)\omega_0(\bb_3+m_1\ba_1+m_2\ba_2)=m_1+m_2,\\
\rho_1\bigl(T_{\alpha',\alpha_0}\bigr) &=\bigl(\omega_1(\ba_1+\ba_2)+1\bigr)\omega_1(\ba_3)\omega_1(\bb_3+m_1\ba_1+m_2\ba_2)=m_1+m_2+1. 
\end{align*}
On the other hand, by Proposition~\ref{propos_phi_values} we have $\nu_{\alpha_0}(T_{\alpha',\alpha_0})=1$, which implies~\eqref{eq_rho_nu}.
\end{proof}

Now, we are ready to prove  Proposition~\ref{propos_theta_main}(b) for type~2 multicurves.

\begin{propos}\label{propos_theta_main2}
$\theta_{\fA,M}=0$.
\end{propos}

\begin{proof}
By Proposition~\ref{propos_rho_coincide}, we have
$$
\theta_{\fA,M}=\sum_{i<j}\rho_{i,M}\rho_{j,M}=\rho_{0,M}^2+\rho_{1,M}^2=\beta(\rho_{0,M}+\rho_{1,M}),
$$
where $\beta$ is the Bockstein homomorphism. On the other hand, by Proposition~\ref{propos_liftrho2} the cohomology class~$\rho_{0,M}+\rho_{1,M}$ is integral, that is, lies in the image of the natural homomorphism $H^1(\I_M;\Z)\to H^1(\I_M;\Z/2)$. Therefore $\beta(\rho_{0,M}+\rho_{1,M})=0$.
\end{proof}

Now, we would like to study five-component multicurves $K\in\M_2'$ that contain~$M$. For each such multicurve~$K$, its principal component is one of the three curves~$\alpha_0$, $\alpha_1$, and~$\alpha_2$. We will say that $K$ is \textit{associated} with~$\alpha_i$ if the principal component of~$K$ is~$\alpha_i$. It follows from Corollary~\ref{cor_injection_H1} that the homomorphism $H_1(\I_K;\Z)\to H_1(\I_M;\Z)$ induced by the inclusion $\I_K\subset \I_M$ is injective. We will identify $H_1(\I_K;\Z)$ with a subgroup of~$H_1(\I_M;\Z)$ via this homomorphism. We would like to study subgroups $H_1(\I_K;\Z)\subset H_1(\I_M;\Z)$ for various $K$ containing~$M$. In the following proposition we always mean that $K$, $K'$, $K_j$ are five-component multicurves in~$\M_2'$ that contain~$M$.

\begin{propos}\label{propos_KH_decompose}
\textnormal{(a)} If $K$ and~$K'$ are associated with the same component~$\alpha_i$, then  the subgroups~$H_1(\I_K;\Z)$ and~$H_1(\I_{K'};\Z)$ of~$H_1(\I_M;\Z)$ coincide with each other.

\textnormal{(b)} If $K_0$, $K_1$, and~$K_2$ are  associated with $\alpha_0$, $\alpha_1$, and~$\alpha_2$, respectively, then
\begin{equation}\label{eq_KH_decompose}
H_1(\I_M;\Z)=
H_1(\I_{K_0};\Z)\oplus
H_1(\I_{K_1};\Z)\oplus
H_1(\I_{K_2};\Z).
\end{equation}
\end{propos}

\begin{proof}
We start with proving~(b). Let  $a_1,a_2,a_3,b_1,b_2,b_3$ be a symplectic basis of~$H$ such that $a_i=[\alpha_i]$, $i=1,2,3$. 

First, let us prove that the three subgroups~$H_1(\I_{K_i};\Z)$ generate the whole group $H_1(\I_M;\Z)$. By Proposition~\ref{propos_generate_type2}, the group~$H_1(\I_M;\Z)$ is generated by classes~$[T_{\alpha',\alpha_i}]$, where $i=0$, $1$, or~$2$ and $\alpha'$ is homologous to~$\alpha_i$ and disjoint from~$M$. So we suffice to show that every twist~$T_{\alpha',\alpha_i}$ is conjugate in~$\I_M$ to an element of~$\I_{K_i}$. We may assume that~$i=1$. Let $\gamma$ be a component of~$K_1$ that is not a component of~$M$. Then $[\gamma]$ is one of the classes~$a_1+a_3$, $a_1-a_3$, and~$a_3-a_1$. It is easy to see that there is a simple closed curve~$\gamma'$ such that $[\gamma']=[\gamma]$ and~$\gamma'$ is disjoint from~$M\cup\alpha'$. Now, it follows from Proposition~\ref{propos_orbits} that the multicurves $K=M\cup\gamma$ and $M\cup\gamma'$ lie in the same~$\I$-orbit. Hence there exists an element $h\in\I_M$ such that $h(\gamma')=\gamma$. Put $\alpha''=h(\alpha')$. Then the twist~$T_{\alpha'',\alpha_1}$ belongs to~$\I_{K_1}$ and is conjugate to~$T_{\alpha',\alpha_1}$ in~$\I_M$.

Second, we prove that the sum of the subgroups~$H_1(\I_{K_i};\Z)$, $i=0,1,2$, is direct, that is, each of these subgroups does not intersect the sum of the other two. Consider all homomorphisms $\nu_{\alpha_i,\CW}$, where $i=0,1,2$. It follows from Proposition~\ref{propos_nu_W_w} that, for any element $h\in H_1(\I_{K_i};\Z)$, at least one of the homomorphisms~$\nu_{\alpha_i,\CW}$ takes a nonzero value on~$h$. On the other hand, all homomorphisms~$\nu_{\alpha_j,\CW}$ with $j\ne i$ vanish on $H_1(\I_{K_i};\Z)$. Indeed, $\I_{K_i}$ is generated by twists about bounding pairs of curves in homology class~$[\alpha_i]$ and every~$\nu_{\alpha_j,\CW}$ vanishes on all such twists. Thus, the sum~\eqref{eq_KH_decompose} is direct. Besides, we see that the summand $H_1(\I_{K_i};\Z)$ is exactly the subgroup of~$H_1(\I_M;\Z)$ obtained as the intersection of kernels of all  homomorphisms~$\nu_{\alpha_j,\CW}$ with $j\ne i$. Since this characterization is independent of~$K_i$, we see that the summand $H_1(\I_{K_i};\Z)$ will not change if we replace~$K_i$ with another five-component multicurve associated to the same~$\alpha_i$. So (a) also follows.
\end{proof}

For every $i=0,1,2$, we denote by~$\pr_{\alpha_i}$ the projection onto the $i$th summand in decomposition~\eqref{eq_KH_decompose}. By  Proposition~\ref{propos_KH_decompose}(a), these projectors are independent of the choice of~$K_0$, $K_1$, and~$K_2$. Let $\sigma_{M,\alpha_i}\colon \I_M\to\BB_3'$ be the composition of homomorphisms
$$
\I_M\to H_1(\I_M;\Z)\xrightarrow{\pr_{\alpha_i}} H_1(\I_{K_i};\Z)\xrightarrow{\sigma_*}\BB_3',
$$
where the latter arrow is induced by the restriction of the Birman--Craggs--Johnson homomorphism~$\sigma$ to~$\I_{K_i}$. Then, for all $h\in\I_M$,
$$
\sigma(h)=\sigma_{M,\alpha_0}(h)+\sigma_{M,\alpha_1}(h)+\sigma_{M,\alpha_2}(h).
$$

The homomorphisms $\sigma_{M,\alpha_i}$ will be used in the next section to construct  homomorphisms $\sigma_{C,c}\colon E^1_{1,1}\to\BB_3'$.

\section{Several useful homomorphisms of the groups~$E^1_{1,1}$ and~$E^1_{2,1}$}\label{section_several_hom}

In this section, we use the Birman--Craggs--Johnson homomorphism~$\sigma$, the homomorphisms~$\sigma_{M,\alpha_i}$ constructed in the previous section, and the homomorphisms~$\nu_{\gamma}$ and~$\mu_{\gamma,\gamma'}$ constructed in Section~\ref{section_nu} to define several useful homomorphisms of the groups~$E^1_{1,1}$ and~$E^1_{2,1}$ to~$\BB_3'$ and $\Z$.

\subsection{Homomorphisms~$\sigma_C$ and~$\sigma_{C,c}$} Suppose that $C$ is a set belonging to~$\CH_p'$, where $p$ is either~$1$ or~$2$. Define a homomorphism $\sigma_C\colon E^1_{p,1}\to\BB_3'$ by 
$$
\sigma_C\bigl([h]_M\bigr)=\left\{
\begin{aligned}
&\sigma(h),&[M]&=C,\\
&0,&[M]&\ne C.
\end{aligned}
\right.
$$  
It is easy to see that the homomorphism~$\sigma_C$ is well defined.

\begin{propos}\label{propos_sigma_d1}
Suppose that $C\in\CH'_1$ and $y\in E^1_{2,1}$; then
\begin{equation}\label{eq_sigma_d1}
\sigma_C(d^1y)=\sum_{D\in\CH_2',\,D\supset C}\sigma_D(y).
\end{equation}
\end{propos}

\begin{proof}
It is sufficient to prove~\eqref{eq_sigma_d1} for generators of the group~$E^1_{2,1}$, i.\,e. for elements $y=[h]_K$, where $K\in\M_2$ and~$h\in\I_K$. We have
\begin{equation}\label{eq_dh}
d^1[h]_K=\sum_{M\in\M_1,\,M\subset K}[P_K\colon P_M]\,[h]_M.
\end{equation}

If $[K]\not\supset C$, then both sides of~\eqref{eq_sigma_d1} obviously vanish.

Suppose that $K\notin\M_2'$ and $[K]\supset C$.  Then the right-hand side of~\eqref{eq_sigma_d1} vanishes. On the other hand, by Proposition~\ref{propos_M2no'}, $K$ contains a bounding pair. Hence, $K$ contains exactly two multicurves~$M$ with $[M]=C$. Thus, the left-hand side of~\eqref{eq_sigma_d1} also vanishes.

Now, suppose that $K\in\M_2'$ and $[K]\supset C$. Then the right-hand side of~\eqref{eq_sigma_d1} is equal to~$\sigma(h)$. On the other hand, $K$  contains exactly one multicurve~$M$ with $[M]=C$. Hence the left-hand side of~\eqref{eq_sigma_d1} is also equal to~$\sigma(h)$.
\end{proof}

In case~$C\in\M_1^{(2)}$, we may use homomorphisms~$\sigma_{M,\gamma}$ constructed in the end of Section~\ref{subsection_type2} to refine the construction of~$\sigma_C$ as follows. Suppose that $c$ is a non-special element of~$C$.
Define a homomorphism $\sigma_{C,c}\colon E^1_{1,1}\to\BB_3'$ by 
$$
\sigma_{C,c}\bigl([h]_M\bigr)=\left\{
\begin{aligned}
&\sigma_{M,\gamma}(h),&[M]&=C,\\
&0,&[M]&\ne C,
\end{aligned}
\right.
$$  
where in the former case $\gamma$ is the component of~$M$ with $[\gamma]=c$.

\begin{propos}\label{propos_sigma2_d1}
Suppose that $C\in\CH_1^{(2)}$, $c$ is a non-special element of~$C$, and $y\in E^1_{2,1}$; then
\begin{equation}\label{eq_sigma2_d1}
\sigma_{C,c}(d^1y)=\sum\sigma_D(y),
\end{equation}
where the sum is taken over all sets~$D\in\CH_2'$ such that $D$ contains~$C$ and the principal element of~$D$ is~$c$.
\end{propos}

\begin{proof}
As in the proof of Proposition~\ref{propos_sigma_d1}, we assume that $y=[h]_K$ and use formula~\eqref{eq_dh}. The only case that needs to be taken care of is the case of~$K$ containing a bounding pair~$\{\alpha^+,\alpha^-\}$ and containing two sub-multicurves~$M^{\pm}$ with $[M^{\pm}]=C$. In this case the right-hand side of~\eqref{eq_sigma2_d1} vanishes. We may assume that $\alpha^+\subset M^+$ and $\alpha^-\subset M^-$. Let $\gamma^+$ and~$\gamma^-$ be the components of~$M^+$ and~$M^-$ in the homology class~$c$. Then either $\gamma^{\pm}=\alpha^{\pm}$ or $\gamma^+=\gamma^-$. It is sufficient to consider cases $h=T_{\alpha^+,\alpha^-}$ and~$h=T_{\delta}$, where $\delta$ is separating and disjoint from~$K$, since  by Proposition~\ref{propos_stab_generate} such elements~$h$ generate~$\I_K$.  Arrange the multicurve~$K$  and the curve~$\delta$ (in case $h=T_{\delta}$) as the multicurve and the curve shown in Fig.~\ref{fig_Kbp}, and consider the rotation~$\iota$ by~$\pi$ around the horizontal axis. Since the homomorphisms~$\sigma_{M^+,\gamma^+}$ and~$\sigma_{M^-,\gamma^-}$ were defined in a certain canonical way, they must be taken to each other by the involution~$\iota$. Hence $\sigma_{M^+,\gamma^+}(h)=\sigma_{M^-,\gamma^-}(h)$ for both $h=T_{\alpha^+,\alpha^-}$ and~$h=T_{\delta}$. Thus, $\sigma_{C,c}(d^1[h]_K)=0$.
\end{proof}

\begin{figure}
\begin{tikzpicture}[scale=.45]
\small
\definecolor{myblue}{rgb}{0, 0, 0.7}
\definecolor{mygreen}{rgb}{0, 0.4, 0}
\tikzset{every path/.append style={line width=.2mm}}
\tikzset{my dash/.style={dash pattern=on 2pt off 1.5pt}}

\begin{scope}
% Invisible lines

\draw[color=myblue,my dash, very thick] (0,-3.03) .. controls (-0.3,-2) .. (0,-.97); % e_0
\draw[color=myblue,my dash, very thick] (0,3.03) .. controls (-0.3,2) .. (0,.97); % e_0'
\draw[color=myblue,my dash, very thick] (-7.03,0) .. controls (-6,0.3) .. (-4.97,0); %e_1
\draw[color=myblue,my dash, very thick] (7.03,0) .. controls (6,0.3) .. (4.97,0); %e_2
\draw[color=myblue,my dash, very thick] (-3.03,0) .. controls (-2,0.3) .. (-.97,0); % e_0+e_1
%\draw[color=red,my dash ] (3.03,0) .. controls (2,0.3) .. (.97,0); % e_0+e_2
\draw[color=red, my dash] (2, -3.03) .. controls (1.7,-2) and (1.7,2) .. (2, 3.03);
%\draw[color=myblue, very thick,my dash] (-3.293,-.707) .. controls (-1.5,-2.2) and (1.5,-2.2) .. (3.293,-.707);
%\draw[color=red,my dash] (-3.293,.707) .. controls (-1.5,2.8) and (1.5,2.8) .. (3.293,.707);

% Border lines

\draw [very thick] (0,0) circle (1);
\draw [very thick] (4,0) circle (1);
\draw [very thick] (-4,0) circle (1);
\draw [very thick]  (-4,3)--(4,3);
\draw [very thick]  (-4,-3)--(4,-3);
\draw [very thick] (-4,3) arc (90:270:3); 
\draw [very thick] (4,3) arc (90:-90:3); 

\draw [] (-9.5,0) -- (-7.3,0);
\draw [] (9.5,0) -- (7.3,0);

\draw (8.5,0) + (-55: .3 and 1) arc (-55: -10: .3 and 1);
\draw [-stealth] (8.5,0) + (10: .3 and 1) arc (10: 160: .3 and 1) node [pos=.4, right] {$\iota$};

% Visible lines 

\draw[color=myblue, very thick] (0,-3.03) .. controls (0.3,-2) .. (0,-.97) node [pos=0.5, right=-2pt] {$\alpha^-$};  %e_0
\draw[color=myblue, very thick] (0,3.03) .. controls (0.3,2) .. (0,.97) node [pos=0.5, right=-2pt]{$\alpha^+$};  %e_0'

\draw[color=myblue, very thick] (-4.97,0) .. controls (-6,-0.3) .. (-7.03,0);%e_1
%\draw[color=red,->-=.55] (4.97,0) .. controls (6,-0.4) .. (7.03,0) node [pos=.3,below=-1pt] {$\alpha_2$}; %e_2
%\draw[color=red,->-=.55] (-7.03,0) .. controls (-6,-0.3) .. (-4.97,0)node [pos=0.4,below] {$\alpha_1$};%e_1 -
\draw[color=myblue, very thick] (7.03,0) .. controls (6,-0.3) .. (4.97,0);%e_2 -

\draw[color=myblue, very thick] (-3.03,0) .. controls (-2,-0.3) .. (-.97,0);%e_0+e_1
%\draw[color=red] (3.03,0) node [right=-2pt] {$\delta$} .. controls (2,-0.3) .. (.97,0) ;%e_0+e_2

\draw[color=red] (2, -3.03) .. controls (2.3,-2) and (2.3,2) .. (2, 3.03) node [pos=0.35, right=-1pt] {$\delta$};  %e_0

%\draw[color=red] (-2,-3.03) .. controls (-1.7,-1.5) and (-1.7,1.5) .. (-2,3.03) node [pos=.25,right=-2pt] {$u$};
%\draw[color=red,my dash] (-2,-3.03) .. controls (-2.3,-1.5) and (-2.3,1.5) .. (-2,3.03);

%\draw[color=red] (-3.293,-.707) .. controls (-1,-2.8) and (1,-2.8) .. (3.293,-.707) node [pos=.8,below=1pt] {$\delta_2$};
%\draw[color=red] (-3.293,.707) .. controls (-1,2.2) and (1,2.2) .. (3.293,.707)node [pos=.8,below] {$\delta$};
\end{scope}

\end{tikzpicture}
\caption{Multicurve~$S$, separating curve~$\delta$, and involution~$\iota$}\label{fig_Kbp}
\end{figure}
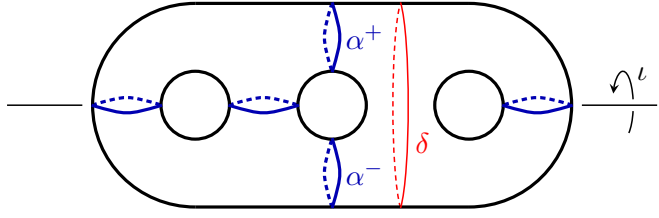

\subsection{Homomorphisms~$\nu_{C,c}$ and~$\nu_{C,c}^+$}
Suppose that 
\begin{itemize}
\item either $C\in\CH_1^{(2)}$ and $c$ is a non-special element of~$C$
\item or~$C\in\CH_2'$ and $c$ is the principal element of~$C$. 
\end{itemize}
Put $p=1$ in the former case and $p=2$ in the latter case. Recall that the set of oriented multicurves~$M$ with $[M]=C$ consists of two $\I$-orbits~$\orb_C^{\pm}$, and in Section~\ref{section_cells} we have agreed to fix (arbitrarily) which of the orbits is~$\orb_C^+$ and which is~$\orb_C^-$.

Define homomorphisms~$\nu_{C,c}$ and~$\nu_{C,c}^+$ of~$E^1_{p,1}$ to~$\Z$ by 
\begin{align*}
\nu_{C,c}\bigl([h]_M\bigr)&=\left\{
\begin{aligned}
&\nu_{\gamma}(h),&M&\in \orb_C^+,\\
-&\nu_{\gamma}(h),&M&\in \orb_C^-,\\
&0,&M&\notin \orb_C^+\cup\orb_C^-,
\end{aligned}
\right.
\\
\nu_{C,c}^+\bigl([h]_M\bigr)&=\left\{
\begin{aligned}
\phantom{-}&\nu_{\gamma}(h),&M&\in \orb_C^+,\\
&0,&M&\notin \orb_C^+,
\end{aligned}
\right.
\end{align*}
where $\gamma$ is the component of~$M$ in the homology class~$c$.
  
These two homomorphisms depend on the choice of orientations of the cells~$P_M$ and also on which of the two $\I$-orbits is taken for~$\orb_C^+$.  Note that, by our  agreement on orientations, we may reverse the orientations only for all cells~$P_M$, $M\in\orb_C^+\cup\orb_C^-$, simultaneously. If we do this, each of the homomorphisms~$\nu_{C,c}$ and~$\nu_{C,c}^+$ will change sign. Also the homomorphism~$\nu_{C,c}$ will change sign if we swap the orbits~$\orb_C^+$ and~$\orb_C^-$. (The homomorphism~$\nu_{C,c}^+$ behaves worse under this operation, namely, it changes to $\nu_{C,c}^+-\nu_{C,c}$.)
  
In case $C\in\CH_2'$ we will usually omit $c$ in notation and write simply~$\nu_C$ and~$\nu_C^+$, since the set~$C$ contains a unique principal element.  

For any $C\in\CH_1^{(2)}$ and any $A\in\CH_0'$ such that $C\supset A$, we put $\nu_{C,A}=\nu_{C,c}$ and $\nu_{C,A}^+=\nu_{C,c}^+$, where $c$ is a unique element of~$C\setminus A$. (Note that $c$ is always non-special.) 

\subsection{Homomorphisms $\mu_C$} Now, suppose that $C$ is a multiset in~$\CH_p$, where $p$ is either~$1$ or~$2$, such that there is a (unique) element~$c$ that has multiplicity~$2$ in~$C$. 
 
Define a homomorphism $\mu_C\colon E^1_{p,1}\to\Z$ by 
$$
\mu_C\bigl([h]_M\bigr)=\left\{
\begin{aligned}
&\mu_{\gamma,\gamma'}(h),&[M]&=C,\\
&0,&[M]&\ne C,
\end{aligned}
\right.
$$  
where $\{\gamma,\gamma'\}$ is the bounding pair in~$M$. If we reverse the orientation of~$P_M$ simultaneously for all~$M$ with~$[M]=C$, the homomorphism $\mu_C$ will change sign.

In the following proposition we use notation~$[D\colon C]$ and~$\varepsilon_{D,C}$ introduced in Section~\ref{section_cells}.

\begin{propos}\label{propos_nu_d1}
Suppose that $C\in\CH^{(2)}_1$, $c$ is a non-special element of~$C$, and $y\in E^1_{2,1}$; then
\begin{equation}\label{eq_nu_d1}
\nu_{C,c}(d^1y)=\sum\varepsilon_{D,C}[D\colon C]\nu_D(y)+2\sum [D\colon C]\mu_D(y),
\end{equation}
where the former sum is taken over all~$D\in\CH_2'$ such that $D\supset C$ and $c$ is the principal element of~$D$, and the latter sum is taken over all~$D\in\CH_2\setminus\CH_2'$ such that $D\supset C$.
\end{propos}

\begin{proof}
As in the proof of Proposition~\ref{propos_sigma_d1}, assume that $y=[h]_K$ and use formula~\eqref{eq_dh}. Put $D_0=[K]$. 
If $D_0\not\supset C$, then both sides of~\eqref{eq_nu_d1} vanish. So we assume that $D_0\supset C$.

First, suppose that $D_0\in\CH_2'$. Let $M\subset K$ be the multicurve with $[M]=C$ and $\gamma$ be the component of~$M$ in the homology class~$c$. If $c$ is the principal element of~$D_0$, then both~$\nu_{C,c}([h]_M)$ and~$\nu_{D_0}([h]_K)$  are equal to~$\pm \nu_{\gamma}(h)$ with the signs depending on which of the two orbits~$\orb_D^{\pm}$ contains~$K$ and which of the two orbits~$\orb_C^{\pm}$ contains~$M$, respectively. Together with $[P_K\colon P_M]$ these two signs give exactly $\varepsilon_{D_0,C}[D_0\colon C]$. All other summands in the right-hand side obviously vanish, so~\eqref{eq_nu_d1} holds. 

If $c$ is a non-principal element of~$D_0$, then the right-hand side of~\eqref{eq_nu_d1} vanishes. But the left-hand side also vanishes by Proposition~\ref{propos_triv_restr2}, so~\eqref{eq_nu_d1} again holds. 

Second, suppose that $D_0\notin\CH_2'$; then $K$ contains a bounding pair~$\{\alpha^+,\alpha^-\}$. Let $M^+$ and~$M^-$ be the multicurves obtained from~$K$ by removing the components~$\alpha^-$ and~$\alpha^+$, respectively. We could swap the pair~$\{\alpha^+,\alpha^-\}$ so that to achieve that $M^+\in\orb_C^+$ and $M^-\in\orb_C^-$. Since all summands in the right-hand side of~\eqref{eq_nu_d1}, except for the one that corresponds to~$D_0$, vanish, the required formula~\eqref{eq_nu_d1} reads as
\begin{equation}\label{eq_gamma_alpha_pm}
\nu_{\gamma^+}(h)+\nu_{\gamma^-}(h)=2\mu_{\alpha^+,\alpha^-}(h),
\end{equation}
where $\gamma^+$ and~$\gamma^-$ are the components of~$M^+$ and~$M^-$ in the homology class~$c$, respectively. 

If $[\alpha^+]=[\alpha^-]=c$, then $\gamma^+=\alpha^+$ and~$\gamma^-=\alpha^-$, and~\eqref{eq_gamma_alpha_pm} is just formula~\eqref{eq_mu_defin}. 

Suppose that $[\alpha^+]=[\alpha^-]\ne c$. Then the curves~$\gamma^+$ and~$\gamma^-$ coincide to each other, so we denote this curve simply by~$\gamma$. We need to prove that $\nu_{\gamma}(h)=\mu_{\alpha^+,\alpha^-}(h)$. We suffice to consider two cases, $h=T_{\alpha^+,\alpha^-}$ and $h=T_{\delta}$ for a separating curve~$\delta$ disjoint from~$K$, since by Proposition~\ref{propos_stab_generate} such elements generate~$\I_K$. In both cases the required equality follows easily from Propositions~\ref{propos_phi_values} and~\ref{propos_psi_values}.
\end{proof}

\section{Subgroup $\Delta$}\label{section_Delta}

To prove Proposition~\ref{propos_d2} we will use Fact~\ref{fact_differential}, so we have to work directly with the bicomplex~$B_{*,*}=C_*(\B;\Z)\otimes\res_*$, where $\res_*$ is the bar resolution for~$\Z$ over~$\Z\I$. Because of that we would like to find homomorphisms $E^0_{0,2}=B_{0,2}\to\Z/2$ that vanish on the image of the differential $d^0=\partial''$ and induce the homomorphisms~$\btheta_A$ in page~$E^1$. 
 
Suppose that $\rho_0$, $\rho_1$, $\rho_2$, and~$\rho_3$ are the four Birman--Craggs homomorphisms corresponding to a set $A\in\CH_0'$, as in Section~\ref{section_theta}. Then we define a homomorphism 
$$
\Theta_{A}^{\rho_0,\rho_1,\rho_2,\rho_3}\colon B_{0,2}\to\Z/2
$$
by
$$
\Theta_{A}^{\rho_0,\rho_1,\rho_2,\rho_3}\bigl(P_M\otimes [h_1|h_2]\bigr)=\left\{
\begin{aligned}
&\sum_{i<j}\rho_i(h_1)\rho_j(h_2)&&\text{if }[M]=A,\\
&0&&\text{if }[M]\ne A.
\end{aligned}
\right.
$$

\begin{propos}\label{propos_Theta_induce}
The homomorphism~$\Theta_{A}^{\rho_0,\rho_1,\rho_2,\rho_3}$ vanishes on the image of~$\partial''$ and induces the homomorphism~$\btheta_A$ in page~$E^1$.
\end{propos}

\begin{proof}
The map~$\partial''\colon B_{0,3}\to B_{0,2}$ is induced by the differential of the bar resolution, so the proposition follows from the standard fact that the cochain $c_{ij}\in\Hom_{\I}(\res_2,\Z/2)$ given by $c_{ij}([h_1|h_2])=\rho_i(h_1)\rho_j(h_2)$ is a cocycle representing the cup-product $\rho_i\rho_j\in H^2(\I;\Z/2)$, cf.~\cite[Section~V.3]{Bro82}.
\end{proof}

Unlike~$\btheta_A$, the homomorphism~$\Theta_{A}^{\rho_0,\rho_1,\rho_2,\rho_3}$ depends on the chosen numeration of the four Birman--Craggs homomomorphisms~$\rho_0,\ldots,\rho_3$. We denote by~$\Gamma$ the subgroup  of~$B_{0,2}$ consisting of all elements~$X$ such that, for all $A\in\CH_0'$ and all permutations~$\rho_0',\rho_1',\rho_2',\rho_3'$ of the corresponding four Birman--Craggs homomorphisms~$\rho_0,\rho_1,\rho_2,\rho_3$, we have
$$
\Theta_{A}^{\rho_0',\rho_1',\rho_2',\rho_3'}(X)=\Theta_{A}^{\rho_0,\rho_1,\rho_2,\rho_3}(X).
$$ 
We denote by~$\Theta_A$ the restriction of any of the homomorphisms~$\Theta_{A}^{\rho_0,\rho_1,\rho_2,\rho_3}$ to~$\Gamma$.

A central role in our proof of Proposition~\ref{propos_d2} will be played by a special subgroup
$$
\Delta\subset B_{1,1}= C_1(\B;\Z)\otimes_{\I}\res_1.
$$
By definition, $\Delta$ is generated by all elements of the following three types:
\begin{enumerate}
\item $P_M\otimes [h]$, where $M\in \M_1^{(1)}$ and $h\in\CC_M$,
\item $P_{M_1}\otimes [h_1]-P_{M_2}\otimes [h_2]$, where $M_1,M_2\in \M_1^{(2)}$,  $[M_1]=[M_2]$, $M_1$ and~$M_2$ lie in different $\I$-orbits, $h_1\in\I_{M_1}$,  $h_2\in\I_{M_2}$, and $\sigma(h_1)=\sigma(h_2)$,
\item $P_M\otimes [h]$, where $M\in \M_1\setminus\left(\M_1^{(1)}\cup \M_1^{(2)}\right)$ and $h\in\I_M$.
\end{enumerate}
We shall refer to these generators as to type~1, type~2, and type~3 generators, respectively.

For each $C\in\CH_1^{(1)}$, we define a homomorphism $\Psi_C\colon\Delta\to\Z/2$ as follows:
\begin{itemize}
\item for each type~1 generator $P_{M}\otimes [h]$, we put
\begin{equation*}
\Psi_C(P_{M}\otimes [h])=
\left\{\begin{aligned}
&\psi_M(h)&&\text{if $[M]=C$},\\
&0&&\text{if $[M]\ne C$,}
\end{aligned}
\right.
\end{equation*}
where $\psi_M\colon\CC_M\to\Z/2$ is the homomorphism constructed in Proposition~\ref{propos_psi1},
\item for each type~2 or type~3 generator, we set the value of~$\Psi_C$ on it to be zero.
\end{itemize}

The uniqueness of the homomorphism~$\psi_M$ satisfying conditions of Proposition~\ref{propos_psi1} implies that $\psi_{g(M)}(ghg^{-1})=\psi_M(h)$ for any $g\in\I$. Hence, $\Psi_C$ is well defined.

As in Section~\ref{section_theta}, suppose that $A\in\CH_0'$, $\fA$ is the reduction of~$A$ modulo~$2$, $\rho_0,\ldots,\rho_3$ (respectively, $\hrho_0,\ldots,\hrho_3$) are the four corresponding Birman--Craggs  (respectively, extended Birman--Craggs) homomorphisms.

\begin{propos}\label{propos_Phi_CA_value}
Suppose that $C\in\CH_1^{(2)}$ and $C\supset A$. Then there exists a unique homomorphism~$\Phi_{C,A}\colon\Delta\to\Z/2$ such that
\begin{itemize}
\item for each type~2 generator, 
\begin{equation}\label{eq_Phi}
\Phi_{C,A}(P_{M_1}\otimes [h_1]-P_{M_2}\otimes [h_2])=
\left\{\begin{aligned}
&\sum_{i\ne j}\hrho_i(g)\rho_j(h_1)&&\text{if}\ \ [M_1]=[M_2]=C,\\
&{}\,\,0&&\text{if}\ \  [M_1]=[M_2]\ne C,
\end{aligned}
\right.
\end{equation}
where $g$ is an element of~$\hI$ such that $g(M_1)=M_2$,
 \item the value of~$\Phi_{C,A}$ on every type~1 or type~3 generator is zero.
\end{itemize}
\end{propos}

\begin{proof}
First, let us show that the value of~$\Phi_{C,A}$ on a type~2 generator is independent of the choice of~$g$. This is obvious if $[M_1]\ne C$, so we assume that $[M_1]=C$. If we replace~$g$ with another element $g'\in\hI$ such that $g'(M_1)=M_2$, then the right-hand side of~\eqref{eq_Phi} will change by a summand 
$$
\sum_{i\ne j}\rho_i\bigl(g^{-1}g'\bigr)\rho_j(h_1),
$$ 
which equals zero by Corollary~\ref{cor_rho_rho_zero}, since both $g^{-1}g'$ and~$h_1$ lie in~$\I_{M_1}$. 

Second, if we swap the pairs~$(M_1,h_1)$ and $(M_2,h_2)$, then we should replace~$g$ by~$g^{-1}$ and~$h_1$ by~$h_2$. Since $\sigma(h_1)=\sigma(h_2)$,  the right-hand side of~\eqref{eq_Phi} will not change.

So the value of~$\Phi_{C,A}$ on every generator of~$\Delta$ is well defined. Now, it is not hard to see that all relations among the described set of generators of~$\Delta$ follow from relations of the form
\begin{multline*}
\bigl(P_{M_1}\otimes [h_1]-P_{M_2}\otimes [h_2]\bigr)+\bigl(P_{M_1'}\otimes [h_1']-P_{M_2'}\otimes [h_2']\bigr)\\
{}=\bigl(P_{M_1}\otimes [h_1]-P_{M_2'}\otimes [h_2']\bigr)+\bigl(P_{M_1'}\otimes [h_1']-P_{M_2}\otimes [h_2]\bigr),
\end{multline*}
where $M_1,M_1',M_2,M_2'\in\M_1^{(2)}$, $[M_1]=[M_1']=[M_2]=[M_2']$, $\I M_1=\I M_1'\ne \I M_2=\I M_2'$, $h_i\in\I_{M_i}$, $h_i'\in\I_{M_i'}$, and $\sigma(h_1)=\sigma(h_1')=\sigma(h_2)=\sigma(h_2')$. Therefore, to prove  the well-definedness of~$\Phi_{C,A}$ it is sufficient to show that the sum of the right-hand sides of~\eqref{eq_Phi} for the four elements $P_{M_1}\otimes [h_1]-P_{M_2}\otimes [h_2]$, $P_{M_1'}\otimes [h_1']-P_{M_2'}\otimes [h_2']$, $P_{M_1}\otimes [h_1]-P_{M_2'}\otimes [h_2']$, and $P_{M_1'}\otimes [h_1']-P_{M_2}\otimes [h_2]$ is zero. (Signs do not matter, since the values are in~$\Z/2$.) Let $g\in\hI$, $f_1,f_2\in\I$ be elements such that $g(M_1)=M_2$, $f_i(M_i)=M_i'$, $i=1,2$. Then the sum we are interested in is equal to
$$
\sum_{i\ne j}\bigl(\hrho_i(g)+\hrho_i(f_2gf_1^{-1})+\hrho_i(f_2g)+\hrho_i(gf_1^{-1})\bigr)\rho_j(h_1)=0.
$$ 
Thus, $\Phi_{C,A}$ is well defined.
\end{proof}

\begin{lem}\label{lem_Delta_1}
Suppose that $M\in\M_1^{(1)}\cup\M_1^{(2)}$ and $h\in[\I_M,\I_M]$. Then there exists an element $X\in B_{1,2}$ such that $\partial'' X=P_M\otimes[h]$, $\partial' X\in \Gamma$, and for all $A\in \CH_0'$,
\begin{equation}\label{eq_Theta_lem}
\Theta_A(\partial' X)=\left\{
\begin{aligned}
&\psi_M(h)&&\text{if $M\in\M_1^{(1)}$ and $A\subset [M]$,}\\
&0&&\text{otherwise.}
\end{aligned}
\right.
\end{equation}
\end{lem}

\begin{proof}
Suppose that $\partial P_M=P_{N_1}-P_{N_2}$; then $N_1$ and~$N_2$ are the two oriented multicurves that belong to~$\M_0$ and are contained in~$M$.

Since $h\in [\I_M,\I_M]$, we have
\begin{equation*}
h=[f_{1},f_{2}][f_{3},f_{4}]\cdots [f_{2p-1},f_{2p}],
\end{equation*}
where $f_{k}\in \I_{M}$. Let us prove the lemma by induction on~$p$.

\textsl{Base of induction.} Suppose that $p=1$, that is, $h=f_1^{-1}f_2^{-1}f_1f_2$. Put
$$
X=-P_M\otimes\bigl([f_1|f_1^{-1}f_2^{-1}f_1f_2] +[f_2|f_2^{-1}f_1f_2]-[f_1|f_2]\bigr).
$$
Then $\partial''X=P_M\otimes[h]$. Suppose that $A\in \CH_0'$ and  $\rho_0,\ldots,\rho_3$  is any permutation of the four Birman--Craggs homomorphisms corresponding to~$A$. If $A\not\subset [M]$, then neither~$[N_1]$ nor~$[N_2]$ coincides with~$A$; hence $\Theta_{A}^{\rho_0,\rho_1,\rho_2,\rho_3}(\partial'X)=0$. If $A\subset [M]$, then exactly one of the two sets~$[N_1]$ and~$[N_2]$ coincides with~$A$; therefore,
$$
\Theta_{A}^{\rho_0,\rho_1,\rho_2,\rho_3}(\partial'X)=
\sum_{i\ne j}\rho_i(f_1)\rho_j(f_2).
$$
By Proposition~\ref{propos_psi1} and Corollary~\ref{cor_rho_rho_zero} the latter sum is equal to~$\psi_M(h)$ whenever $M\in\M_1^{(1)}$ and vanishes whenever $M\in\M_1^{(2)}$. Thus, $\partial'X\in\Gamma$ and~\eqref{eq_Theta_lem} holds.

\textsl{Induction step.} Assume the assertion of the lemma for all products of commutators of lengths smaller than $p$, and prove it for length exactly~$p$, where $p\ge 2$. We have 
$h=h_1h_2$, where
$$
h_1=[f_1,f_2],\qquad h_2=[f_3,f_4]\cdots [f_{2p-1},f_{2p}].
$$
Applying the induction hypothesis to~$h_1$ and~$h_2$, we obtain that there exist elements $X_s\in B_{1,2}$ such that $\partial'' X_s=P_M\otimes[h_s]$, $\partial' X_s\in \Gamma$,
and~\eqref{eq_Theta_lem} holds for $X_s$, $s=1,2$. We put 
$$
X=X_1+X_2+P_M\otimes[h_1|h_2].
$$
Then $\partial''X=P_M\otimes[h]$.
Since $h_1$ and~$h_2$ lie in the commutator subgroup of~$\I$, every Birman--Craggs homomorphism vanish on each of them, so the condition~$\partial'X\in\Gamma$ and formula~\eqref{eq_Theta_lem} follow from the same conditions for~$X_1$ and~$X_2$.
\end{proof}

\begin{propos}\label{propos_Delta}
Suppose that an element $U\in\Delta$ lies in the image of the differential
$\partial''\colon B_{1,2}\to B_{1,1}$.
Then there exists an element $X\in B_{1,2}$ such that $\partial'' X=U$, $\partial' X\in \Gamma$, and for all $A\in \CH_0'$,
\begin{equation}\label{eq_Delta_main}
\Theta_A(\partial' X)=\sum_{C\in\CH_1^{(1)},\, C\supset A}\Psi_C(U)+
\sum_{C\in\CH_1^{(2)},\, C\supset A}\Phi_{C,A}(U).
\end{equation}
\end{propos}

\begin{proof}
The column~$B_{1,*}$ regarded as a chain complex with differential~$\partial''$ splits into the direct sum of chain complexes~$B_{1,*}(C)$ indexed by multisets $C\in\CH_1$, where~$B_{1,*}(C)$ is generated by all elements $P_M\otimes Z$ with $[M]=C$. An element $U\in B_{1,1}$ belongs to~$\Delta\cap\partial''(B_{1,2})$ if and only if all its $C$-components with respect to this splitting belong to~$\Delta\cap\partial''(B_{1,2})$. Hence, it is sufficient to prove the claim of the proposition in case $U$ is a sum of elements of the form $P_{M}\otimes[h]$ with all $[M]$ being equal to the same~$C_0\in\CH_1$.

Consider the three cases:

\textsl{Case 1: $C_0\in\CH_1^{(1)}$}. Since $U\in\Delta$, we have
\begin{equation}\label{eq_U2}
U=\sum_{i=1}^m\varepsilon_i (P_{M_{i}}\otimes[h_i]),
\end{equation}
where $[M_{i}]=C_0$, $h_{i}\in\CC_{M_{i}}$, and $\varepsilon_i=\pm 1$. Then $\Psi_C(U)=0$ unless $C=C_0$, and 
$$
\Psi_{C_0}(U)=\sum_{i=1}^m\psi_{M_i}(h_i).
$$
Besides, $\Phi_{C,A}(U)=0$ for all $C\in\CH_1^{(2)}$.
Hence the desired formula~\eqref{eq_Delta_main} reads as
\begin{equation}\label{eq_Delta_main_2}
\Theta_A(\partial'X)=\left\{
\begin{aligned}
&\sum_{i=1}^m\psi_{M_i}(h_i)&&\text{if $A\subset C_0$,}\\
&0&&\text{if $A\not\subset C_0$.}
\end{aligned}
\right.
\end{equation}

Let us prove the existence of~$X$ satisfying $\partial''X=U$, $\partial'X\in\Gamma$, and~\eqref{eq_Delta_main_2} by induction on~$m$. 

\textsl{Base of induction.} Suppose, $m=1$, i.\,e. $U=P_M\otimes[h]$, where $[M]=C_0$ and~$h\in\CC_M$. Since $U$ lies in the image of~$\partial''$,  the element $h$ represents zero homology class in~$H_1(\I_M;\Z)$ and hence belongs to the commutator subgroup~$[\I_M,\I_M]$. So the base of induction is provided by Lemma~\ref{lem_Delta_1}.

\textsl{Induction step.} Assume the existence of~$X$ with the required properties for elements~$U$ of the form~\eqref{eq_U2} with strictly less than $m$ summands, and prove it for an element~$U$ with exactly $m$ summands, where $m\ge 2$. We may assume that $\varepsilon_1=1$. By Proposition~\ref{propos_orbits}, $M_2=f(M_1)$ for some $f\in\I$. 
We put
\begin{equation*}
U_1=P_{M_{1}}\otimes[h_{1}f^{-1}h_{2}^{\varepsilon_2}f]+\sum_{i=3}^m\varepsilon_i\left(P_{M_{i}}\otimes[h_{i}]\right).
\end{equation*}
Then
$$
U_1-U=\partial''Z,
$$
where
\begin{multline*}
Z=P_{M_{1}}\otimes
\left(
\bigl[f^{-1}|h_{2}^{\varepsilon_2}\bigr]-\bigl[f^{-1}h_{2}^{\varepsilon_2}f|f^{-1}\bigr]+ \bigl[h_{1}|f^{-1}h_{2}^{\varepsilon_2}f\bigr]
\right)
\\
{}+
\left\{
\begin{aligned}
&0&&\text{if $\varepsilon_2=1$,}\\
&P_{M_2}\otimes\left([1|1]-[h_{2}|h_{2}^{-1}]\right)&&\text{if $\varepsilon_2=-1$.}\\
\end{aligned}
\right.
\end{multline*}
Hence $U_1$ lies in the image of~$\partial''$. Also $U_1\in\Delta$, since $h_{1}f^{-1}h_{2}^{\varepsilon_2}f\in\CC_{M_1}$. Besides,
$$\psi_{M_1}(h_{1}f^{-1}h_{2}^{\varepsilon_2}f)=\psi_{M_1}(h_1)+\psi_{M_2}(h_2).$$

By the induction hypothesis, there is $X_1\in B_{1,2}$ such that $\partial'' X_1=U_1$, $\partial' X_1\in \Gamma$, $\Theta_A(\partial' X_1)=\sum_{i=1}^m\psi_{M_i}(h_i)$ whenever $A\subset C_0$,  and $\Theta_A(\partial' X_1)=0$ whenever $A\not\subset C_0$. Besides, since $h_1$ and~$h_2$ lie in the kernels of all Birman--Craggs homomorphisms, we see that $\partial' Z\in\Gamma$ and $\Theta_A(\partial'Z)=0$ for all $A\in \CH_0'$. Thus, the element 
\begin{equation*}
X=X_1- Z
\end{equation*}
satisfies $\partial''X=U$, $\partial'X\in\Gamma$, and equality~\eqref{eq_Delta_main_2} for all~$A$.

\textsl{Case 2: $C_0\in\CH_1^{(2)}$}. 
Since $U\in\Delta$, we have
\begin{equation}\label{eq_U}
U=\sum_{i=1}^m\varepsilon_i\left(P_{M_{i}^+}\otimes[h_{i,+}]-P_{M_{i}^-}\otimes[h_{i,-}]\right),
\end{equation}
where $M_i^{\pm}\in\orb_{C_0}^{\pm}$, $h_{i,\pm}\in\I_{M_{i}^{\pm}}$, $\sigma(h_{i,+})=\sigma(h_{i,-})$, and $\varepsilon_i=\pm 1$. Then $\Psi_C(U)=0$ for all~$C\in\CH_1^{(1)}$ and $\Phi_{C,A}(U)=0$ unless $C=C_0$. Hence the desired formula~\eqref{eq_Delta_main} reads as
\begin{equation}\label{eq_Delta_main_1}
\Theta_A(\partial'X)=\left\{
\begin{aligned}
&\Phi_{C_0,A}(U)&&\text{if $A\subset C_0$,}\\
&0&&\text{if $A\not\subset C_0$.}
\end{aligned}
\right.
\end{equation}

As in the previous case, we prove the existence of~$X$ satisfying $\partial''X=U$, $\partial'X\in\Gamma$, and~\eqref{eq_Delta_main_1}  by induction on~$m$.

\textsl{Base of induction.} Suppose that $m=1$.  Then 
\begin{equation*}
U=\pm(P_{M^+}\otimes [h_+]-P_{M^-}\otimes [h_-]).
\end{equation*}
Since $U$ belongs to the image of~$\partial''$ and $M^+$ and~$M^-$ lie in different $\I$-orbits, we see that both elements~$P_{M^{\pm}}\otimes[h_{\pm}]$ belong to the image of~$\partial''$. Consequently, $h_{\pm}$ represent zero homology classes in~$H_1(\I_{M^{\pm}};\Z)$ and hence belong to the commutator subgroups of~$\I_{M^{\pm}}$, respectively. Therefore, $h_+$ and~$h_-$ lie in the kernels of all Birman--Craggs homomorphisms. Thus, $\Phi_{C_0,A}(U)=0$. The existence of~$X$ with required properties now follows immediately from Lemma~\ref{lem_Delta_1} for the elements $P_{M^{\pm}}\otimes [h_{\pm}]$.

\textsl{Induction step.} Assume the existence of~$X$ satisfying $\partial''X=U$, $\partial'X\in\Gamma$, and~\eqref{eq_Delta_main_1} for elements~$U$ of the form~\eqref{eq_U} with strictly less than $m$ summands, and prove it for an element~$U$ with exactly $m$ summands, where $m\ge 2$. We may assume that $\varepsilon_1=1$. Let $f_+,f_-\in \I$ and $g\in \hI$ be elements such that $f_{\pm}(M_{1}^{\pm})=M_{2}^{\pm}$ and $g(M_{1}^+)=M_{1}^-$. Put
\begin{multline*}
U_1=P_{M_1^+}\otimes[h_{1,+}f_+^{-1}h_{2,+}^{\varepsilon_2}f_+]-P_{M_{1}^-}\otimes[h_{1,-}f_-^{-1}h_{2,-}^{\varepsilon_2}f_-]\\{}+\sum_{i=3}^m\varepsilon_i\left(P_{M_{i}^+}\otimes[h_{i,+}]-P_{M_{i}^-}\otimes[h_{i,-}]\right).
\end{multline*}
Then
$$
U_1-U=\partial''(Z_+-Z_-),
$$
where
\begin{multline*}
Z_{\pm}=P_{M_{1}^{\pm}}\otimes
\left(
\bigl[f_{\pm}^{-1}|h_{2,\pm}^{\varepsilon_2}\bigr]-\bigl[f_{\pm}^{-1}h_{2,\pm}^{\varepsilon_2}f_{\pm}|f_{\pm}^{-1}\bigr]+ \bigl[h_{1,\pm}|f_\pm^{-1}h_{2,{\pm}}^{\varepsilon_2}f_{\pm}\bigr]
\right)
\\
{}+
\left\{
\begin{aligned}
&0&&\text{if $\varepsilon_2=1$,}\\
&P_{M_{2}^{\pm}}\otimes\left([1|1]-[h_{2,{\pm}}|h_{2,{\pm}}^{-1}]\right)&&\text{if $\varepsilon_2=-1$.}\\
\end{aligned}
\right.
\end{multline*}
Hence $U_1$ lies in the image of~$\partial''$. Also $U_1\in\Delta$, since $h_{1,{\pm}}f_{\pm}^{-1}h_{2,\pm}^{\varepsilon_2}f_{\pm}\in\I_{M_{1}^{\pm}}$ and 
$\sigma\bigl(h_{1,+}f_+^{-1}h_{2,+}^{\varepsilon_2}f_+\bigr)=\sigma\bigl(h_{1,-}f_-^{-1}h_{2,-}^{\varepsilon_2}f_-\bigr)$.

By the induction hypothesis, there exists $X_1\in B_{1,2}$ such that $\partial'' X_1=U_1$, $\partial' X_1\in \Gamma$, $\Theta_A(\partial' X_1)=\Phi_{C_0,A}(U_1)$ whenever $A\subset C_0$,  and $\Theta_A(\partial' X_1)=0$ whenever $A\not\subset C_0$.

Suppose that $A\in \CH_0'$ and $\rho_0,\rho_1,\rho_2,\rho_3$ is a permutation of the four corresponding Birman--Craggs homomorphisms. If $A\not\subset C_0$, then we obviously have 
$$
\Theta_A^{\rho_0,\rho_1,\rho_2,\rho_3}\bigl(\partial'(Z_+-Z_-)\bigr)=0.
$$ 
If $A\subset C_0$, then using that $\rho_i(h_{k,+})=\rho_i(h_{k,-})$ for $i=0,1,2,3$ and $k=1,2$, we easily obtain that 
$$
\Theta_A^{\rho_0,\rho_1,\rho_2,\rho_3}\bigl(\partial'(Z_+-Z_-)\bigr)=\sum_{i\ne j}\bigl(\rho_i(f_+)+\rho_i(f_-)\bigr)\rho_j(h_{2,+})=\Phi_{C_0,A}(U_1-U).
$$
Thus, the element 
\begin{equation*}
X=X_1-Z_++Z_-
\end{equation*}
satisfies $\partial''X=U$, $\partial'X\in\Gamma$, and equality~\eqref{eq_Delta_main_1} for all~$A$.

\textsl{Case 3: $C_0\in\CH_1\setminus\left(\CH_1^{(1)}\cup\CH_1^{(2)}\right)$.} Then $\Psi_C(U)=0$ for all $C\in\CH_1^{(1)}$ and $\Phi_{C,A}(U)=0$ for all $C\in\CH_1^{(2)}$. Since $U$ belongs to the image of $\partial''$, we see that $U=\partial''X$ for some 
$$
X=\sum_{i=1}^m\pm P_{M_i}\otimes[g_i|h_i],
$$
where $M_i\in\M_1\setminus\left(\M_1^{(1)}\cup\M_1^{(2)}\right)$. Suppose that $A\in\CH_0'$. Then every multicurve~$M_i$ either does not contain an oriented multicurve~$N$ with $[N]=A$ or contains exactly two such multicurves. In both cases, we obtain that $\Theta_A^{\rho_0,\rho_1,\rho_2,\rho_3}\left(\partial'(P_{M_i}\otimes[g_i|h_i])\right)=0$ for any permutation of the Birman--Craggs homomorphisms corresponding to~$A$. Thus, $\partial'X\in\Gamma$ and $\Theta_A(\partial'X)=0$ for all $A\in\CH_0'$.
\end{proof}

\section{Scheme of the proof of Proposition~\ref{propos_d2}}\label{section_scheme}

The main difficulty in proving Proposition~\ref{propos_d2} consists in the fact that we have no good description of the kernel of the differential $d^1\colon E^1_{2,1}\to E^1_{1,1}$ and hence of the group~$E^2_{2,1}$. The reason is that the differential~$d^1$ mixes different summands in the direct sum decomposition~\eqref{eq_CL_2} for~$E^1_{2,1}$. This difficulty is overcome by first replacing the condition~$d^1y=0$ with a weaker system of linear conditions each of which can be checked by looking on the component of~$y$ in only one summand of decomposition~\eqref{eq_CL_2}, see Proposition~\ref{propos_1} below, and then formulating and proving a proper analogue of Proposition~\ref{propos_d2} for any  $y\in E^1_{2,1}$ satisfying these weaker conditions, see Proposition~\ref{propos_2} below. 

In formulations and proofs of Propositions~\ref{propos_1} and~\ref{propos_2} we extensively use homomorphisms constructed in Sections~\ref{section_several_hom} and~\ref{section_Delta}.

\begin{propos}\label{propos_1}
Suppose that $y\in E^1_{2,1}$ and $d^1y=0$. Then

\textnormal{(a)} $\sigma_D(y)=0$ for all $D\in\CH_2'$,

\textnormal{(b)} $\nu_D(y)$ is divisible by~$4$ for all $D\in\CH_2'$.
\end{propos}

\begin{propos}\label{propos_2}
Suppose that $y\in E^1_{2,1}$ is an element such that $\sigma_D(y)=0$ and $\nu_D(y)$ is divisible by~$4$ for all~$D\in \CH_2'$. Then there exist elements $Y\in B_{2,1}$ and $X\in B_{1,2}$ satisfying the following conditions:
\begin{itemize}
\item $\partial''Y=0$ and $Y$ represents~$y$ in~$E^1_{2,1}$, 
\item $\partial'Y+\partial''X\in\Delta$, 
\item $\partial'X\in\Gamma$, 
\item for all $A\in\CH_0'$, we have $\Theta_A(\partial'X)=0$  and
\begin{multline}\label{eq_key2}
\sum_{C\in\CH_1^{(1)},\, C\supset A}\Psi_C(\partial'Y+\partial''X)+
\sum_{C\in\CH_1^{(2)},\, C\supset A}\Phi_{C,A}(\partial'Y+\partial''X)\\{}=
\sum_{a\in A}\mu_{[A,a]}(d^1y)+\sum_{C\in\CH_1^{(2)},\, C\supset A}\nu^+_{C,A}(d^1y)\pmod 2,
\end{multline}
where $[A,a]$ denotes the four-element multiset obtained from~$A$ by duplicating the element~$a$.
\end{itemize}
\end{propos}

The proofs of Propositions~\ref{propos_1} and~\ref{propos_2} will be given in Sections~\ref{section_1} and~\ref{section_2}, respectively. Now, we deduce Proposition~\ref{propos_d2} from them.

\begin{proof}[Proof of Proposition~\ref{propos_d2}]
Every element in~$E^2_{2,1}$ can be represented by a cycle $y\in E^1_{2,1}$ satisfying $d^1y=0$. By Proposition~\ref{propos_1} the cycle~$y$ satisfies the conditions in Proposition~\ref{propos_2}. Let $Y\in B_{2,1}$ and~$X\in B_{1,2}$ be elements provided by this proposition. 

Since $d^1y=0$, the element $\partial'Y$ lies in the image of~$\partial''$. Hence,  $\partial'Y+\partial''X$ lies in the intersection of the group~$\Delta$ and the image of~$\partial''$. By Proposition~\ref{propos_Delta}, there exists an element
$X_1\in B_{1,2}$ such that $\partial''X_1=\partial'Y+\partial''X$, $\partial' X_1\in \Gamma$ and, for all $A\in \CH_0'$,
\begin{equation*}
\Theta_A(\partial' X_1)=\sum_{C\in\CH_1^{(1)},\, C\supset A}\Psi_C(\partial'Y+\partial''X)+
\sum_{C\in\CH_1^{(2)},\, C\supset A}\Phi_{C,A}(\partial'Y+\partial''X)=0.
\end{equation*}
(The latter equality is true, since $d^1y=0$ and hence the right-hand side of~\eqref{eq_key2} vanishes.)

Since $\partial''(X-X_1)=-\partial' Y$, we obtain using Fact~\ref{fact_differential} that the element $\partial'(X-X_1)$ represents the class~$d^2y$ in~$E^2_{0,2}$. For all $A\in \CH_0'$, we know that $\Theta_A(\partial'X)=0$ and it has been shown above that $\Theta_A(\partial'X_1)=0$. Since by Proposition~\ref{propos_Theta_induce} the homomorphism~$\Theta_A$ induces~$\btheta_A$ in page~$E^1$ and hence in page~$E^2$, we finally obtain that $\btheta_A(d^2y)=0$.
\end{proof}

\section{Proof of Proposition~\ref{propos_1}}\label{section_1}

\subsection{Part~(a)}
We put $s_D=\sigma_D(y)$ for $D\in\CH'_2(L)$. Obviously, $s_D=0$ for all but a finite number of sets~$D$.

\begin{lem}\label{lem_eq_sigma}
The elements $s_D\in\BB_3'$, where $D\in\CH_2'$, satisfy the following system of linear equations: 
\begin{itemize}
\item for each $C\in\CH_1^{(1)}$,
\begin{equation}\label{eq_ses1}
\sum_{D\in\CH_2',\,D\supset C}s_D=0,
\end{equation}
\item for each $C\in\CH_1^{(2)}$ and each non-special element~$c\in C$, 
\begin{equation}\label{eq_ses2}
s_{C\cup\{c+d\}}+s_{C\cup\{c-d\}}+s_{C\cup\{d-c\}}=0,
\end{equation}
where $d$ is the special element of~$C$.
\end{itemize}
\end{lem}

\begin{proof}
Equalities~\eqref{eq_ses1} follow from Proposition~\ref{propos_sigma_d1}. Equalities~\eqref{eq_ses2} follow from Proposition~\ref{propos_sigma2_d1}, since it is easy to see that $C\cup\{c+d\}$, $C\cup\{c-d\}$, and $C\cup\{d-c\}$ are exactly all the sets~$D\in\CH_2'$ such that $D\supset C$ and~$c$ is a principal element of~$D$.
\end{proof}

We say that a solution~$\{s_D\}$ for the system of linear equations~\eqref{eq_ses1}, \eqref{eq_ses2} is \textit{finite} if $s_D=0$ for all but a finite number of sets~$D$.

\begin{lem}\label{lem_sol_sigma}
The system of equations~\eqref{eq_ses1}, \eqref{eq_ses2} has no nonzero finite solutions.
\end{lem}

\begin{proof}
For each set $A=\{a_1,a_2,a_3\}\in\CH_0'$, we put $n(A)=n_1+n_2+n_3$, where $n_i$ are the coefficients in the decomposition $x=n_1a_1+n_2a_2+n_3a_3$.

Assume that $\{s_D\}$ is a finite solution for the system of equations~\eqref{eq_ses1},~\eqref{eq_ses2}. Let $\Xi$  be the subset of~$\CH_0'$ consisting of all~$A$ that are contained in at least one set~$D$ such that $s_D\ne 0$. Then the set~$\Xi$ is finite. 

Let us prove that the set~$\Xi$ is empty. Assume the converse, and choose the set $A_0=\{a_1,a_2,a_3\}\in \Xi$ (any if several) with the greatest~$n(A_0)$.

Consider the $9$ sets $\{a_i-a_j,a_j,a_k\}$ and $\{a_i-a_j-a_k,a_j,a_k\}$, where $i,j,k$ is a permutation of~$1,2,3$. By Proposition~\ref{propos_H0}, all these sets belong to~$\CH_0'$, and it is easy to compute that $n(A)>n(A_0)$ for each of them. Hence $s_D= 0$ for every $D\in\CH_2'$ that contains at least one of these $9$ sets~$A$. Using Proposition~\ref{propos_H}, it is not hard to check that there are exactly $9$ sets $D\in\CH_2'$ that contain~$A_0$ and contain none of the $9$ listed sets $A$, namely,
\begin{itemize}
\item $3$ sets $D_k=A_0\cup\{a_i+a_k,a_j+a_k\}$, 
\item $3$ sets $D_k^+=A_0\cup\{a_i+a_j,a_i+a_j+a_k\}$,
\item $3$ sets $D_k^-=A_0\cup\{a_i+a_j,a_i+a_j-a_k\}$.
\end{itemize}

Equation~\eqref{eq_ses2} for $C= A_0\cup\{a_i+a_k\}$ and $c=a_k$ yields that $s_{D_k}=0$.

Equation~\eqref{eq_ses1} for $C=A_0\cup\{a_i+a_j-a_k\}$ yields that $s_{D_k^-}=0$.

Then equation~\eqref{eq_ses2} for $C= A_0\cup\{a_i+a_j\}$ and $c=a_i+a_j$ yields that $s_{D_k^+}+s_{D_k^-}=0$ and hence $s_{D_k^+}=0$.

Thus,  $s_D=0$ for all $D\in\CH_2'$ containing~$A_0$, which contradicts the assumption $A_0\in \Xi$. Hence, the set $\Xi$ is empty. Therefore, $s_D=0$ for all $D\in\CH_2'$ that contain at least one set $A\in\CH_0'$.

It is not hard to check that any set $D\in\CH_2'$ that contains no subsets belonging to~$\CH_0'$ has the form $D=\{x,c_1,x-c_1,c_2,x-c_2\}$, where 
$\{x,c_1,c_2\}$ is a basis of a Lagrangian subgroup~$L$ of~$H$. Consider equation~\eqref{eq_ses1} for $C=\{c_1,x-c_1,c_2,x-c_2\}$. It is easy to see that any set $D'\in\CH_2'$ that contains~$C$ and is different from~$D$ contains a subset belonging to~$\CH_0'$. Hence $s_{D'}=0$ for all such~$D'$. Therefore, $s_D=0$. 

Thus, $s_D=0$ for all $D\in\CH_0'$.
\end{proof}

Part~(a) of Proposition~\ref{propos_1} follows immediately from Lemmas~\ref{lem_eq_sigma} and~\ref{lem_sol_sigma}.

\subsection{Part~(b)}\label{subsection_b}
First, let us prove that the numbers~$\nu_D(y)$ are all even.
The modulo~$2$ reduction of every~$D\in\CH_2'$ has the form $\{\ba_1,\ba_2,\ba_3,\ba_1+\ba_2,\ba_1+\ba_2+\ba_3\}$ for certain~$\ba_1$, $\ba_2$, and~$\ba_3$. Then the modulo~$2$ reduction of the principal element~$d\in D$ is~$\ba_1+\ba_2$. Extend $\ba_1,\ba_2,\ba_3$ to a symplectic basis $\ba_1,\ba_2,\ba_3,\bb_1,\bb_2,\bb_3$ of~$H_{\Z/2}$, and let $\rho_0,\ldots,\rho_3$ be the  associated four  Birman--Craggs homomorphisms  numbered so that the corresponding quadratic functions $\omega_0,\ldots,\omega_3$ satisfy~\eqref{eq_omega_numeration}.

Suppose that $K\in\M_2$ is an oriented multicurve with $[K]=D$, and let $\delta$ be the principal component of it. By Proposition~\ref{propos_liftrho2} we have
$$
\sigma(h)(\omega_0)+\sigma(h)(\omega_1)=\rho_0(h)+\rho_1(h)=\nu_{\delta}(h)\mod 2
$$
for all $h\in \I_K$. It follows easily that 
\begin{equation}\label{eq_sigma_nu}
\sigma_D(y)(\omega_0)+\sigma_D(y)(\omega_1)=\nu_{D}(y)\mod 2
\end{equation}
for all $y\in E^1_{2,1}$. By part~(a) of Proposition~\ref{propos_1}, which has already been proved, we know that the left-hand side of~\eqref{eq_sigma_nu} vanishes and hence $\nu_D(y)$ is even.

Now, for all~$D\in\CH_2'$, we put 
$$
\lambda_D=\frac{\nu_D(y)}{2} \mod 2.
$$
Obviously, all but a finite number of elements~$\lambda_D$ are equal to zero.

\begin{lem}\label{lem_eq_lambda}
The elements $\lambda_D\in\Z/2$, where $D\in\CH_2'$, satisfy the following system of linear equations: 
\begin{itemize}
\item for each $C=\{c_0,c_1,c_2,c_3\}\in\CH_1^{(1)}$ such that $c_0=c_1+c_2+c_3$, 
\begin{equation}\label{eq_se2}
\lambda_{C\cup\{c_1+c_2\}}=
\lambda_{C\cup\{c_2+c_3\}}=
\lambda_{C\cup\{c_3+c_1\}},
\end{equation}
\item for each $C=\{c_0,c_1,c_2,c_3\}\in\CH_1^{(1)}$ such that $c_0+c_3=c_1+c_2$, \begin{equation}
\label{eq_se3}
\begin{split}
\lambda_{C\cup\{c_1+c_2\}}&=
\lambda_{C\cup\{c_1-c_3\}}+\lambda_{C\cup\{c_3-c_1\}}=
\lambda_{C\cup\{c_2-c_3\}}+\lambda_{C\cup\{c_3-c_2\}},
\end{split}
\end{equation}
\item for each $C=\{c_0,c_1,c_2,c_3\}\in\CH_1^{(2)}$ with the special element~$c_3$, 
\begin{multline}\label{eq_se1}
\lambda_{C\cup\{c_0+c_3\}}+\lambda_{C\cup\{c_0-c_3\}}+\lambda_{C\cup\{c_3-c_0\}}=
\lambda_{C\cup\{c_1+c_3\}}+\lambda_{C\cup\{c_1-c_3\}}+\lambda_{C\cup\{c_3-c_1\}}\\
{}=\lambda_{C\cup\{c_2+c_3\}}+\lambda_{C\cup\{c_2-c_3\}}+\lambda_{C\cup\{c_3-c_2\}}.
\end{multline}
\end{itemize}
\end{lem}

\begin{proof}
First, let us prove~\eqref{eq_se2}. Put $D_i=C\cup\{c_j+c_k\}$ for every permutation~$i,j,k$ of~$1,2,3$. Consider an oriented multicurve~$M=\alpha_0\cup\alpha_1\cup\alpha_2\cup\alpha_3$ such that~$[\alpha_i]=c_i$, $i=0,1,2,3$. Arrange~$M$ as the multicurve shown in Figs.~\ref{fig_psi} and~\ref{fig_lantern_type1}, and use the same notation for curves in these figures. Consider the six five-component oriented multicurves
\begin{align*}
K_1^+&=M\cup\gamma_1,& 
K_2^+&=M\cup\delta_1,& 
K_3^+&=M\cup\varepsilon_1,\\ 
K_1^-&=M\cup\gamma_2,& 
K_2^-&=M\cup\delta_2,& 
K_3^-&=M\cup\varepsilon_2,
\end{align*}
where the components $\gamma_i$, $\delta_i$, and~$\varepsilon_i$ are oriented so that their homology classes are $c_2+c_3$, $c_3+c_1$, and~$c_1+c_2$, respectively.
Without less of generality we may assume that $K_i^+\in\orb^+_{D_i}$ and~$K_i^-\in\orb^-_{D_i}$. Note that these are exactly all the six $\I$-orbits of multisets in~$\CH_2$ that contain~$C$. We can write $y$ in the form
$$
y =\sum_{i=1}^3\left([h_i^+]_{K_i^+}+[h_i^-]_{K_i^-}\right)+\ldots,
$$
where dots stay for elements~$[h]_D$ such that $D\not\supset C$. Since the incidence coefficients $[P_{K^+_i}\colon P_M]$ and $[P_{K^-_i}\colon P_M]$ are equal to each other, we have
$$
0=d^1y=\sum_{i=1}^3\pm\left([h_i^+]_{M}+[h_i^-]_{M}\right)+\dots,
$$
where dots stay for elements~$[h]_{M'}$ such that $[M']\ne C$. Therefore, the following equality holds in~$H_1(\I_M;\Z)$ for some choice of signs~$\pm$:
$$
\sum_{i=1}^3\pm\left([h_i^+]+[h_i^-]\right)=0.
$$
Applying the homomorphisms~$\xi_1$ and~$\xi_2$ introduced in Section~\ref{subsection_type1}, we get
\begin{equation}\label{eq_sum_xi_h}
\sum_{i=1}^3\pm\left(\xi_s(h_i^+)+\xi_s(h_i^-)\right)=0,\qquad s=1,2.
\end{equation}
It follows from Proposition~\ref{propos_Stab_5comp} that each stabilizer~$\I_{K_i^{\pm}}$ is generated by twists about bounding pairs that are disjoint from~$M$ and separate~$\alpha_0\cup\alpha_i$ from~$\alpha_j\cup\alpha_k$, where $i,j,k$ is a permutation of~$1,2,3$. Hence, by Proposition~\ref{propos_xi_eps} we have 
\begin{align*}
\xi_1(h_1^+)&=-\nu_{\gamma_1}(h_1^+),&
\xi_1(h_2^+)&=0,&
\xi_1(h_3^+)&=\nu_{\varepsilon_1}(h_3^+),\\
\xi_1(h_1^-)&=\nu_{\gamma_2}(h_1^-),&
\xi_1(h_2^-)&=0,&
\xi_1(h_3^-)&=-\nu_{\varepsilon_2}(h_3^-).
\end{align*}
Therefore
\begin{multline*}
\sum_{i=1}^3\pm\left(\xi_1(h_i^+)+\xi_1(h_i^-)\right)=
\pm\left(\nu_{\gamma_1}(h_1^+)-\nu_{\gamma_2}(h_1^-)\right)
\pm\left(\nu_{\varepsilon_1}(h_3^+)-\nu_{\varepsilon_2}(h_1^-)\right)\\
{}=\pm\nu_{D_1}(y)\pm\nu_{D_3}(y),
\end{multline*}
where signs~$\pm$ on different sides of equalities are not supposed to agree each other.
Since both~$\nu_{D_1}(y)$ and~$\nu_{D_3}(y)$ are even, equality~\eqref{eq_sum_xi_h} for $s=1$ reads as $\lambda_{D_1}=\lambda_{D_3}$. Similarly, equality~\eqref{eq_sum_xi_h} for $s=2$ reads as $\lambda_{D_2}=\lambda_{D_3}$. Thus, we obtain~\eqref{eq_se2}.

The proof of~\eqref{eq_se3} is completely similar, with the only difference that we have two possibilities to orient each of the curves~$\gamma_1$, $\gamma_2$, $\delta_1$, and~$\delta_2$.

Now, let us prove~\eqref{eq_se1}. For every~$i=0,1,2$, the sets $C\cup\{c_i+c_3\}$, $C\cup\{c_i-c_3\}$, and $C\cup\{c_3-c_i\}$ are all the three sets $D\in\CH_2'$ such that $D\supset C$ and the principal element of~$D$ is~$c_i$. Hence formula~\eqref{eq_nu_d1} in Proposition~\ref{propos_nu_d1} yields that
$$
\pm\nu_{C\cup\{c_i+c_3\}}(y)\pm\nu_{C\cup\{c_i-c_3\}}(y)\pm\nu_{C\cup\{c_3-c_i\}}(y)=2\sum_{D\in\CH_2\setminus\CH_2',\,D\supset C} \pm\mu_D(y)
$$
for some choice of signs~$\pm$. Since the reduction of the right-hand side modulo~$4$ is independent of~$i$, we obtain the required equality~\eqref{eq_se1}.
\end{proof}

\begin{lem}\label{lem_sol_lambda}
The system of equations~\eqref{eq_se2}--\eqref{eq_se1} has no nonzero finite solutions.
\end{lem}

\begin{proof}
Consider a linear function $f\colon H\to\R$ such that $f(x)=0$ and the image of~$f$ is a subgroup of~$\R$ of rank~$5$. The latter condition implies that $f(c)\ne 0$ unless $c$ is proportional to~$x$.

To each set $A\in\CH_0$ (not necessarily belonging to~$\CH_0'$), we assign a pair of nonnegative real numbers
\begin{align*}
F_1(A)&=\max_{a,a'\in A}|f(a)-f(a')|,\\
F_2(A)&=\sum_{a\in A}|f(a)|,
\end{align*}
and put $F(A)=\bigl(F_1(A),F_2(A)\bigr)$. 
Note that if we replace the function~$f$ with~$-f$, then the map~$F$ will not change. Hence, we may reverse the sign of~$f$ whenever it is convenient for us.

We endow the set $\R_{\ge 0}\times\R_{\ge 0}$ with the lexicographic order, that is, put 
$(r_1,r_2)\succ(r_1',r_2')$ whenever either $r_1>r_1'$ or $r_1=r_1'$ and $r_2>r_2'$.

Assume that $\{\lambda_D\}$ is a nonzero finite solution of the system of equations~\eqref{eq_se2}--\eqref{eq_se1}.
Let~$\Xi$ be the subset of~$\CH_0$ consisting of all~$A$ that are contained in at least one set~$D\in\CH_2'$ such that $\lambda_D=1$. (Note that we do not require that $A\in\CH_0'$ but we do require that $D\in\CH_2'$.) Since $\lambda_D=1$ for at least one set $D$ and every set $D\in\CH_2'$ contains at least two (in fact even at least three) different subsets $A\in\CH_0$, we obtain that $\Xi$ is nonempty and, moreover, contains at least one set~$A$ different from~$\{x\}$. On the other hand, $\lambda_D=0$ for all but a finite number of~$D$. Hence the set~$\Xi$ is finite. 

We choose $A_0\in\Xi$ (any if several) with the greatest~$F(A_0)$ with respect to~$\succ$. 
Obviously, $F(\{x\})=(0,0)$, and $F(A)\succ (0,0)$ for all other $A\in\CH_0$. Since $\Xi$ contains at least one set different from~$\{x\}$, it follows that $A_0\ne\{x\}$. So $A_0$ consists of either two or three elements. Consider these two cases.

\textsl{Case 1: $A_0$ consists of $3$ elements, i.\,e. belongs to~$\CH_0'$.} Suppose that $A_0=\{a_1,a_2,a_3\}$.  Then $x=n_1a_1+n_2a_2+n_3a_3$ for some strictly positive integer coefficients $n_1$, $n_2$, and~$n_3$. Since $f(c)\ne 0$ unless $c$ is proportional to~$x$, we see that the three values $f(a_i)$ are nonzero and the ratio of any two of them is irrational. Besides, we have 
$$
n_1f(a_1)+n_2f(a_2)+n_3f(a_3)=f(x)=0.
$$  
Hence at least one of the values~$f(a_i)$ is positive and at least one of the values~$f(a_i)$ is negative. Permuting the elements of~$A_0$ and reversing if necessary the sign of~$f$, we may achieve that 
$$
f(a_1)>f(a_2)>0>f(a_3).
$$
We put $r_i=|f(a_i)|$. Then $r_1>r_2>0$, $r_3>0$, $F_1(A_0)=r_1+r_3$, and $F_2(A_0)=r_1+r_2+r_3$. 

Consider the sets $A_1,\ldots,A_8$ listed in Table~\ref{table_A}. It follows easily from Proposition~\ref{propos_H0} that they all belong to~$\CH_0$. Compute the values~$F_1(A_j)$ for them. For each of these sets, except for~$A_4$ in case $r_1\ge r_2+r_3$, we have $F_1(A_j)>F_1(A_0)$ and hence $F(A_j)\succ F(A_0)$. In the only remaining case of the set~$A_4$ when $r_1\ge r_2+r_3$, we have $F_1(A_4)=r_1+r_3=F_1(A_0)$ and $F_2(A_4)=r_1+r_2+2r_3>F_2(A_0)$, so we also have $F(A_4)\succ F(A_0)$. Therefore, none of the sets $A_1,\ldots, A_8$ belongs to~$\Xi$. Thus, $\lambda_D=0$ for any $D$ that contains at least one of the sets $A_1,\ldots, A_8$. 

\begin{table}
\caption{The sets~$A_j$ and the values~$F_1(A_j)$}\label{table_A}
\begin{tabular}{|l|l|c|c|c|}
\hline
&\multicolumn{1}{c|}{$A_j$}&\multicolumn{1}{c|}{$F_1(A_j)$}\\
\hline
$A_1$&
$\begin{aligned}
&\{a_1+a_2,a_3,a_1\}&&\text{if }n_1>n_2\\
&\{a_1+a_2,a_3,a_2\}&&\text{if }n_1<n_2\\
&\{a_1+a_2,a_3\}&&\text{if }n_1=n_2
\end{aligned}$
& $r_1+r_2+r_3$ \\
\hline
$A_2$&
$\{a_1-a_3,a_2,a_3\}$ &
 $r_1+2r_3$\\
\hline
$A_3$&
$\{a_3-a_1,a_1,a_2\}$ &
 $2r_1+r_3$\\
\hline
$A_4$&
$\{a_2-a_3,a_1,a_3\}$ &
 $\max\{r_1+r_3,r_2+2r_3\}$\\
\hline
$A_5$&
$\{a_3-a_2,a_1,a_2\}$ &
 $r_1+r_2+r_3$\\
\hline
$A_6$&$\begin{aligned}
&\{a_1+a_2-a_3,a_3,a_1\}&&\text{if }n_1>n_2\\
&\{a_1+a_2-a_3,a_3,a_2\}&&\text{if }n_1<n_2\\
&\{a_1+a_2-a_3,a_3\}&&\text{if }n_1=n_2
\end{aligned}$ &
 $r_1+r_2+2r_3$\\
\hline
$A_7$& $\{a_3-a_1-a_2,a_1,a_2\}$ &  $2r_1+r_2+r_3$\\
\hline
$A_8$& $\begin{aligned}
&\{a_2+a_3-a_1,a_1,a_2\}&&\text{if }n_2>n_3\\
&\{a_2+a_3-a_1,a_1,a_3\}&&\text{if }n_2<n_3\\
&\{a_2+a_3-a_1,a_1\}&&\text{if }n_2=n_3
\end{aligned}$ &  $2r_1-r_2+r_3$\\
\hline
\end{tabular}
\end{table}

Using Proposition~\ref{propos_H}, it is not hard to check that there are exactly $14$ sets $D\in\CH_2'$ that contain~$A_0$ but contain none of the sets $A_1,\ldots, A_8$, namely:
\begin{align*}
D_1&=A_0\cup\{a_1-a_2,\,a_1+a_3\},&
D_2&=A_0\cup\{a_1-a_2,\,a_2+a_3\},\\
D_3&=A_0\cup\{a_2-a_1,\,a_1+a_3\},&
D_4&=A_0\cup\{a_2-a_1,\,a_2+a_3\},\\
D_5&=A_0\cup\{a_1+a_3,\,a_2+a_3\},&
D_6&=A_0\cup\{a_1-a_2,\,a_1-a_2+a_3\},\\
D_7&=A_0\cup\{a_2-a_1,\,a_1-a_2+a_3\},&
D_8&=A_0\cup\{a_1-a_2,\,a_1-a_2-a_3\},\\
D_9&=A_0\cup\{a_2-a_1,\,a_2-a_1-a_3\},&
D_{10}&=A_0\cup\{a_1+a_3,\,a_1+a_2+a_3\},\\
D_{11}&=A_0\cup\{a_1+a_3,\,a_1-a_2+a_3\},&
D_{12}&=A_0\cup\{a_1+a_3,\,a_2-a_1-a_3\},\\
D_{13}&=A_0\cup\{a_2+a_3,\,a_1+a_2+a_3\},&
D_{14}&=A_0\cup\{a_2+a_3,\,a_1-a_2-a_3\}.
\end{align*}

We put $\lambda_j=\lambda_{D_j}$. Let us prove that $\lambda_{j}=0$, $j=1,\ldots,14$. 

Equation~\eqref{eq_se2} for~$C=A_0\cup\{a_1+a_2+a_3\}$ reads as $\lambda_{10}=\lambda_{13}=0$.

Equation~\eqref{eq_se2} for~$C=A_0\cup\{a_1-a_2-a_3\}$ reads as $\lambda_{8}=\lambda_{14}=0$.

Equation~\eqref{eq_se2} for~$C=A_0\cup\{a_2-a_1-a_3\}$ reads as $\lambda_{9}=\lambda_{12}=0$.

Equation~\eqref{eq_se3} for~$C=A_0\cup\{a_1-a_2+a_3\}$ reads as $\lambda_{6}+\lambda_7=\lambda_{11}=0$.

Equations~\eqref{eq_se1} for~$C=A_0\cup\{a_1+a_3\}$, $C=A_0\cup\{a_2+a_3\}$, and $C=A_0\cup\{a_1-a_2\}$ read as
\begin{gather*}
\lambda_1+\lambda_3=\lambda_5=\lambda_{10}+\lambda_{11}+\lambda_{12},\\
\lambda_2+\lambda_4=\lambda_5=\lambda_{13}+\lambda_{14},\\
\lambda_1=\lambda_2=\lambda_6+\lambda_8,
\end{gather*}
respectively.
Hence $\lambda_5=0$ and $\lambda_1=\lambda_2=\lambda_3=\lambda_4=\lambda_6=\lambda_7$.

Finally, consider the two sets
\begin{gather*}
C'=\{a_1,a_2-a_1,a_3,a_2+a_3\},\\
D'=\{a_1,a_2-a_1,a_3,a_2+a_3,a_2+a_3-a_1\}
\end{gather*}
Applying Proposition~\ref{propos_H} to the set $\{a_1,a_2-a_1,a_3\}$, which belongs to~$\CH_0'$ by Proposition~\ref{propos_H0}, we see that $C'\in\CH_1^{(1)}$ and~$D'\in\CH_2'$. Equation~\eqref{eq_se2} for $C=C'$ reads as 
$\lambda_4=\lambda_{D'}$.
If $n_2\le n_3$, then $D'\supset A_8$ and hence $\lambda_{D'}=0$. If $n_2>n_3$, then $D'$ contains the subset $A'=\{a_1,a_2-a_1,a_2+a_3-a_1\}$, which by Proposition~\ref{propos_H0} belongs to~$\CH_0'$. We have 
\begin{gather*}
F_1(A')=2r_1-r_2+r_3> F_1(A_0).
\end{gather*}
Therefore $F(A')\succ F(A_0)$ and hence $A'\notin\Xi$. Consequently, we again obtain that $\lambda_{D'}=0$.

Thus, $\lambda_j =0$ for $j=1,\ldots,14$. Hence $\lambda_D=0$ for all $D\supset A_0$, which yields a contradiction, since $A_0\in\Xi$.

\textsl{Case 2: $A_0$ consists of $2$ elements.} Suppose that $A_0=\{a_1,a_2\}$. Then $x=n_1a_1+n_2a_2$ for some positive integers~$n_1$ and~$n_2$. 
We have $f(a_1)\ne 0$ and $n_1f(a_1)+n_2f(a_2)=0$. Hence $f(a_1)=n_2r$ and $f(a_2)=-n_1r$ for some nonzero~$r$. Without loss of generality we may assume that $r=1$. Then $$F(A_0)=(n_1+n_2,n_1+n_2).$$  

Let $\Upsilon$ be the set consisting of all elements~$c\in H$ satisfying the following conditions:
\begin{itemize}
\item $a_1,a_2,c$ is a basis of a Lagrangian subgroup of~$H$,
\item there exists $D\in\CH_2'$ such that  $\lambda_D=1$ and $D$ contains at least one of the sets $C_1=\{a_1,a_2,c,a_1-c\}$ and $C_2=\{a_1,a_2,c,a_2-c\}$.
\end{itemize}

Since the rank of the image of~$f$ is equal to~$5$, we have $f(c)\notin\Q$ for all~$c\in\Upsilon$. Since $\lambda_D=0$ for all but a finite number of~$D$, the set $\Upsilon$ is finite. 

\begin{lem}\label{lem_Xi}
Suppose that $c\in\Upsilon$. Then either $c+a_1\in\Upsilon$ or $c+a_2\in\Upsilon$.
\end{lem}

\begin{proof}
We may assume that there exists $D\in\CH_2'$ such that $\lambda_D=1$ and $D$ contains the set $C_1$.  Put $c'=a_1-c$. By Proposition~\ref{propos_H0} the set $A_1=\{a_2,c,c'\}$ belongs to~$\CH_0'$.
Since $A_1\subset C_1$ and $\lambda_D=1$ for at least one~$D$ containing~$C_1$, we obtain that $A_1\in\Xi$. Therefore $F_1(A_1)\le F_1(A_0)=n_1+n_2$. Since 
$$
f(c)+f(c')=f(a_1)=n_2,
$$ 
we have
$$
F_1(A_1)=\max\bigl\{|n_1+f(c)|,|n_1+n_2-f(c)|, |n_2-2f(c)|\bigr\}.
$$
If $f(c)$ was either negative or greater than $n_2$, we would obtain that $F_1(A_1)>n_1+n_2$. Hence $0\le f(c)\le n_2$. But $f(c)\notin\Q$.  Therefore $0<f(c)<n_2$ and hence $0<f(c')<n_2$. 

It follows from Proposition~\ref{propos_H0} that the sets $A_2=\{a_1-a_2,a_2\}$, $A_3=\{a_1,a_2-a_1\}$, $A_4=\{a_1,c,a_2-c\}$, $A_5=\{a_1,c',a_2-c'\}$ and $A_6=\{a_2,c,c'-a_2\}$ belong to~$\CH_0$. We have
\begin{align*}
F_1(A_2)&=n_1+2n_2,&
F_1(A_3)&=2n_1+n_2,\\
F_1(A_4)&=n_1+n_2+f(c),&
F_1(A_5)&=n_1+n_2+f(c').
\end{align*}
Each of these numbers is strictly greater than $F_1(A_0)$. Hence the sets~$A_2$, $A_3$, $A_4$, and~$A_5$ do not belong to~$\Xi$. Therefore $\lambda_D=0$ for any $D$ that contains at least one of them. 

Using Proposition~\ref{propos_H} for the set~$A_1$, we see that, firstly, $C_1\in\CH_1^{(2)}$, and secondly, there are exactly $5$ sets $D\in\CH_2'$ that contain~$C_1$ but contain none of the sets~$A_2$, $A_3$, $A_4$, and~$A_5$, namely, the sets
\begin{gather*}
D_0=\{a_1,a_2,c,c',a_1+a_2\},\\
\begin{aligned}
D_1&=\{a_1,a_2,c,c',a_2+c\},&&& D_2&=\{a_1,a_2,c,c',a_2+c'\},\\
D_3&=\{a_1,a_2,c,c',c-a_2\},&&& D_4&=\{a_1,a_2,c,c',c'-a_2\}.
\end{aligned}
\end{gather*}
Hence $\lambda_{D_i}=1$ for at least one of these five sets and $\lambda_D=0$ for all sets~$D\in \CH_2'$ that contain~$C_1$ and are different from $D_0,\ldots,D_4$.

It follows from Proposition~\ref{propos_H} for~$A_6$ that the set~$\{a_2,c,c'-a_2,a_1\}$ belongs to~$\CH_1^{(1)}$. Equation~\eqref{eq_se2} for this set yields that
$$
\lambda_{D_4}=\lambda_{\{a_1,a_2,c,c'-a_2,a_1-a_2\}}=0,
$$ 
where the latter equality is true, since the set $\{a_1,a_2,c,c'-a_2,a_1-a_2\}$ contains~$A_2$. Similarly, $\lambda_{D_3}=0$.

Now, equation~\eqref{eq_se1} for $C=C_1$ yields that $\lambda_{D_0}=\lambda_{D_1}=\lambda_{D_2}$. Since $\lambda_{D_i}$ must be equal to~$1$ for at least one of these three sets, we obtain that $\lambda_{D_0}=\lambda_{D_1}=\lambda_{D_2}=1$.

Finally, consider the set $C'=\{a_1,a_2,c',a_2+c\}$. It is easy to see that there is a four-component oriented multicurve with the set of homology classes~$C'$, and this oriented multicurve belongs to~$\M_1^{(1)}$. Hence $C'\in \CH_1^{(1)}$. Equation~\eqref{eq_se3} for $C=C'$ yields that 
$$
\lambda_{D_1}+\lambda_{\{a_1,a_2,c',a_2+c,-c\}}=\lambda_{\{a_1,a_2,c',a_2+c,c'-a_2\}}+\lambda_{\{a_1,a_2,c',a_2+c,a_2-c'\}}.
$$
But $\lambda_{\{a_1,a_2,c',a_2+c,a_2-c'\}}$=0, since the set $\{a_1,a_2,c',a_2+c,a_2-c'\}$ contains~$A_5$. Hence either $\lambda_{\{a_1,a_2,c',a_2+c,-c\}}=1$ or $\lambda_{\{a_1,a_2,c',a_2+c,c'-a_2\}}=1$. In both cases, we see that $c+a_2\in\Upsilon$.

Similarly, if there is $D\in\CH_2'$ such that  $\lambda_D=1$ and $D\supset C_2$, then we obtain that $c+a_1\in\Upsilon$.
\end{proof}

\begin{cor}\label{cor_Xi}
The set~$\Upsilon$ is empty.
\end{cor}

\begin{proof}
Suppose that $\Upsilon$ contained an element~$c$. Applying recursively Lemma~\ref{lem_Xi}, we would obtain that, for every~$k=1,2,\ldots$, there exists a partition $k=k_1+k_2$,  $k_1,k_2\ge 0$, such that $c+k_1a_1+k_2a_2\in\Upsilon$. So we would find infinitely many pairwise different elements of~$\Upsilon$, which is impossible.
\end{proof}

Let us proceed with the proof of Lemma~\ref{lem_sol_lambda}. Since $A_0=\{a_1,a_2\}$ belongs to~$\Xi$, there is at least one set $D\in\CH_2'$ that contains~$A_0$ and satisfies $\lambda_D=1$. Then $D$ is the set of homology classes of a multicurve~$M$ as in Fig.~\ref{fig_types}(d). Hence $D$ spans a Lagrangian subgroup of~$H$, and any three linearly independent elements of~$D$ constitute a basis of this Lagrangian subgroup. Therefore there is an element $c\in D$ such that $\{a_1,a_2,c\}$ is a basis of a Lagrangian subgroup. Since $\Upsilon$ is empty, we have $c\notin\Upsilon$; hence, $D$ contains neither $a_1-c$ nor~$a_2-c$. Besides, the sets $\{a_1-a_2,a_2\}$ and  $\{a_1,a_2-a_1\}$ do not belong to~$\Xi$,  since the values of~$F_1$ on both of them are strictly greater than $F_1(A_0)=n_1+n_2$. Therefore $D$ contains neither $a_1-a_2$ nor $a_2-a_1$,

Further, it follows from the definition of the set~$\CH$ (see Section~\ref{subsection_complex_cycles}) that $x$ can be written as a linear combination of elements of~$D$ with nonnegative coefficients so that the coefficient of~$c$ is strictly positive. Since $x=n_1a_1+n_2a_2$, we obtain that $D$ must contain an element~$d$ whose coefficient of~$c$ in the decomposition over the basis~$\{a_1,a_2,c\}$ is strictly negative. But either all the four homology classes $a_1,a_2,c,d$ or some three of them satisfy a two-sided relation, i.\,e., linear relation with coefficients $\pm 1$ such that  at least one coefficient is~$1$ and at least one coefficient is~$-1$. Since we have already proved that $d\ne a_1-c$ and $d\ne a_2-c$, we obtain that $d$ is one of the $3$ homology classes $a_1+a_2-c$,  $a_1-a_2-c$, and~$a_2-a_1-c$. In all the three cases, we see that all the four homology classes $a_1,a_2,c,d$ satisfy a two-sided relation and hence none of them is the principal element of~$D$.  Let $e$ be the fifth element of~$D$; then $e$ is principal. Hence, $e$ is either the sum or the difference of some two of the four elements $a_1,a_2,c,d$, and simultaneously either the sum or the difference of the rest two of the four elements $a_1,a_2,c,d$. Besides, $e$ is not among the classes $a_1-a_2$, $a_2-a_1$, $a_1-c$, $a_2-c$, $a_1-d$, and~$a_2-d$. (The last two classes are excluded, since $d\notin\Upsilon$.) A direct case analysis shows that the only possibility is $d=a_1+a_2-c$ and $e=a_1+a_2$. Thus, $D=\{a_1,a_2,a_1+a_2,c,a_1+a_2-c\}$.

 Equation~\eqref{eq_se3} for $C=\{a_1,a_2,c,a_1+a_2-c\}\in\CH_1^{(1)}$ yields that
$$
1=\lambda_D=\lambda_{\{a_1,a_2,c,a_1+a_2-c,a_1-c\}}+\lambda_{\{a_1,a_2,c,a_1+a_2-c,c-a_1\}}.
$$
Hence either $c$ or~$c-a_1$ belongs to~$\Upsilon$, which contradicts Corollary~\ref{cor_Xi}. This contradiction competes the proof of Lemma~\ref{lem_sol_lambda}.
\end{proof}

Part~(b) of Proposition~\ref{propos_1} follows immediately from Lemmas~\ref{lem_eq_lambda} and~\ref{lem_sol_lambda}.

\section{Proof of Proposition~\ref{propos_2}}\label{section_2}
We have 
\begin{equation}\label{eq_decompose_G}
E^1_{2,1}=\bigoplus_{D\in\CH_2}E^1_{2,1}(D),
\end{equation}
where $E^1_{2,1}(D)$ is the direct sum of $H_1(\I_K;\Z)$ for representatives~$K$ of $\I$-orbits of oriented multicurves with~$[K]=D$. Obviously, an element $y\in E^1_{2,1}$ satisfies the conditions in Proposition~\ref{propos_2} if and only if all components of~$y$ with respect to decomposition~\eqref{eq_decompose_G} satisfy the same conditions. Therefore, it is sufficient to prove the claim of Proposition~\ref{propos_2} for~$y$ lying in one of the summands~$E^1_{2,1}(D)$.

We consider two cases:

\textsl{Case 1: $y\in E^1_{2,1}(D)$ for $D\in\CH_2'$.} Choose oriented multicurves~$K^+$ and~$K^-$ in the orbits~$\orb^+_K$ and~$\orb^-_K$, respectively, so that $K^+$ and~$K^-$ share four components and are taken to each other by a hyperelliptic involution~$\iota$, see Fig.~\ref{fig_Kpm}. (In this figure, $\iota$ is the rotation by~$\pi$ around the horizontal axis and $K^{\pm}=\gamma^{\pm}\cup\alpha_1\cup\alpha_2\cup\alpha_3\cup\alpha_4$.) Recall that $\iota\in\hI$. We have 
$$
E^1_{2,1}(D)=H_1(K^+;\Z)\oplus H_1(K^-;\Z).
$$
Hence 
$$
y=[h_+]_{K^+}-[h_-]_{K^-}
$$
for certain elements $h_{\pm}\in\I_{K^{\pm}}$.

\begin{figure}
\begin{tikzpicture}[scale=.45]
\small
\definecolor{myblue}{rgb}{0, 0, 0.7}
\definecolor{mygreen}{rgb}{0, 0.4, 0}
\tikzset{every path/.append style={line width=.2mm}}
\tikzset{my dash/.style={dash pattern=on 2pt off 1.5pt}}

\begin{scope}
% Invisible lines

\draw[color=myblue,my dash, very thick] (0,-3.03) .. controls (-0.3,-2) .. (0,-.97); % e_0
\draw[color=myblue,my dash, very thick] (0,3.03) .. controls (-0.3,2) .. (0,.97); % e_0'
\draw[color=myblue,my dash, very thick] (-7.03,0) .. controls (-6,0.3) .. (-4.97,0); %e_1
\draw[color=myblue,my dash, very thick] (7.03,0) .. controls (6,0.3) .. (4.97,0); %e_2
\draw[color=myblue,my dash, very thick] (-3.03,0) .. controls (-2,0.3) .. (-.97,0); % e_0+e_1
\draw[color=myblue,my dash, very thick ] (3.03,0) .. controls (2,0.3) .. (.97,0); % e_0+e_2
%\draw[color=red, my dash] (2, -3.03) .. controls (1.7,-2) and (1.7,2) .. (2, 3.03);
%\draw[color=myblue, very thick,my dash] (-3.293,-.707) .. controls (-1.5,-2.2) and (1.5,-2.2) .. (3.293,-.707);
%\draw[color=red,my dash] (-3.293,.707) .. controls (-1.5,2.8) and (1.5,2.8) .. (3.293,.707);

% Border lines

\draw [very thick] (0,0) circle (1);
\draw [very thick] (4,0) circle (1);
\draw [very thick] (-4,0) circle (1);
\draw [very thick]  (-4,3)--(4,3);
\draw [very thick]  (-4,-3)--(4,-3);
\draw [very thick] (-4,3) arc (90:270:3); 
\draw [very thick] (4,3) arc (90:-90:3); 

\draw [] (-9.5,0) -- (-7.3,0);
\draw [] (9.5,0) -- (7.3,0);

\draw (8.5,0) + (-55: .3 and 1) arc (-55: -10: .3 and 1);
\draw [-stealth] (8.5,0) + (10: .3 and 1) arc (10: 160: .3 and 1) node [pos=.4, right] {$\iota$};

% Visible lines 

\draw[color=myblue, very thick] (0,-3.03) .. controls (0.3,-2) .. (0,-.97) node [pos=0.5, right=-2pt] {$\gamma^-$};  %e_0
\draw[color=myblue, very thick] (0,3.03) .. controls (0.3,2) .. (0,.97) node [pos=0.5, right=-2pt]{$\gamma^+$};  %e_0'

\draw[color=myblue, very thick] (-4.97,0) .. controls (-6,-0.3) .. (-7.03,0) node [pos=0.4,below=-1pt] {$\alpha_1$};%e_1
%\draw[color=red,->-=.55] (4.97,0) .. controls (6,-0.4) .. (7.03,0) node [pos=.3,below=-1pt] {$\alpha_2$}; %e_2
%\draw[color=red,->-=.55] (-7.03,0) .. controls (-6,-0.3) .. (-4.97,0)node [pos=0.4,below] {$\alpha_1$};%e_1 -
\draw[color=myblue, very thick] (7.03,0) .. controls (6,-0.3) .. (4.97,0) node [pos=0.4,below=-1pt] {$\alpha_4$};%e_2 -

\draw[color=myblue, very thick] (-3.03,0) .. controls (-2,-0.3) .. (-.97,0) node [pos=0.6,below=-1pt] {$\alpha_2$};%e_0+e_1
\draw[color=myblue, very thick] (3.03,0) .. controls (2,-0.3) .. (.97,0) node [pos=0.4,below=-1pt] {$\alpha_3$};%e_0+e_2

%\draw[color=red] (2, -3.03) .. controls (2.3,-2) and (2.3,2) .. (2, 3.03) node [pos=0.35, right=-1pt] {$\delta$};  %e_0

%\draw[color=red] (-2,-3.03) .. controls (-1.7,-1.5) and (-1.7,1.5) .. (-2,3.03) node [pos=.25,right=-2pt] {$u$};
%\draw[color=red,my dash] (-2,-3.03) .. controls (-2.3,-1.5) and (-2.3,1.5) .. (-2,3.03);

%\draw[color=red] (-3.293,-.707) .. controls (-1,-2.8) and (1,-2.8) .. (3.293,-.707) node [pos=.8,below=1pt] {$\delta_2$};
%\draw[color=red] (-3.293,.707) .. controls (-1,2.2) and (1,2.2) .. (3.293,.707)node [pos=.8,below] {$\delta$};
\end{scope}

\end{tikzpicture}
\caption{Multicurves~$K^+$ and~$K^-$}\label{fig_Kpm}
\end{figure}
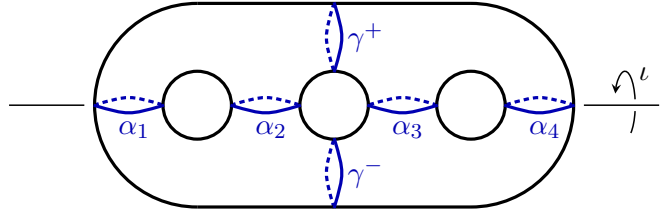

Conditions $\sigma_D(y)=0$ and~$\nu_D(y)\equiv 0\pmod 4$ read as
\begin{gather*}
\sigma(h_+)=\sigma(h_-)\qquad\text{and}\qquad\nu_{\gamma^+}(h_+)+\nu_{\gamma^-}(h_-)\equiv 0\pmod 4,
\end{gather*}
respectively.

We put $M_0=\alpha_1\cup\alpha_2\cup\alpha_3\cup\alpha_4$ and $C_0=[M_0]$. This oriented multicurve may or may not belong to~$\M$. If it belongs to~$\M$, then it obviously belongs to~$\M_1^{(1)}$. In this case we assume that the orientation of~$P_{M_0}$ is chosen  so that $[P_{K^{\pm}}\colon P_{M_0}]=1$.

Let $M_1^+,\ldots,M_m^+$ be all submulticurves of~$K^+$ that belong to~$\M_1^{(2)}$. Then every multicurve~$M_i^+$ contains~$\gamma^+$. Besides, $M_i^-=\iota(M_i^+)$, $i=1,\ldots,m$, are all submulticurves of~$K^-$ that belong to~$\M_1^{(2)}$. Put $C_i=[M_i^{\pm}]$.
The multicurves $M_i^{\pm}$ lie in the two different orbits~$\orb_{C_i}^{\pm}$. Nevertheless, we cannot guarantee that $M_i^+\in\orb_{C_i}^+$ and $M_i^-\in\orb_{C_i}^-$; the opposite situation, $M_i^+\in\orb_{C_i}^-$ and $M_i^-\in\orb_{C_i}^+$, is also possible. (If we swapped the orbits~$\orb_{C_i}^{\pm}$, we would change the homomorphisms~$\nu_{C_i,A}^+$, which enter formula~\eqref{eq_key2}, so we would not like to do it.)   
We may assume that the orientations of the cells~$P_{M_i^{\pm}}$ are chosen so that $[P_{K^+}\colon P_{M_i^+}]=[P_{K^-}\colon P_{M_i^-}]=1$.

Represent the class $y$ by the cycle 
$$
Y=P_{K^+}\otimes[h_+]-P_{K^-}\otimes[h_-].
$$
Then
$$
\partial'Y=\sum_{i=1}^m (P_{M_i^+}\otimes[h_+]-P_{M_i^-}\otimes[h_-])+
\left\{
\begin{aligned}
&P_{M_0}\otimes([h_+]-[h_-])&&\text{if }M_0\in\M,\\
&0&&\text{if }M_0\notin\M.
\end{aligned}
\right.
$$

If $M_0\notin\M$, then $\partial'Y\in\Delta$, since $\sigma(h_+)=\sigma(h_-)$. Hence, we may put $X=0$.

If $M_0\in\M$, then we put 
$$
X=-P_{M_0}\otimes[h_+h_-^{-1}|h_-],
$$ 
and obtain that
$$
\partial'Y+\partial''X=\sum_{i=1}^m \bigl(P_{M_i^+}\otimes[h_+]-P_{M_i^-}\otimes[h_-]\bigr)+P_{M_0}\otimes[h_+h_-^{-1}].
$$
Since $\sigma(h_+h_-^{-1})=0$, it follows that $\partial'Y+\partial''X\in\Delta$,
$\partial'X\in\Gamma$, and $\Theta_A(\partial'X)=0$ for all $A\in\CH_0'$. 

Let us prove formula~\eqref{eq_key2}. Since none of the multicurves~$M_0,M_1^{\pm}\ldots,M_m^{\pm}$ contains a bounding pair, we have $\mu_{[A,a]}(d^1y)=0$ for all $a\in A$.

Obviously, if $A\not\subset D$, then both sides of~\eqref{eq_key2} vanish. So we may assume that $A\subset D$. Let $N\subset K^+$ be the multicurve such that $[N]=A$. Then $P_{N}$ is a vertex of the $2$-cell~$P_{K^+}$. Hence it is contained in exactly two of the edges of~$P_{K^+}$. But edges of~$P_{K^+}$ are exactly $P_{M_1^+} ,\ldots,P_{M_m^+}$, and also $P_{M_0}$ in case $M_0\in\M$.

First, suppose that $A\not\subset C_0$. Then $\Psi_C(\partial'Y+\partial''X)=0$ for all $C\in\CH_1^{(1)}$ such that $C\supset A$. We may assume that the two edges of~$P_K^+$ containing~$P_N$ are $P_{M_1^+}$ and~$P_{M_2^+}$.
By formula~\eqref{eq_Phi},
$$
\Phi_{C_1,A}(\partial'Y+\partial''X)=\Phi_{C_2,A}(\partial'Y+\partial''X)=\sum_{i\ne j}\hrho_i(\iota)\rho_j(h_+),
$$
and $\Phi_{C,A}(\partial'Y+\partial''X)=0$ for all other $C\supset A$. 
Therefore, 
$$
\sum_{C\in\CH_1^{(2)},\, C\supset A}\Phi_{C,A}(\partial'Y+\partial''X)=0.
$$

Now, let $\delta$ be the component of~$M_1^+$ that is not contained in~$N$. Since $A\not\subset C_0$, we have $\delta\ne \gamma^+$ and hence $\delta\subset M_1^-$. By Proposition~\ref{propos_triv_restr2}, $\nu_{\delta}(h_+)=\nu_{\delta}(h_-)=0$.
Therefore $\nu^+_{C_1, A}(d^1y)=0$, independently of which of the multicurves~$M_1^{\pm}$ belongs to~$\orb_{C_1}^+$. Similarly, $\nu^+_{C_2, A}(d^1y)=0$. Finally, we obviously have $\nu_{C,A}^+(d^1y)=0$ for all $C$ different from~$C_1$ and~$C_2$.
 
Thus, we obtain~\eqref{eq_key2}.

Second, assume that $A\subset C_0$. Then $N$ is contained in $M_0$ and in exactly one of the multicurves $M_1^+,\ldots,M_m^+$. We may assume that $N\subset M_1^+$; then $M_1^+=N\cup\gamma^+$. In this case the multicurve~$M_0$ must belong to~$\M$. We have
\begin{equation}\label{eq_d2_2}
\begin{split}
\sum_{C\in\CH_1^{(1)},\, C\supset A}\Psi_C(\partial'Y+\partial''X)&=\Psi_{C_0}(\partial'Y+\partial''X)=\psi_{M_0}\bigl(h_+h_-^{-1}\bigr) \\
{}&=\psi_{M_0}\bigl(h_+\iota h_-^{-1}\iota\bigr) +\psi_{M_0}\bigl(\iota h_-\iota h_-^{-1}\bigr).
\end{split}
\end{equation}
Since $h_+\iota h_-^{-1}\iota$ lies in~$\CC_{K^+}$, it follows from Proposition~\ref{propos_psi_nu} that 
\begin{equation}\label{eq_d2_4a}
\psi_{M_0}\bigl(h_+\iota h_-^{-1}\iota\bigr)=\frac{\nu_{\gamma^+}\bigl(h_+\iota h_-^{-1}\iota\bigr)}{2}=\frac{\nu_{\gamma^+}(h_+)-\nu_{\gamma^-}(h_-)}{2}\equiv \nu_{\gamma^+}(h_+)\equiv \nu_{\gamma^-}(h_-)\pmod 2,
\end{equation}
where the last two congruences use that $\nu_{\gamma^+}(h_+)+\nu_{\gamma^-}(h_-)$ is divisible by~$4$. 
On the other hand, by Proposition~\ref{propos_psi2},
\begin{equation}\label{eq_d2_4}
\psi_{M_0}(\iota h_-\iota h_-^{-1})=\sum_{i\ne j}\hrho_i(\iota)\rho_j(h_-)=\sum_{i\ne j}\hrho_i(\iota)\rho_j(h_+).
\end{equation}
Further, by formula~\eqref{eq_Phi},
\begin{equation}\label{eq_d2_1}
\sum_{C\in\CH_1^{(2)},\, C\supset A}\Phi_{C,A}(\partial'Y+\partial''X)=\Phi_{C_1,A}(\partial'Y+\partial''X)=\sum_{i\ne j}\hrho_i(\iota)\rho_j(h_+).
\end{equation}
Finally,
\begin{equation}\label{eq_d2_5}
\sum_{C\in\CH_1^{(2)}\colon C\supset A}\nu^+_{C, A}(d^1y)=
\nu^+_{C_1,A}(d^1y)=\left\{
\begin{aligned}
&\nu_{\gamma^+}(h_+)&&\text{if }M_1^+\in\orb^+_{C_1}\,,\\
-&\nu_{\gamma^-}(h_-)&&\text{if }M_1^-\in\orb^+_{C_1}\,.
\end{aligned}
\right.
\end{equation}
Combining~\eqref{eq_d2_2}--\eqref{eq_d2_5}, we again get~\eqref{eq_key2}.

\textsl{Case 2: $y\in E^1_{2,1}(D)$ for $D\in\CH_2\setminus\CH_2'$.} By Proposition~\ref{propos_M2no'},  some element~$c$ has multiplicity~$2$ in~$D$. Note that in this case we automatically have $\sigma_{D'}(y)=0$ and $\nu_{D'}(y)=0$ for all $D'\in\CH_2'$. Hence, no further restriction should be imposed on~$y$. So it is sufficient to prove the claim of the proposition in case $y=[h]_K$, where $[K]=D$ and $h\in\I_K$, since such elements generate~$E^1_{2,1}(D)$. Let $\gamma^+$ and~$\gamma^-$ be the two components of~$K$ in homology class~$c$, and~$\alpha_1$, $\alpha_2$, and~$\alpha_3$ be the three other components of~$K$. It follows from Proposition~\ref{propos_stab_generate} that~$\I_K$ is generated by twists about separating curves~$\delta$ that are disjoint from~$K$, and by the twist~$T_{\gamma^+,\gamma^-}$. So we may assume that $h$ is either~$T_{\delta}$ or~$T_{\gamma^+,\gamma^-}$. Arrange the multicurve~$K$ and the curve~$\delta$ (in case $h=T_{\delta}$) as in Fig.~\ref{fig_K2}.

\begin{figure}
\begin{tikzpicture}[scale=.45]
\small
\definecolor{myblue}{rgb}{0, 0, 0.7}
\definecolor{mygreen}{rgb}{0, 0.4, 0}
\tikzset{every path/.append style={line width=.2mm}}
\tikzset{my dash/.style={dash pattern=on 2pt off 1.5pt}}

\begin{scope}
% Invisible lines

\draw[color=myblue,my dash, very thick] (0,-3.03) .. controls (-0.3,-2) .. (0,-.97); % e_0
\draw[color=myblue,my dash, very thick] (0,3.03) .. controls (-0.3,2) .. (0,.97); % e_0'
\draw[color=myblue,my dash, very thick] (-7.53,0) .. controls (-6.5,0.3) .. (-5.47,0); %e_1
\draw[color=myblue,my dash, very thick] (7.03,0) .. controls (6,0.3) .. (4.97,0); %e_2
\draw[color=myblue,my dash, very thick] (3.03,0) .. controls (2,0.3) .. (.97,0); % e_0+e_1
%\draw[color=red,my dash ] (3.03,0) .. controls (2,0.3) .. (.97,0); % e_0+e_2
\draw[color=red, my dash] (-2.5, -3.03) .. controls (-2.8,-2) and (-2.8,2) .. (-2.5, 3.03);
%\draw[color=myblue, very thick,my dash] (-3.293,-.707) .. controls (-1.5,-2.2) and (1.5,-2.2) .. (3.293,-.707);
%\draw[color=red,my dash] (-3.293,.707) .. controls (-1.5,2.8) and (1.5,2.8) .. (3.293,.707);

% Border lines

\draw [very thick] (0,0) circle (1);
\draw [very thick] (4,0) circle (1);
\draw [very thick] (-4.5,0) circle (1);
\draw [very thick]  (-4.5,3)--(4,3);
\draw [very thick]  (-4.5,-3)--(4,-3);
\draw [very thick] (-4.5,3) arc (90:270:3); 
\draw [very thick] (4,3) arc (90:-90:3); 

\draw [] (-10,0) -- (-7.8,0);
\draw [] (9.5,0) -- (7.3,0);

\draw (8.5,0) + (-55: .3 and 1) arc (-55: -10: .3 and 1);
\draw [-stealth] (8.5,0) + (10: .3 and 1) arc (10: 160: .3 and 1) node [pos=.4, right] {$\iota$};

% Visible lines 

\draw [color=mygreen] (-3,0) arc (0:360:1.5)node [pos=.75, below=-1pt] {$\beta_1$} ;

\draw [color=mygreen] (1.5,0) arc (0:360:1.5) node [pos=.85, right=-1pt] {$\beta_2$} ;
\draw [color=mygreen] (0,1.8) arc (90:270:1.8) -- (4,-1.8) node [pos=.75,below=-1pt] {$\beta_3$} arc (-90:90:1.8) -- (0,1.8) ;

\draw[color=myblue, very thick] (0,-3.03) .. controls (0.3,-2) .. (0,-.97) node [pos=0.3, right=-2pt] {$\gamma^-$};  %e_0
\draw[color=myblue, very thick] (0,3.03) .. controls (0.3,2) .. (0,.97) node [pos=0.25, right=-2pt]{$\gamma^+$};  %e_0'

\draw[color=myblue, very thick] (-5.47,0) .. controls (-6.5,-0.3) .. (-7.53,0) node [pos=.5, below=-1pt] {$\alpha_1$};%e_1
%\draw[color=red,->-=.55] (4.97,0) .. controls (6,-0.4) .. (7.03,0) node [pos=.3,below=-1pt] {$\alpha_2$}; %e_2
%\draw[color=red,->-=.55] (-7.03,0) .. controls (-6,-0.3) .. (-4.97,0)node [pos=0.4,below] {$\alpha_1$};%e_1 -
\draw[color=myblue, very thick] (7.03,0) .. controls (6,-0.3) .. (4.97,0) node [pos=.3, below=-1pt] {$\alpha_3$};%e_2 -

\draw[color=myblue, very thick] (3.03,0) .. controls (2,-0.3) .. (.97,0) node [pos=.3, below=-1pt] {$\alpha_2$};%e_0+e_1
%\draw[color=red] (3.03,0) node [right=-2pt] {$\delta$} .. controls (2,-0.3) .. (.97,0) ;%e_0+e_2

\draw[color=red] (-2.5, -3.03) .. controls (-2.2,-2) and (-2.2,2) .. (-2.5, 3.03) node [pos=0.2, right=-1pt] {$\delta$};  %e_0

%\draw[color=red] (-2,-3.03) .. controls (-1.7,-1.5) and (-1.7,1.5) .. (-2,3.03) node [pos=.25,right=-2pt] {$u$};
%\draw[color=red,my dash] (-2,-3.03) .. controls (-2.3,-1.5) and (-2.3,1.5) .. (-2,3.03);

%\draw[color=red] (-3.293,-.707) .. controls (-1,-2.8) and (1,-2.8) .. (3.293,-.707) node [pos=.8,below=1pt] {$\delta_2$};
%\draw[color=red] (-3.293,.707) .. controls (-1,2.2) and (1,2.2) .. (3.293,.707)node [pos=.8,below] {$\delta$};

\end{scope}

\end{tikzpicture}
\caption{Multicurve $K$}\label{fig_K2}
\end{figure}
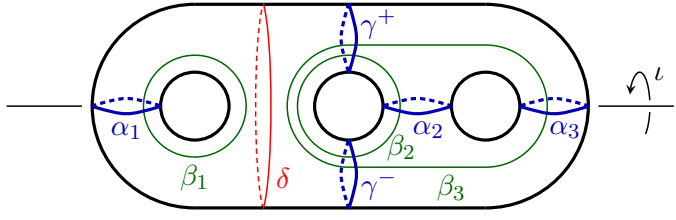

Let $M^+$ and~$M^-$  be the oriented multicurves obtained from~$K$ by removing the components~$\gamma^-$ and~$\gamma^+$, respectively. Put $C_0=[M^+]=[M^-]$. Then $M^+, M^-\in\M_1^{(2)}$ and hence $C_0\in\CH_1^{(2)}$. We may assume that $M^+\in\orb^+_{C_0}$ and $M^-\in\orb^-_{C_0}$. (As we have already mentioned, swapping~$\orb^+_{C_0}$ and~$\orb^-_{C_0}$ is not an innocent operation. However, we can swap~$\gamma^+$ and~$\gamma^-$ instead.) Besides, we choose the orientations of cells so that $[P_{K}\colon P_{M^+}]=1$. Then, according to our agreement, $[P_K:P_{M^-}]=-1$. 

Let $M_1,\ldots, M_m$ be all submulticurves of~$K$  that belong to~$\M_1$ and are different from~$M^{\pm}$; then each of the multicurves~$M_i$ contains the bounding pair~$\{\gamma^+,\gamma^-\}$. We endow every cell~$P_{M_i}$ with the orientation so that $[P_K\colon P_{M_i}]=1$.

Represent~$y$ by the cycle $Y=P_K\otimes[h]$. Then 
$$
\partial'Y=P_{M^+}\otimes[h]-P_{M^-}\otimes[h]+\sum_{i=1}^m P_{M_i}\otimes [h].
$$
We have $\partial'Y\in\Delta$, so we may take $X=0$. 

Let us prove formula~\eqref{eq_key2}.
Obviously,  both sides of~\eqref{eq_key2} vanish whenever $A\not\subset C_0$, so we assume that $A\subset C_0$.

We have $\Psi_C(\partial'Y)=0$ for all $C\in\CH_1^{(1)}$. Further, by Proposition~\ref{propos_Phi_CA_value}, we have $\Phi_{C,A}(\partial'Y)=0$ unless~$C= C_0$, and 
$$
\Phi_{C_0,A}(\partial'Y)=\sum_{i\ne j}\hrho_i(\iota)\rho_j(h),
$$
where~$\iota$ the rotation by~$\pi$ around the horizontal axis, see Fig~\ref{fig_K2}. 

Also, we have 
$$
\nu^+_{C_0,A}(\partial'Y)=\nu_{\varepsilon}(h),
$$
where $\varepsilon$ is the (unique) component of~$M^+$ whose homology class does not belong to~$A$, and $\nu^+_{C,A}(\partial'Y)=0$ unless $C=C_0$.

First, suppose that $c\in A$; then $\varepsilon\ne \gamma^+$. Hence exactly one of the multisets~$[M_s]$  is the set~$A$ with one element duplicated. By Proposition~\ref{propos_psi_values} we obtain that
$$
\sum_{a\in A}\mu_{[A,a]}(d^1y)=\pm\mu_{\gamma^+,\gamma^-}(h)=\left\{
\begin{aligned}
0&&&\text{if }h=T_{\gamma^+,\gamma^-},\\
\pm1&&&\text{if }h=T_{\delta}.
\end{aligned}
\right.
$$
Further, $\omega_i(c)=1$ for the four quadratic functions~$\omega_i$ corresponding to~$A$. Hence, by~\eqref{eq_eBC3} we obtain that $\hrho_i(\iota)=0$ for $i=0,\ldots,3$. Therefore $\Phi_{C_0,A}(\partial'Y)=0$. 
Finally, since $\varepsilon\ne \gamma^+$, Proposition~\ref{propos_phi_values}  implies that
$$
\nu_{\varepsilon}(h)=\left\{
\begin{aligned}
0&&&\text{if }h=T_{\gamma^+,\gamma^-},\\
1&&&\text{if }h=T_{\delta}.
\end{aligned}
\right.
$$
Thus, we obtain~\eqref{eq_key2}.

Second, suppose that $c\notin A$; then $\varepsilon=\gamma^+$. Hence $A=\{a_1,a_2,a_3\}$, where $a_i=[\alpha_i]$, $i=1,2,3$. Let $\ba_1,\ba_2,\ba_3,\bb_1,\bb_2,\bb_3$ be the modulo~$2$ homology classes of the curves $\alpha_1,\alpha_2,\alpha_3,\beta_1, \beta_2,\beta_3$ shown in Fig.~\ref{fig_K2}, respectively. Number the quadratic functions $\omega_0,\ldots,\omega_3$ so that formulae~\eqref{eq_omega_numeration} hold, and number the extended Birman--Craggs homomorphisms $\hrho_0,\ldots,\hrho_3$ accordingly. Then it follows from~\eqref{eq_BC1}, \eqref{eq_BC2}, and~\eqref{eq_eBC3} that 
\begin{gather*}
\rho_0(h)=\rho_1(h)=0,\qquad \rho_2(h)=\rho_3(h)=1,\qquad h\in\left\{T_{\gamma^+,\gamma^-},T_{\delta}\right\},\\
\hrho_0(\iota)=\hrho_1(\iota)=\hrho_2(\iota)=0,\qquad \hrho_3(\iota)=1.
\end{gather*}
Hence
$$
\Phi_{C_0,A}(\partial'Y)=1
$$
in both cases $h=T_{\gamma^+,\gamma^-}$ and $h=T_{\delta}$.
Further, $\nu_{\gamma^+}\bigl(T_{\gamma^+,\gamma^-}\bigr)=-1$ and   $\nu_{\gamma^+}(T_{\delta})=1$. Finally,  $\mu_{[A,a]}(d^1y)=0$ for all $a\in A$, since none of the multisets~$[M_1],\ldots,[M_m]$ coincides with~$[A,a]$. Thus, we again obtain~\eqref{eq_key2}.

\end{document}